\documentclass[11pt, reqno]{amsart}
\usepackage{graphics,verbatim, tikz}
\usetikzlibrary{matrix,arrows,patterns, snakes, fadings}
\usepackage{amsmath,amstext,amsbsy,amssymb,amscd,amsxtra,amsfonts,amsthm,amscd}
\usepackage{adjustbox,enumitem}
\usepackage{graphicx}%
\usepackage{amssymb}
\usepackage{eucal}
\usepackage{color}
\usepackage{multirow}
\usepackage[enableskew,vcentermath]{youngtab}
 \usepackage[colorlinks=true, citecolor=blue]{hyperref}

\newtheorem{thm}{Theorem}[section]
\theoremstyle{plain}
\newtheorem{lem}[thm]{Lemma}
\newtheorem{prop}[thm]{Proposition}
\newtheorem{prp}[thm]{Proposition}
\newtheorem{cor}[thm]{Corollary}

\theoremstyle{definition}
\newtheorem{dfn}[thm]{Definition}
\newtheorem{example}[thm]{Example}

\newtheorem{rmk}[thm]{Remark}

\newtheorem*{thm*}{{\bf Theorem}}
\newtheorem*{thma}{{\bf Theorem A}}
\newtheorem*{thmb}{{\bf Theorem B}}

\theoremstyle{remark}

\definecolor{A}{rgb}{.75,1,.75}

\numberwithin{equation}{section}

\newcommand{\cA}{{\mathcal{A}}}

\newcommand{\bbA}{{\mathbb{A}}}

\newcommand{\bB}{{\mathbf{B}}}
\newcommand{\bi}{{\mathbf{i}}}
\newcommand{\bI}{{\mathbf{I}}}
\newcommand{\bj}{{\mathbf{j}}}
\newcommand{\cB}{{\mathcal{B}}}

\newcommand{\BB}{{\mathbf{B}}}
\newcommand{\cL}{{\mathcal L}}

\newcommand{\catOint}{{\mathcal O_{\rm int}}}

\newcommand{\Q}{{\mathbb{Q}}}
\newcommand{\N}{{\mathbb{N}}}
\newcommand{\Z}{{\mathbb{Z}}}
\newcommand{\Qq}{{\Q(q)}}

\newcommand{\UU}{{\mathbf U}}
\newcommand{\VV}{{\mathbf V}}

\newcommand{\Uint}{ { {}_\bbA\! \UU } }

\newcommand{\Up}{\UU^+}

\newcommand{\Um}{\UU^-}
\newcommand{\Umint}{{}_\bbA\!\Um}
\newcommand{\Uz}{\UU^0}

\newcommand{\osp}{{\mathfrak{osp}}}
\newcommand{\fsl}{{\mathfrak{sl}}}

\newcommand{\te}{{\tilde{e}}}
\newcommand{\tf}{{\tilde{f}}}
\newcommand{\bt}{{\mathbf t}}

\newcommand{\height}{{\mathrm{ht}}}

\newcommand{\End}{{\rm End}}

\newcommand{\barmap}{{\bar{\phantom{x}}}}

\newcommand{\set}[1]{\left\{#1\right\}}
\newcommand{\parens}[1]{\left(#1\right)}
\newcommand{\ang}[1]{\left\langle#1\right\rangle}

\newcommand{\bbinom}[2]{{\begin{bmatrix}#1 \\ #2\end{bmatrix}}}

\renewcommand{\bar}[1]{\overline{#1}}

\newcommand{\gl}{{\mathfrak{gl}}}

\newcommand{\epsp}{{\epsilon}}
\newcommand{\tP}{{\tilde P}}
\newcommand{\Pt}{{P_{\rm typ}}}

\newcommand{\tPp}{{\tP^+}}
\newcommand{\tPt}{{\tP_{\rm typ}}}

\newcommand{\PBW}{{\mathcal{PBW}}}
\newcommand{\Lint}{L_{\bbA}}
\newcommand{\Vint}{V_{\bbA}}
\newcommand{\Kint}{K_{\bbA}}

\begin{document}

\title
[Canonical bases for $\gl(m|1)$]{Canonical bases for the quantum enveloping algebra of $\mathfrak{gl}(m|1)$
	and its modules}
\author{Sean Clark}
\address{Department of Mathematics\\ Northeastern University\\ Boston, MA 02115\\USA}
\curraddr{Max Planck Institute for Mathematics\\ 53119 Bonn\\ Germany}
\email{se.clark@neu.edu\\se.clark@mpim-bonn.mpg.de}
\begin{abstract}
	We construct a crystal basis for the negative half of the quantum enveloping algebra $\UU$ associated to
	the standard super Cartan datum of $\gl(m|1)$,  which is compatible with known crystals 
	on Kac modules and simple modules.
	We show that these crystals admit globalization which
	produce compatible canonical bases. We then define a braid group action
	on a family of quantum enveloping algebras including $\UU$, and use this action
	to show that our canonical basis agrees with those constructed from PBW bases.
	Finally, we explicitly compute some small-rank examples.
\end{abstract}
\maketitle
\tableofcontents
\section{Introduction}
\addtocontents{toc}{\protect\setcounter{tocdepth}{0}}

\subsection{}
A feature of quantum algebras which has taken on increasing importance
is the construction of canonical bases; that is, bases which arise
from the algebra in a natural way and which have
desirable compatibilities with other features of the algebra.
Two standard (and related) examples of such bases are the Kazhdan-Lusztig basis 
of a Iwahori-Hecke algebra \cite{KL} 
and the Lusztig-Kashiwara canonical bases of quantum enveloping
algebra of a Kac-Moody algebra \cite{Ka, LuPaper}.

These canonical bases are remarkable in many ways, but one of the most striking
is their connection to categorification: canonical bases
tend to predict categorifications, and vice versa.
Moreover, such categorifications often reveal the canonical basis elements to
encode a wealth of geometric data (e.g. Lusztig's construction of
the canonical basis of a quantum enveloping algebra via perverse sheaves, cf. \cite[Part 2]{LuBook}) and 
representation-theoretic data (e.g. as the decategorification of an indecomposable projective
module over a KLR algebra, cf. \cite{VV}).

One natural setting to look to construct further examples of canonical bases is the quantum enveloping
algebras of Lie superalgebras.
These quantum enveloping Lie superalgebras have been gaining significant interest in the
years since the foundational work of Yamane \cite{Ya}. For instance,
crystal structures have been constructed and studied on their modules 
\cite{BKK, Kw09, Kw14, Kw15, MZ} and some
examples of categorifications have been constructed \cite{Kh, KS, Sar, Tia2}.
However, despite this attention, it is not known if canonical bases can be defined for many
of these quantum algebras or their modules. 

Recently, there has been some success in constructing canonical bases, at least
on a case-by-case level.
Indeed, crystal and canonical bases for some quantum enveloping
superalgebras associated to a family of Kac-Moody Lie superalgebras, of which the only
finite-type Lie superalgebra is $\mathfrak{osp}(1|2n)$, have been constructed
by the author in joint work with Hill and Wang \cite{CHW1,CHW2};
this construction is expected to play a role in a higher-representation-theoretic
construction of ``odd'' knot homologies \cite{HW,EL, C}.

In another direction, using an approach \`a la \cite{Lec}, it was shown
in \cite{CHW3} that any quantum enveloping Lie superalgebra of basic type
admits a construction of PBW bases using the combinatorics
of Lyndon words and shuffle products. For a few families of Lie superalgebras, including 
the quantum enveloping algebra of the general linear Lie superalgebra $\UU_q(\gl(m|n))$, 
these PBW bases are shown to lead to bar-invariant bases of the half-quantum enveloping algebra 
$\Um_q(\gl(m|n))$ associate to the standard positive root system by a standard argument. 
(We note that \cite{DG} also constructed a bar-invariant basis in the general linear case,
using connections to quantum Schur superalgebras.)

\subsection{} 

It is natural to want to construct canonical bases for $\UU_q(\gl(m|n))$, as it is, in many
ways, the fundamental example of quantum enveloping superalgebras. In particular, one already
encounters many of the structural features complicating such a construction.

Arguably the most fundamental change from the non-super theory 
is the existence of several non-conjugate bases for the root system in general.
This, in turn, leads to non-isomorphic half-quantum enveloping algebras (compare 
Sections \ref{subsec:CB(1|1|1)} and \ref{subsec:CB(1|1|1)}) with sometimes
complicated defining relations; see \cite[Proposition 2.7]{CHW3} for a description of these relations.
While the full quantum enveloping algebras are isomorphic (which is not obvious a priori, but, for
example, follows from the results of Section \ref{sec: braid}),
one must be careful to indicate which positive root system, and thus half-quantum enveloping algebra, 
associated to $\gl(m|n)$ is being considered.
For the most part, we will follow the literature (cf. \cite{BKK,Kw14, KS})
and work with the standard positive system
associated to the Borel subalgebra of upper-triangular matrices.
Henceforth, unless otherwise stated, we write $\UU_q(\gl(m|n))$ (respectively,  $\UU_q^-(\gl(m|n))$) 
to mean the  (negative part of the) quantum enveloping algebra defined
with respect to the standard simple system.

Working over the standard simple system, the results of \cite{CHW3} imply the existence
of a bar-invariant basis of $\Um_q(\gl(m|n))$ associated to the standard PBW basis
(and indeed, there is a such a bar-invariant basis associated to any PBW basis 
by a slight generalization of the arguments therein).
While these bases satisfies one of the characteristic
properties of canonical bases, they fundamentally depend on the choice of PBW basis in general.
In essence, this dependence arises because the underlying super-space structure induces 
a chirality in the quantum parameter $q$, in the sense that the even subspace having
the parameter $q$ means the odd subspace should have the parameter $q^{-1}$.
In particular, the even subalgebra of $\UU_q(\gl(m|n))$ is
\[\UU_q(\gl(m|0)\oplus \gl(0|n))\cong \UU_q(\gl(m))\otimes\UU_{q^{-1}}(\gl(n)).\] 
This chirality is of fundamental
importance when considering the bilinear form and PBW basis
(see Example \ref{ex:CB(2|2)} for a brief discussion of this phenomenon).
This chirality essentially disappears in the case $m=1$ or $n=1$,
in which case the bar-invariant bases are also almost-orthogonal (in the sense of Lusztig)
and thus the bases agree up to possible sign changes.

The last main complication 
is that the category of finite-dimensional representations is not semisimple;
this causes difficulties with constructing crystal bases of representations,
which is one way to approach constructing canonical bases. This can be resolved
by restricting to the subcategory of polynomial representations, which is semisimple (cf. \cite[Section 3.2.6]{CW} for a proof in the classical case).
In fact, the highest weights of irreducible polynomial modules correspond to hook partitions, which means there is a natural
combinatorial model for constructing crystals on these modules;
see \cite{BKK}. One can also
produce crystals on Kac modules \cite{Kw14}; that is, modules
induced from the simple modules of the even subalgebra. However, one should
note that in both of these papers, it is crucial that one works with the standard system
of simple roots. Indeed, it is observed in \cite{Kw09} that the combinatorial
crystals associated to indecomposable modules are, in general, no longer connected when there is 
more than one isotropic simple root, which raises substantial difficulties in constructing
a crystal basis for the module.

\subsection{}

The main goals of this paper is as follows. We want to construct a canonical (unsigned) basis for the 
half-quantum enveloping algebra $\Um$, and show that this canonical basis induces compatible 
canonical bases on the Kac modules and simple polynomial modules,
resolving a conjecture in \cite{CHW3}. Furthermore, we want to show that this canonical basis is in fact precisely
equal to the bar-invariant basis of $\Um$ constructed in {\em loc. cit.} as described above. Our last objective is to get
a sense of to what extent the canonical basis is compatible with the simple modules whose highest weight is not polynomial.

To accomplish these goals, we begin by constructing a crystal basis on $\Um$.
A version of a crystal lattice for $\Um$ has been constructed before by Zou \cite{Zou},
which is essentially defined using only the $\gl(m)$-crystal operators and the odd root
vectors. However, such a lattice is not closed under
all the Kashiwara operators. We will show that under the usual definition of the
crystal lattice (i.e. the lattice generated from $1$ by all the $\tf_i$), the lattice is indeed
closed under all the Kashiwara operators. See Section \ref{subsec:CB(2|1)} 
for a comparison of our lattice and the one in {\em loc. cit.} in the $m=2$ case.

Moreover, we show that this crystal structure
is compatible with the crystals on the Kac modules (defined in \cite{Kw14}) and 
simple polynomial modules (defined in \cite{BKK}). The latter statement follows
naturally from the construction of the lattice, which is a version
of Kashiwara's ``Grand Loop'' restricted to the statements involving $\Um$.
Compatibility with Kac modules then follows from the $\gl(m)$-crystal
structure on $\Um$ and the definition of the crystal on a Kac module.

We then study some of the implications of this crystal structure, deducing analogues of
Kashiwara's characterizations of the crystal lattice and crystal basis,
as well as defining an integral version of the lattice.
It follows from these statements and the construction in \cite{CHW3} 
that the PBW basis, and thus the associated bar-invariant basis, lies in the integral lattice,
and that each PBW element (and thus associated bar-invariant basis element) 
is equivalent modulo $q$ to a crystal basis element, up to a sign.
This immediately implies the existence of a globalization as in \cite[\S 7]{Ka},
from which we deduce the following theorem.

\begin{thma}
The algebra $\Um$ admits a crystal basis $(L(\infty), B(\infty))$.
Furthermore, there exists a globalization $G:L(\infty)/qL(\infty)\rightarrow 
\Umint\cap L(\infty) \cap \bar{L(\infty)}$ such that $\cB=G(B(\infty))$
is a canonical bar-invariant basis of $\Um$.
Moreover, let $V$ be a simple finite-dimensional polynomial module or a finite-dimensional 
Kac module with highest weight vector $v$. Then $\cB(V)=\set{b\in \cB\mid bv\neq 0}$ maps
bijectively to a bar-invariant basis of $V$.
\end{thma}

 As a consequence of the proof of this theorem, we see that $\cB$ is equal (up to signs)
 with the bar-invariant bases constructed in \cite{CHW3}. In fact, these bases coincide.
 The key to proving this in the non-super setting is to utilize Lusztig's braid automorphisms
 of the quantum enveloping algebra \cite{LuPaper,Sai}. 
 These automorphisms provide a systematic way to construct and compare different PBW bases,
 and then it is straightforward to show that 
 the PBW bases coincide modulo $q$ to the crystal basis; see \cite{Tin} for a concise
 exposition of these ideas.
 
 To employ this strategy in our setting, we need braid automorphisms for each simple root.
 However, one quickly runs into an obvious obstruction for isotropic roots: any reasonable
 definition of the automorphism would send a non-nilpotent generator to a nilpotent root
 vector. That such an obstruction occurs is not surprising, 
 as the braid automorphisms are essentially a lift of the Weyl
 group action on the Cartan part, and there is no reflection in the Weyl group corresponding
 to isotropic roots. Instead, we should attempt to lift 
 the Weyl groupoid \cite{Se, HY}: an enlargement
 of the Weyl group which allows for ``odd reflections'' at the cost of no longer
 being a group. This idea has been utilized in \cite{He},
 to produce braid operators for a wide variety of quantum algebras.
 
 In our case, we consider the family of quantum enveloping algebras associated
 to each possible generalized Cartan matrix of $\gl(m|1)$.
 These quantum enveloping algebras have braid automorphisms corresponding to each of their
 even simple roots, whereas the odd simple roots yield isomorphisms between pairs
 of these algebras.
 Despite this extra layer of complexity, these maps still satisfy the type $A$ braid relations,
 and we can use them to define PBW bases on $\Um$ which
 coincide with the bases defined in \cite{CHW3}. Then we can use the standard arguments
 to prove the following theorem.

\begin{thmb}
	Let $<$ be a total order on the simple roots of $\gl(m|1)$, and let $\cB(<)$
	be the associated canonical basis on $\Um$ as defined in \cite{CHW3}. 
	Then $\cB(<)=\cB$.
\end{thmb}

As a consequence, we note that Theorem A and Theorem B proves \cite[Conjecture 8.9]{CHW3}.
This conjecture was already proven in the polynomial module case by Du and Gu \cite{DG}. 
We note that our proof improves on the conjecture: not only does the canonical
basis of $\Um$ induce canonical bases on the polynomial and Kac modules, 
but in fact this canonical basis also agrees with the crystal bases of \cite{BKK,Kw14}.

We finish this paper with a few small examples. These examples are motivated by a desire
to better understand two questions which arise naturally from Theorems A and B. 
First, there are finite-dimensional
simple modules which are not included in Theorem A, and thus we would like to understand their compatibility,
or lack thereof, with $\cB$. 
Second, since Theorems A and B are only proven for $\UU^-_q(\gl(m|1))$ associated to the standard positive root
system, a natural question is whether or not these arguments can be similarly applied
to the half-quantum enveloping algebras associated to other positive root systems. 

To gain insight into the first question, we consider
the canonical basis constructions for $\UU^-_q(\gl(2|1))$ and its finite-dimensional
irreducible modules from \cite[Section 7]{CHW3}
from the viewpoint of crystal bases.
We also explicitly construct the canonical basis for $\UU^-_q(\gl(3|1))$,
and compute the images of this basis in some examples of atypical simple modules.
In all examples we compute, we observe that there is indeed an induced canonical basis, despite
some cases having linear dependencies in the image of $\cB$.

For the second question, we consider the case of the non-standard Cartan datum for $\gl(2|1)$ with two
isotropic roots in some detail. This case is not included in the statements
of Theorems A and B, and we comment on how in particular the results going into Theorem B
fail in this setting. Nevertheless, we observe that the negative part of the quantum enveloping
algebra still has a sensible definition for a canonical basis which satisfies some compatibility
with its representations.

\subsection{}

Given the results of our paper in combination with those of \cite{C, CHW2, CHW3, Kw09},
there are many obvious and interesting lines of study to be pursued.
We hope that our results and explicit examples can help to develop further categorification
results along the lines of \cite{KS}. In particular, it would be interesting
to interpret our canonical basis in terms of the categorification of {\em loc. cit.},
and our canonical basis on modules suggests that one look for categorical representations
as well.

We also expect that our construction can be extended to the entire (idempotented) 
quantum enveloping algebra in the manner of Lusztig (cf. \cite[Part IV]{LuBook},
and see \cite{EK} for a Schur superalgebras prototype of $\dot\UU_q(\gl(m|n))$), 
and plan to pursue
this in subsequent work. In particular, this should admit a diagrammatic categorification
in the sense of Khovanov-Lauda \cite{KL} by building on the construction in \cite{KS}.
We note that a geometrically motivated categorification of the idempotented
algebra for $\gl(1|1)$ was constructed in \cite{Tia2}.

Unfortunately, it is still unclear to us how to uniformly approach the construction of canonical bases
for Lie superalgebras at this time. Nevertheless, an ad-hoc approach
to constructing canonical bases in specific examples certainly seems viable and should
help to clarify the general situation. To that end, the fundamental case to study
is the standard datum associated to $\gl(2|2)$; 
we make some comments on this in Example \ref{ex:CB(2|2)},
but it deserves further study. 
Another further direction would be to try and apply similar
techniques in other quantum enveloping algebras of basic Lie superalgebras; for instance,
the family $\osp(2|2n)$ is one of the families
for which a signed canonical basis was produced in \cite{CHW3}, and one should
be able to remove the signs in a similar fashion. 

Alternatively, one can hope to find
canonical bases associated to non-standard Cartan data for $\gl(m|n)$.
For instance, Section \ref{subsec:CB(1|1|1)} addresses this
for the non-standard datum for $\gl(2|1)$, but we  plan to extend this 
work to some higher rank cases in the future. We also expect that these
algebras should admit categorifications.

\subsection{}

This paper is organized as follows.

In Section \ref{sec:gl(m|1)}, we set our notations and conventions, and recall the
definition of the quantum enveloping algebra $\UU$. We also recall, and 
in some cases elaborate on, the results on crystal bases of 
various $\UU$-modules.

In Section \ref{sec:crystal}, we introduce our definition of the crystal
lattice and basis for $\Um$. We then state and prove a truncated
version of Kashiwara's ``grand loop'' induction argument
to construct the crystal basis for $\Um$ by using the crystals
for modules described in \cite{BKK}. Subsequently,
we prove that the crystal is characterized by the bilinear form
on $\Um$, and satisfies compatibility with Kwon's crystal
bases of Kac modules. Finally, we introduce an integral form
of the lattice on each level ($\Um$, the Kac modules, and the simple
modules), and use a (signed) canonical basis from \cite{CHW3} to
quickly deduce the existence of the globalization.

In Section \ref{sec: braid}, we introduce a family of quantum enveloping algebras associated
to $\gl(m|1)$, which correspond to different choices of Dynkin
diagram. We then define braid isomorphisms between these quantum
groups which lift the Weyl groupoid action on the simple systems
of roots, and use these braid isomorphisms to construct PBW
bases, which agree with the PBW bases defined in \cite{CHW3}.
In the special case of the standard system, it is shown
that these PBW bases span the same $\Z[q]$-lattice,
and thus that the corresponding canonical bases coincide.

Finally, in Section \ref{sec:examples}, we compute some small-rank examples.
First, we consider the case $m=2$ and compute the crystal
lattice in terms of the canonical basis as defined in \cite{CHW3};
we also compare our crystal lattice to that of \cite{Zou}, and
define a compatible crystal structure on the atypical finite-dimensional 
simple modules. Second, we consider the case $m=3$ and compute the canonical
basis. We then consider some examples of atypical modules, and 
observe some additional instances of canonical bases in these examples. 
Finally, we discuss the case of the unique non-standard
simple system of rank $2$, and how it is an obstruction to a more general
application of the results of Sections \ref{sec:crystal} and \ref{sec: braid}. We observe that there is nevertheless
a natural candidate for canonical basis in this case, and analyze
its compatibility with modules.

\vspace{1em}

\noindent{\em Acknowledgements.} We would like to thank the Max Planck Institute
for Mathematics for the pleasant work environment while this project
was carried out. We also thank Weiqiang Wang for his stimulating
questions and comments on this project.

\addtocontents{toc}{\protect\setcounter{tocdepth}{2}}
\section{Standard quantum $\gl(m|1)$}\label{sec:gl(m|1)}

We begin by introducing our conventions, notations, and definitions for quantum
$\gl(m|1)$, and recalling some of the essential results in the literature.

\subsection{The root data}\label{subsec:root data}
We write $\Z/2\Z=\set{0,1}$. Throughout, with a $\Z/2\Z$ graded set $X=X_0\cup X_1$, 
we write $p(x)\in\Z/2\Z$ where $x\in X_{p(x)}$.
Define the $\Z/2\Z$ graded sets
\begin{equation}
[m|1]=\underbrace{\set{\epsp_1,\ldots, \epsp_m}}_{[m|1]_0}\cup \underbrace{\set{\epsp_{m+1}}}_{[m|1]_1},\quad
I=\underbrace{\set{1,\ldots, m-1}}_{I_0}\cup \underbrace{\set{m}}_{I_1}
\end{equation}
We define the weight lattice $P=\Z[m|1]=\bigoplus_{k=0}^{m+1} \Z \epsp_k$
and endow $P$ with the symmetric bilinear form defined by 
$(\epsp_i,\epsp_j)=(-1)^{p(\epsp_i)}\delta_{ij}$ for $1\leq i,j\leq m+1$,
where here (and throughout) $\delta_{ab}$ is the Kronecker delta with $\delta_{ab}=1$ if $a=b$
and $0$ otherwise.
We define the coweight lattice $P^\vee=\bigoplus_{k=1}^{m+1} \Z \epsp_k^\vee$ 
and we have the pairing
$\ang{\cdot,\cdot}:P^\vee\times P\rightarrow \Z$ with $\ang{\epsp_i^\vee, \epsp_j}=\delta_{ij}$.
Note that $P$ and $P^\vee$ are naturally $\Z/2\Z$-graded with 
\[p\parens{\sum_{i=1}^{m+1} a_i \epsp_i}=p\parens{\sum_{i=1}^{m+1} a_i \epsp_i^\vee}=\sum_{i\in [m|1]} a_ip(\epsp_i).\]
Then we have the root system $\Phi=\set{\epsp_i-\epsp_j\mid 1\leq i\neq j\leq m+1}$. 
This root system has a parity induced by the parity on $P$, with
\[\Phi_0=\set{\epsp_i-\epsp_j\mid 1\leq i\neq j\leq m},\qquad \Phi_1=\set{\epsp_i-\epsp_{m+1}, \epsp_{{m+1}}-\epsp_i \mid 1\leq i\leq m}.\]
Let $\Pi=\set{\alpha_i=\epsp_{ i}-\epsp_{{i+1}}\mid i\in I}$, and note $p(\alpha_i)=p(i)$. 
Then $\Pi$ is a system
of simple roots for $\Phi$ (called the {\em standard} simple roots), 
and we let $\Phi^+$ be the associated set of positive roots.
We set $h_i=\epsp_{i}^\vee-(-1)^{p(\epsp_{i+1})}\epsp_{i+1}^\vee$ for $i\in I$ 
to be the simple coroots in $P^\vee$,
which satisfy $\ang{h_i,\alpha_j}=(\alpha_i,\alpha_j)$.
More generally, for $\alpha\in \Phi^+$ with $\alpha=\sum a_i \alpha_i$,
we let $h_\alpha=\sum a_i h_i$.

The root lattice is $Q=\sum_{i\in I} \Z\alpha_i \subset P$.
We set $Q^+=\sum_{i\in I} \Z_{\geq 0}\alpha_i$ and $Q^-=-Q^+$.
As usual, we have a partial order on $P$ induced by $Q^+$: we say
$\lambda\leq \mu$ if $\mu-\lambda\in Q^+$.
A useful statistic for $Q^+$ is the {\em height} of a root defined by
$\height(\sum a_i \alpha_i)=\sum a_i$, and for $l\in\Z_{\geq 0}$, 
we define $Q^+(l)=\set{\alpha\in Q^+\mid \height(\alpha)\leq l}$.

There are two distinguished vectors in $P$ of use in our discussions.
First, we define
\begin{equation}\label{eq:parity sum weight}
1_{m|1}=\sum_{k=1}^{m+1} (-1)^{p(\epsp_k)} \epsp_k\in P,
\end{equation}
and note that $\ang{h_i, 1_{m|1}}=0$ for all $i\in I$.
Next, the {\em shifted Weyl vector} $\rho\in P$ associated to this root data is given by
\begin{equation}\label{eq:weyl vector}
\rho=\sum_{k=1}^m (m-k+1) \epsp_k - \epsp_{m+1}
\end{equation}
and satisfies $(\alpha_i,\rho)=\frac{1}{2}(\alpha_i,\alpha_i)$ 
for all $i\in I$.

A weight $\lambda\in P$ is {\em $\gl(m)$-dominant}, or simply {\em dominant}, 
if $\ang{h_i,\lambda}\geq 0$ for
all $i\in I_0$. Explicitly,
the set $P^+$ of dominant weights is given by 
\[P^+=\set{\sum_{k=1}^{m+1} a_k \epsp_k\mid a_k\geq a_{k+1} \text{ for all }1\leq k<m}.\]

We further say that $\lambda\in P^+$ is {\em $\gl(m|1)$-dominant}, or {\em fully dominant}, 
if $\ang{h_m,\lambda}\geq 0$,
and we denote the set of fully dominant weights by $P^{++}$. Explicitly,
\[P^{++}=\set{\sum_{k=1}^{m+1} a_k \epsp_k\mid a_k\geq a_{k+1} \text{ for all }1\leq k<m\text{ and }a_m+a_{m+1}\geq 0}.\]
As observed in 
A weight $\lambda=\sum_{k=1}^{m+1} a_k\epsp_k$ is {\em polynomial}
if $a_k\geq 0$ for all $k$. We denote the set of polynomial
weights by $\tP$.
Finally, we say a weight $\lambda\in P$ is {\em typical}
if $\ang{h_\alpha,\lambda+\rho}\neq 0$ for all $\alpha\in \Phi_1^+$
and we denote the set of typical weights by
$\Pt$. We also will combine these notations in obvious ways; e.g. $\tPt^{++}$ would
denote the set of typical fully dominant polynomial weights.

In particular, we note that $P^{++}=\tP^{++}+\Z 1_{m|1}$
and $\tP^{++}=\tP^+$.
We also note that a fully dominant weight $\lambda=\sum a_k \epsp_k\in P^{++}$ 
is typical if and only if
$\ang{h_m,\lambda}> 0$; indeed, we compute that
\[\ang{h_{\epsp_i-\epsp_{m+1}},\sum a_k \epsp_k+\rho}=a_i+a_{m+1}+(m-i)\geq a_m+a_{m+1}=\ang{h_m,\lambda}\geq 0.\]
\subsection{The quantum enveloping algebra}\label{subsec:quagrp}

We work over the base field $\Q(q)$, where $q$ is an indeterminant parameter. 
We will also occasionally work over the following subrings of $\Qq$:
\begin{itemize}
\item $\bbA=\Z[q,q^{-1}]$;
\item $\cA$, the subring of rational functions in $q$ with no pole at $0$;
\item $\cA_\Z$ is the $\Z$-subalgebra of $\Qq$ generated by $q$ and $\frac{1}{1-q^{2t}}$
for $t\geq 1$.
\end{itemize} 
Recall some
standard notation for $q$-integers: for $a\in \Z$ and $b\in \Z_{\geq 0}$
\begin{equation}
[a]=\frac{q^a-q^{-a}}{q-q^{-1}};\quad [b]!=\prod_{c=1}^{b} [a];\quad \bbinom{a}{b}=\frac{\prod_{c=0}^{b-1} [a-c]}{[b]!} \in \bbA.
\end{equation}
As usual, a $\Qq$-algebra $A$ has a natural notion of divided powers:
if $a\in A$ and $n\in\Z_{\geq 0}$, we set 
\[a^{(n)}=\frac{a^n}{[n]!}.\]

Now let us recall the definition of $\UU_q(\gl(m|1))$.

\begin{dfn}
The algebra $\UU=\UU_q(\gl(m|1))$ is the $\Qq$-algebra on generators
$E_i$, $F_i$, $q^h$ for $i\in I$, $h\in P^\vee$ with parity grading
given by $p(E_i)=p(F_i)=p(i)$, $p(q^h)=0$;  and relations given by
\[
q^{0}=1,\quad q^{h}q^{h'}=q^{h+h'},\quad 
q^hE_jq^{-h}=q^{\ang{h,\alpha_j}}E_j,\quad q^hF_jq^{-h}=q^{-\ang{h,\alpha_j}}F_j\quad
\text{ for } h,h'\in P^\vee,\ j\in I;\]
\[E_iF_j-(-1)^{p(i)p(j)}F_jE_i=\delta_{ij}\frac{K_i-K_i^{-1}}{q-q^{-1}}\quad \text{ for } i,j\in I;\]
\[E_i^2E_j-[2]E_iE_jE_i+E_jE_i^2=0\quad \text{ for } i\in I_0,\ j\in I \text{ such that } |i-j|=1;\]
\[F_i^2F_j-[2]F_iF_jF_i+F_jF_i^2=0\quad \text{ for } i\in I_0,\ j\in I \text{ such that } |i-j|=1;\]
\[E_m^2=F_m^2=0.\]
Here, $K_i=q^{h_i}$.
\end{dfn}

The algebra $\UU$ has the standard structural properties of quantum enveloping algebras,
which we will now recall. First, we recall the Hopf algebra structure on $\UU$.

\begin{lem}
The algebra $\UU$ is a Hopf superalgebra
with coproduct $\Delta$, antipode $S$, and counit $\varepsilon$ defined as follows:
\[\Delta(E_i)=E_i\otimes K_i^{-1}+ 1\otimes E_i,\qquad 
\Delta(F_i)=F_i\otimes 1+K_i\otimes F_i,\qquad \Delta(q^h)=q^h\otimes q^h;\]
\[S(E_i)=-E_iK_i,\qquad S(F_i)=-K_i^{-1} F_i,\qquad S(q^h)=q^{-h};\]
\[\varepsilon(E_i)=\varepsilon(F_i)=0,\qquad\varepsilon(q^h)=1.\]
\end{lem}

Next, recall that $\UU$ has the weight decomposition 
\begin{equation}
\UU=\bigoplus_{\zeta\in Q} \UU_\zeta,\qquad  \UU_{\zeta}=\set{u\in\UU\mid q^h u q^{-h}= q^{\ang{h,\zeta}} u};
\end{equation}
we write $|u|=\zeta$ if $u\in\UU_\zeta$.
It also has the triangular decomposition 
\begin{equation}
\UU\cong \Um\otimes \Uz\otimes \Up\cong \Up\otimes\Uz\otimes\Um.
\end{equation}
Here, $\Up$ is the subalgebra generated by the $E_i$, $\Um$
is the subalgebra generated by the $F_i$, and $\Uz$ is the subalgebra
generated by $q^h$ for $h\in P^\vee$.
We also define the {\em integral form} $\Uint$ of $\UU$ to be the $\Z[q,q^{-1}]$-subalgebra
of $\UU$ generated by $E_i^{(a)}$, $F_i^{(a)}$, $q^h$ for $i\in I$, $a\in \Z_{\geq 0}$, 
and $h\in P^\vee$.
We further set $\UU^{\pm}_{\zeta}=\UU^{\pm}\cap \UU_{\zeta}$ for $\zeta\in Q$
and $\Uint^{\pm 1}=\UU^{\pm}\cap \Uint$.

The algebra $\UU$ has several important involutions:
\begin{enumerate}
\item The {\em bar involution} $\barmap$ defined by
\[\bar{E_i}=E_i,\qquad \bar{F_i}=F_i,\qquad \bar{q^h}=q^{-h},
\qquad \bar{q}=q^{-1},\qquad \bar{xy}=\bar{x}\bar{y}.\]
\item The anti-involution $\tau$ defined by
\[\tau(E_i)=E_i,\qquad \tau(F_i)=F_i,\qquad \tau(q^h)=(-1)^{p(h)}q^{-h},\qquad
\tau(xy)=\tau(y)\tau(x).\]
Note that $\tau$ restricts to an anti-involution of $\UU^{\pm}$.
\item The anti-involution $\eta$ defined by
\[\eta(E_i)=qF_iK_i^{-1},\qquad \eta(F_i)=qE_iK_i,\qquad \eta(q^h)=q^{h},\qquad
\eta(xy)=\eta(y)\eta(x).\]
\end{enumerate}

Moreover, we recall the following results from \cite{CHW3}.
Let $i\in I$. 
We define the left-multiplication maps $f_i:\Um\rightarrow\Um$ by $f_i(x)=F_ix$.
We also define the quantum differential $e_i:\Um\rightarrow\Um$
by $e_i(F_j)=\delta_{ij}$
and, for homogenous $u,v\in \UU^-$,
\begin{equation}\label{eq:quandiff}
e_i(uv)=e_i(u)v+(-1)^{p(u)p(i)}q^{(\alpha_i,|u|)}ue_i(v).
\end{equation}
By twisting with automorphisms, we obtain the following variants
\[\bar{e_i}=\barmap\circ e_i\circ \barmap,\qquad e_i^\tau=\tau\circ e_i\circ \tau,\]
which satisfy $\bar{e_i}(F_j)=e_i^\tau(F_j)=\delta_{ij}$ and
\[\bar{e_i}(uv)=\bar{e_i}(u)v+(-1)^{p(u)p(i)}q^{-(\alpha_i,|u|)}u\bar{e_i}(v),\]
\[e_i^\tau(uv)=(-1)^{p(u)p(i)}q^{(\alpha_i,|u|)}e_i^\tau(u)v+ue_i^\tau(v).\]
Observe that $e_ie_j^\tau=e_j^\tau e_i$.

\begin{prp}\label{P:BilinearForm2}
The algebra $\Um$ is equipped with symmetric nondegenerate bilinear forms $(-,-),\set{-,-}:\Um\times \Um\rightarrow\Qq$ satisfying $(1,1)=\set{1,1}=1$, $(f_i(x),y)=(x,e_i(y))$,
$\set{f_i(x),y}=\set{x,\bar e_i(y)}$, and $\set{x,y}=\bar{(\bar x, \bar y)}$. In particular, 
the subalgebra of $\End_{\Qq}(\Um)$ generated by the $e_i$ (resp. $\bar{e_i}$, $f_i$) is isomorphic to $\Um$.
\end{prp}

We observe the following useful property of the bilinear form.

\begin{lem}
For any $x,y\in \Um$, $(\tau(x),\tau(y))=(x,y)$.
\end{lem}

\begin{proof}
Let $l\geq 0$ and $\xi\in Q^+(l)$. We will prove that for all $x\in \Um_{-\xi}$,
$(xF_i,y)=(x,e_i^\tau(y))$ and $(\tau(x),\tau(y))=(x,y)$ for all  $i\in I$ and 
$y\in \Um$ by induction on $l$.
If $l=0$ this is obvious. Otherwise, we can write $x=F_jz$ for some $j\in I$ and $z\in \Um$.
Then $$(F_jxF_i,y)=(zF_i,e_j(y))=(z,e_i^\tau e_j(y))=(z,e_je_i^\tau(y))=(x,e_i^\tau(y)),$$
and 
$$(\tau(x),\tau(y))=(F_j\tau(z),\tau(y))=(\tau(z),e_j\tau(y))=(\tau(z),\tau(e_j^\tau(y)))=
(z,e_j^\tau(y))=(x,y).$$

\end{proof}

We now recall one of the crucial results from \cite{CHW3} on distinguished bases for
$\Um$. To wit, using quantum shuffles, one can inductively construct PBW bases associated
to any total order on the simple roots. These PBW bases then allow us to construct a signed
canonical basis.

\begin{lem}[{\cite{CHW3}}]\label{lem:CHW3PBW}
Let $\prec$ be a total order on $I$. 
Then there is an induced convex order $\prec$ on $\Phi^+$, 
so write $\Phi^+=\set{\beta_1\prec\ldots\prec \beta_N}$.
Let $N=|\Phi^+|$, and let 
$$\Z_{\geq 0}^{\Phi^+}=\set{a=(a_1,\ldots, a_N)\in \Z_{\geq 0}^N\mid a_r\leq 1
\text{ if } p(\beta_r)=1}.$$
Then $\Umint$ has an $\bbA$-basis
$$\BB(\prec)=\set{F_{(a_1,\ldots, a_r)}=\prod_{r=1}^N F_{\prec,\beta_r}^{(a_r)}\mid 
(a_1,\ldots, a_N)\in\Z_{\geq 0}^{\Phi^+}}.$$
Here, $F_{\prec,\beta_r}$ is a root vector of weight $-\beta_r$ depending on the ordering $\prec$,
and  $F_{\prec,\beta_r}^{(a)}=F_{\prec,\beta_r}^a/[a]!$ as usual.
Moreover, for any $a=(a_1,\ldots, a_N)$ and $a'=(a_1',\ldots, a_N')$ with 
$a,a'\in \Z_{\geq 0}^{\Phi^+}$,
we have
\[(F_a,F_{a'})\in \delta_{a,a'}+q\cA_\Z.\]
(That is, $\BB(\prec)$ is almost-orthonormal under $(-,-)$.)
Finally, there exists a homogeneous $\bbA$-basis $\cB(\prec)=\set{b_a\mid a\in \Z_{\geq 0}^{\Phi^+}}$ 
of $\Umint$ satisfying
\begin{align*}
\text{\em (1) }\tag{bar-invariance}&\bar b_a=b_a, \\
\text{\em (2) }\tag{$q$-unitriangularity} &b_a-F_a\text{ is in the }q\Z[q]\text{-span of }\PBW,\\
\text{\em (3) }\tag{almost-orthogonality} &\text{and }(b_a,b_{a'})\in \delta_{a,a'}+q\cA_\Z\text{ for all }
a,a'\in \Z_{\geq 0}^{\Phi^+}.
\end{align*}  
\end{lem}

\begin{proof}
These statements follow from \cite{CHW3} in the standard ordering case.
In a non-standard ordering $\prec$, it is only shown that the PBW basis exists for $\Um$ 
and is orthonormal. However, direct computation using the formulas in {\em loc. cit.} 
shows that the PBW basis still has the claimed properties; alternatively, this will
follow from the construction in Section \ref{sec: braid}. The existence
of the canonical basis then follows from similar arguments to {\em loc. cit.} Section 7.

Note that to prove $(F_a,F_{a'})\in \delta_{a,a'}+q\cA_\Z$, Theorem 5.7 in 
{\em loc. cit.} implies it suffices to show this in the case of a divided power of a single
root vector. For root vectors of even parity, this follows from the same calculation as
in $\gl(m)$; for the odd root vectors, it suffices to show that the norm of the root
vector lies in $1+q\Z[q]$ but this is easy to verify directly using the fact
that root vectors for $\beta\in \Phi^+\setminus \Pi$ are always of the form 
$F_\beta=F_iF_{\beta-\alpha_i}-qF_{\beta-\alpha_i}F_i$ for some $i\in I$.
\end{proof}

We call $\cB(\prec)$ the $\prec$-canonical basis of $\Um$.

\begin{rmk}
It is well-known that the conditions on $\cB(\prec)$ imply it is unique
up to a sign (cf. \cite[Theorem 14.2.3]{LuBook}), so the signed basis $\cB(\prec)\cup -\cB(\prec)$
does not depend on $\prec$. One of our goals will be to show that $\cB(\prec)$ itself
is independent of $\prec$.
\end{rmk}

\begin{example}\label{ex:CB(2|2)}
A version of Lemma \ref{lem:CHW3PBW} is proven in \cite[Section 7]{CHW3} for the standard data
associated to some families of basic Lie superalgebras.
In particular, $\UU_q(\gl(m|n))$ is observed to have a 
pseudo-canonical basis; that
is, a basis $\cB$ which is bar-invariant and $q$-unitriangular with respect to
the standard PBW basis, but which is not almost-orthogonal under the bilinear form.
(Another construction of this basis is given in \cite{DG}, motivated
by connections to quantum Schur superalgebras.)
However, it is easy to see this basis crucially depends
on the choice of PBW basis. Indeed, in the case of $m=n=2$, the standard Cartan datum
has $I=\set{1<2<3}$ with $I_1=\set{2}$ and the GCM $A=\begin{bmatrix}2&-1&0\\-1&0&1\\0&1&-2\end{bmatrix}$. The PBW vectors are then
$F_{12}=F_2F_1-qF_1F_2$, $F_{23}=F_3F_2-q^{-1}F_2F_3$, and $F_{123}=F_3F_2F_1-qF_1F_3F_2-q^{-1}F_{2}F_1F_3+F_1F_2F_3$. On the other hand, 
applying $\tau$ gives us the PBW vectors for the opposite 
order; that is, the order $I=\set{3<^{\rm op}2<^{\rm op}1}$.
Then, for instance, it is easy to see that
we have 
\[F_2F_3,\ F_3F_2-[2]F_2F_3\in \cB(<),\text{ but }F_3F_2,\ F_2F_3-[2]F_3F_2\in \cB(<^{\rm op}).\]
It is unclear how to resolve this sort of incompatibility, but
morally one would want the canonical basis of $\UU_q(\gl(m|n))$ to be compatible
with its Levi subalgebras. In particular, this means that the root spaces
corresponding to the subalgebras $\UU_q(\gl(2|1))$
and $\UU_{q}(\gl(1|2))\cong \UU_{q^{-1}}(\gl(2|1))$ 
should have the canonical bases given in \cite[\S 8]{CHW3}.
For this to be possible, we see that both almost-orthogonality and $q$-unitriangularity
fail, as necessarily one subalgebra will satisfy the conditions in $q$ while the other
satisfies the conditions in $q^{-1}$. We do not yet see how to glue these conditions in such
a way to have a well-defined canonical basis on the remaining root spaces, hence
we restrict our attention to the $n=1$ case.
\end{example}

\subsection{Weight modules}

Throughout, a $\UU$-module is a $\Z/2\Z$-graded weight module;
that is, a $\Z/2\Z$-graded $\Qq$-vector space $M=M^0\oplus M^1$ with a decomposition
$M^t=\bigoplus_{\lambda\in P} M^t_{\lambda}$ for $t=0,1$ such that  $M_\lambda^t$ 
is finite-dimensional and
$q^h|_{M^t_{\lambda}}=q^{\ang{h,\lambda}}$. 
We write $M_\lambda=M^0_\lambda\oplus M^1_\lambda$,
and for $m\in M_\lambda^t$ we write $p(m)=t$ and $|m|=\lambda$.
\begin{rmk}
A $\UU$-module homomorphism
is always assumed to be homogeneous with respect to the $\Z/2\Z$-grading.
In particular, we generally consider $M=M^0\oplus M^1$ and its {\em parity-shift}
$\widetilde M=\widetilde M^0\oplus \widetilde M^1$ with $\widetilde M^t=M^{1-t}$
(with the same $\UU$-action)
to be non-isomorphic. However, in general we will ignore this facet of the representation
theory, since we will essentially exclusively work with a family of modules with a canonical
choice of parity, and the parity is non-essential except in tensor products.
\end{rmk}

Given modules $M,N$, we define the module $M\otimes N=M\otimes_{\Qq} N$
with the action of $\UU$ using the Hopf superalgebra structure; in other words,
for homogeneous $x\in M$ and $y\in N$, we define 
$u\cdot x\otimes y=\sum (-1)^{p(u_2)p(x)}u_1x\otimes u_2 y$,
where $\Delta(u)=\sum u_1\otimes u_2$.

Let us now recall some properties $\UU$-modules.
We say that $M$ is {\em polynomial}
if $M_\lambda\neq 0$ implies $\lambda$ is polynomial.
As usual, we say that $M$ is a {\em highest weight} module there is a $\lambda\in P$
such that $\dim_{\Qq} M_\lambda=1$, $M_\mu\neq 0$ only if $\mu\leq \lambda$,
and $M=\UU M_\lambda$.
Finally, we say $M$ has a {\em polarization}
if there is a bilinear form $(\cdot,\cdot):M\times M\rightarrow \Qq$
(which we call the polarization on $M$) satisfying $(uv,w)=(v,\eta(u)w)$.
\begin{lem}[\cite{BKK}]
Let $M$ and $N$ be $\UU$-modules with polarizations. 
Then $M\otimes N$ has a polarization defined by
\[(v\otimes w, v'\otimes w')_{M\otimes N}=(v,v')_M(w,w')_N.\]
\end{lem}

Now let us recall a standard construction of (irreducible) finite-dimensional $\UU$-modules,
analogous to the usual construction by Verma modules.
Let $\VV$ be
the subalgebra of $\UU$ generated by $q^h$ for $h\in P^\vee$ and $E_i$, $F_i$,
for $i<m$; in other words $\VV\cong\UU_q(\gl(m)\oplus \gl(1))$.
We also set $\VV^{\pm}=\VV\cap \UU^{\pm}$.
For $\lambda\in P$, let $\lambda|_\VV=\sum_{k\in I_0} \lambda_k \epsilon_k$.
Let $V_\VV(\lambda)$ be the irreducible $\VV$-module of highest
weight $\lambda|_\VV$. Let $\VV'$ be the subalgebra
of $\UU$ generated by $\VV$ and $E_m$.
Then we define $K(\lambda)=\UU\otimes_{\VV'} V_\VV(\lambda)$,
where the action of $\VV'$ is given by setting $E_m$
to act trivially, and the parity on $K(\lambda)$ is induced by $\UU$ by defining
$p(v)=p(\lambda)$ for $v\in V_\VV(\lambda)$. 
(We choose this convention for parity of the highest weight space to be compatible
with the combinatorial realization in terms of Young tableaux in \cite{BKK}.)
We let $V(\lambda)$
be the maximal irreducible quotient of $K(\lambda)$.

Henceforth, let $k_\lambda$ denote a choice of highest weight vector of $K(\lambda)$
and let $v_\lambda$ be the image of $k_\lambda$ in $V(\lambda)$
for all $\lambda\in P$. We define $\pi_\lambda^K:\Um\rightarrow K(\lambda)$ to be
the $\Um$-module homomorphism with $\pi_\lambda^K(x)=xk_\lambda$,
$\pi_\lambda^{K,V}: K(\lambda)\rightarrow V(\lambda)$ to
and write $\pi_\lambda=\pi_\lambda^{K,V}\circ \pi_\lambda^{K}:\Um\rightarrow V(\lambda)$. 

These modules satisfy the following properties.

\begin{lem}\label{lem:Kac facts} We observe the following properties of Kac modules
and irreducible modules.
\begin{enumerate}
\item Let $\lambda\in P$. Then the following are equivalent: 
\begin{enumerate} 
\item $\dim_{\Qq} K(\lambda)<\infty$;
\item $\dim_\Qq V_\VV(\lambda)<\infty$;
\item $\dim_\Qq V(\lambda)<\infty$.
\end{enumerate}
\item Let $\lambda\in P^+$.
As $\UU^-$-modules, $K(\lambda)\cong\UU^-/I_\lambda$
where $I_\lambda$ is the ideal generated by $F_i^{\ang{h_i,\lambda}+1}$
for $i\in I_0$. 
\item If $\lambda\in P^+$ is typical, $K(\lambda)=V(\lambda)$.
\item Let $\lambda\in P^+$, and let $M$ be a weight module and 
$v\neq 0\in M_\lambda$ such that $E_iv=0$ for all $i\in I$
and $F_i^{\ang{h_i,\lambda}+1}v=0$ for $i\in I_0$. Then there is a unique $\UU$-module homomorphism
$K(\lambda)\rightarrow M$ (or from the parity shift $\widetilde{K(\lambda)}\rightarrow M$) such that $k_\lambda\mapsto v$.
\item If $M$ is a finite-dimensional irreducible module, then $M\cong V(\lambda)$
(or its parity shift $\widetilde{V(\lambda)}$)
for some $\lambda\in P^+$.
\item If $\lambda\in P^{+}$,
then $V(\lambda)$ carries a unique polarization with $(v_\lambda,v_\lambda)=1$.
\end{enumerate}
\end{lem}

\begin{proof}
We observe that (1), (2) and (3) can essentially be found in \cite{Kw14}, 
but we will recall the broad strokes here.
First, note that (1) follows directly from the PBW theorem and triangular decomposition
in a way entirely analogous to the classical case (cf. \cite[Proposition 2.1]{CW}).
Indeed, we note that the implications (a) $\Rightarrow$ (c) $\Rightarrow$ (b) are clear, 
and the remaining implication follows from observing that, as vector spaces, 
$K(\lambda)\cong \bigwedge(\Phi_1^-)\otimes V_\VV(\lambda)$ (where $\bigwedge(\Phi_1^-)$
is the exterior algebra of the formal vector space on the set $\Phi_1^-$).
For (2), note the action of $\UU^-$ on $K(\lambda)$ induces a 
surjection $\Um/I_\lambda\rightarrow K(\lambda)$ sending the image of $1$
to $k_\lambda$. On the other hand, it is easy to see from the PBW basis
with respect to the opposite standard ordering (i.e. where $m$ is minimal and
the odd root vectors are left factors) that these vector spaces
have the same dimension, so the surjection must be an isomorphism.
We can take a classical limit of $\UU$ and $K(\lambda)$ to deduce (3) from
the classical case.

Observe that (4) now follows immediately from (2), by observing that the $\Um$-module
projection $\Um\rightarrow M$ given by $x\mapsto xv$ factors through $K(\lambda)$.
It is easy to see that this map preserves the weights and the action of the $E_i$
for $i\in I$.

Next, for (5) note that since $M$ is finite-dimensional, it has a highest
weight space (that is, a weight space $M_\lambda\neq 0$ such that $\Up M_\lambda=0$).
Then by (4), given a nonzero vector $v\in M_\lambda$, we have a non-zero 
$\UU$-module homomorphism $K(\lambda)\rightarrow M$. Since $M$ is irreducible,
this map is surjective, hence it must be that $M\cong V(\lambda)$.

Finally, (6) is proven in \cite{BKK} when $\lambda\in P^{++}$ 
by restricting the polarization on tensor
powers of the standard representation (possibly with an additional $1$-dimensonal
tensor factor twisting the weight by $a1_{m|1}$), 
but we note this is easily proven directly for all $\lambda\in P^+$ using
standard arguments. (In particular, that $V(\lambda)\cong V(\lambda)^*$ where
the action of $\UU$ on the dual is defined using $\eta$.)
\end{proof}

Let $\xi\in Q^+$. We will say $\lambda\gg\xi$ if $\ang{h_i,\lambda-\xi}>0$
for all $i\neq m$. Note that if $\lambda\gg\xi$, then 
$\UU^-_{-\xi}\cong K(\lambda)_{\lambda-\xi}$ as vector spaces.
In particular, when $\lambda\in P_{\rm typ}^+$ and $\lambda\gg \xi$, 
$\Um_{-\xi}\cong V(\lambda)_{\lambda-\xi}$ 

\subsection{Crystal bases of modules}
We shall now recall the definition and some facts about crystal bases
as defined in \cite{BKK, Kw14}. 
For $i\in I$, let $\UU(i)$ be the subalgebra
of $\UU$ generated by $E_i$, $F_i$, and $K_i$.

\begin{dfn}[{\cite[Definition 2.2]{BKK}}]
The category $\catOint$ is the full subcategory of $\UU$-weight modules
such that if $M\in\catOint$:
\begin{enumerate}
\item $M$ is locally $\UU(i)$-finite; that is, $\dim(\UU(i)v)<\infty$
for any $v\in M$.
\item If $M_\mu\neq 0$ for some $\mu\in P$, then $\ang{h_m,\mu}\geq 0$.
\item If $v\in M_\mu$ such that $\ang{h_m,\mu}= 0$, then $E_mv=F_mv=0$.
\end{enumerate}
\end{dfn}
We note that $V(\lambda)\in\catOint$ if and only if $\lambda\in P^{++}$,
but that $K(\lambda)\notin\catOint$ for $\lambda\notin P^{++}_{\rm typ}$.

If $M\in \catOint$, then we may define Kashiwara operators on $M$ as follows.
Let $v\in M_\lambda$ for some $\lambda\in P$. 
If $i\neq m$, then there exists a unique family of $v_n\in M_{\lambda+n\alpha_i}$
with $\ang{h_i,\lambda+n\alpha_i}\geq n\geq 0$ such that
$v=\sum_{n\in\N} F_i^{(n)}v_n$,
and we set
\begin{equation}
\te_i v=\sum_{n\in\N} F_i^{(n-1)}v_n,\qquad \tf_i v=\sum_{n\in\N} F_i^{(n+1)}v_n.
\end{equation}
(Here, we denote $F_i^{(-1)}=0$ for convenience.) On the other hand, if $i=m$,
then we set
\begin{equation}
\te_m v=q^{-1}K_mE_mv,\qquad \tf_mv=F_mv.
\end{equation}
Now recall that $\cA$ is the subalgebra of $\Qq$ of rational functions with
no poles at $0$.
\begin{dfn}[{\cite[Definitions 2.3, 2.4, 2.10]{BKK}}]
Let $M\in\catOint$. We say that an $\cA$-submodule $L$ of $M$ is a {\em crystal
lattice} of $M$ if:
\begin{enumerate}
\item $M=\Qq\otimes_\cA L$;
\item $L=\bigoplus_{\lambda} (L^0_\lambda\oplus L^1_\lambda)$ 
where $L^t_\lambda=L\cap M^t_\lambda$;
\item $\te_i L\subset L$ and $\tf_i L\subset L$ for all $i\in I$.
\end{enumerate}
Suppose $L$ is a crystal lattice of $M$, and let $B\subset L/qL$.
We say that $(L,B)$ is a {\em signed crystal basis} of $M$ if:
\begin{enumerate}[resume]
\item $B$ is a signed $\Q$-basis of $L/qL$ (that is, $B=B'\cup(-B')$ for some $\Q$-basis of
$L/qL$);
\item $B=\bigcup_{\lambda} (B^0_\lambda\cup B^1_\lambda)$ where 
$B^t_\lambda=B\cap (L^t_\lambda/qL^t_\lambda)$;
\item $\te_i B\subset B\cup\set{0}$ and $\tf_i B\subset B\cup \set{0}$ for all $i\in I$.
\item If $b,b'\in B$ and $i\in I$, then $\te_i b=b'$ if and only if $\tf_i b'=b$.
\end{enumerate}
If $B$ satisfies (5)-(7) and is a $\Q$-basis of $L/qL$, 
then we say $(L,B)$ is an unsigned crystal basis of $M$,
or simply a crystal basis of $M$.
Finally, we say $(L,B)$ is {\em polarizable} if $M$ has a polarization $(\cdot,\cdot)$
such that $(L,L)\subset \cA$ and the induced bilinear form 
$(\cdot,\cdot)_0$ on $L/qL$ satisfies $(b,b)=1$ and $(b,b')=0$ 
for all $b\neq \pm b'\in B$. Note that by the same proof as in \cite{Ka},
if a crystal basis is polarizable, then $L=\set{x\in M\mid (x,x)\in \cA}$.
\end{dfn}

Let $M$ be a module with a (signed) crystal basis $(L,B)$.
For $i\in I_0$, let 
\[\phi_i(b)=\max\set{t\in \N\mid \tf_i^t b\neq 0},\qquad 
\varepsilon_i(b)=\max\set{t\in \N\mid \te_i^t b\neq 0}.\]
Then note we have $\phi_i(b)-\varepsilon_i(b)=\ang{h_i,|b|}$.
Then crystal bases of modules in $\catOint$
satisfy the following tensor product rule.

\begin{prop}[{\cite[Proposition 2.8, Lemma 2.11, and Theorem 2.12]{BKK}}]
Let $M_1,M_2$ be modules with (signed) crystal bases $(L_1,B_1)$ and $(L_2,B_2)$.
Set $M=M_1\otimes M_2$, $L=L_1\otimes L_2$, and 
$B=B_1\otimes B_2\subset (L_1/qL_1)\otimes (L_2/qL_2)=L/qL$.
Then $(L,B)$ is a signed crystal basis of $M$. 
Moreover, for $b_1\in B_1$ and $b_2\in B_2$,
we have for $i\in I_0$:
\begin{align*}
\te_i(b_1\otimes b_2)&=\begin{cases}
\te_i(b_1)\otimes b_2 & \text{ if } \phi_i(b_1)\geq \varepsilon_i(b_2),\\
b_1\otimes \te_i(b_2)&\text{ otherwise.}
\end{cases}\\
\te_m(b_1\otimes b_2)&=\begin{cases}
\te_m(b_1)\otimes b_2 & \text{ if } \ang{h_m,|b|}>0\\
(-1)^{p(b_1)}b_1\otimes \te_m(b_2)&\text{ otherwise.}
\end{cases}\\
\tf_i(b_1\otimes b_2)&=\begin{cases}
\tf_i(b_1)\otimes b_2 & \text{ if } \phi_i(b_1)> \varepsilon_i(b_2),\\
b_1\otimes \tf_i(b_2)&\text{ otherwise.}
\end{cases}\\
\tf_m(b_1\otimes b_2)&=\begin{cases}
\tf_m(b_1)\otimes b_2 & \text{ if } \ang{h_m,|b|}>0\\
(-1)^{p(b_1)} b_1\otimes \tf_m(b_2)&\text{ otherwise.}
\end{cases}
\end{align*}
Finally, if $(L_1,B_1)$ and $(L_2,B_2)$ are polarizable, then so is $(L,B)$, and 
$M_1\otimes M_2$ is completely reducible.
\end{prop}

The following theorem is a special case of the main result in \cite{BKK}.
\begin{thm}\label{thm:BKK}
Let $\lambda\in P^{++}$. Then $V(\lambda)$ has a polarizable signed crystal basis
$(L(\lambda), B(\lambda)\cup -B(\lambda))$, where 
$$L(\lambda)=\sum \cA \tilde x_{i_1}\ldots \tilde x_{i_t} v_\lambda,\quad \text{with the sum being over } t\geq 0, x\in\set{e,f}, i_1,\ldots, i_t\in I;$$
$$B(\lambda)=\set{\tilde x_{i_1}\ldots \tilde x_{i_t} v_\lambda+qL(\lambda)\mid t\geq 0, x\in\set{e,f}, i_1,\ldots, i_t\in I}\setminus 0.$$
\end{thm}

In the context of $\gl(m|1)$, this result can be improved in two ways.
First, note that in {\em loc. cit.} it was shown that, for $\gl(m|n)$ with $n\geq 2$, 
there may exist $x\in L(\lambda)$ such that $\te_i x\in qL$ for all $i\in I$
but $x\notin L(\lambda)_\lambda$. Consequently, we may need to apply 
both $\te$ and $\tf$ operators to reach every element of the crystal. 

Second, Theorem \ref{thm:BKK} is proved by realizing $V(\lambda)$ as a summand
of a tensor power of the standard representation $V(\epsilon_1)$ (up to
a twist by a one-dimensional module of weight $a1_{m|1}$). Then we can identify
the crystal $B(\lambda)\cup -B(\lambda)$ in a tensor power of $B(\epsilon_1)\cup -B(\epsp_1)$.
However, since the crystal operators on $B(\lambda)$ are determined by the tensor product rule,
there may be signs introduced by commuting the odd Kashiwara operators past odd-parity
vectors. As such, despite $B(\epsilon_1)\cup\set  0$
being closed under the Kashiwara operators, it is not necessarily true
that $B(\lambda)\cup\set  0$ is closed hence $V(\lambda)$ only has a signed crystal basis
in general. 

However, in the case of $\gl(m|1)$, we can remove both of these ambiguities
as shown by the following lemma.

\begin{lem}
Let $\lambda\in P^{++}$.
Then $$B(\lambda)=
\set{\tilde f_{i_1}\ldots \tilde f_{i_t} v_\lambda+qL(\lambda)
\mid t\geq 0, i_1,\ldots, i_t\in I}$$
is a basis of $L(\lambda)/qL(\lambda)$, hence
$$L(\lambda)=\sum \cA \tilde f_{i_1}\ldots \tilde f_{i_t} v_\lambda,\quad \text{with the sum being over } t\geq 0,  i_1,\ldots, i_t\in I;$$
and $(L(\lambda), B(\lambda))$ is a crystal basis of $V(\lambda)$.
\end{lem}

\begin{proof}
We prove this using the realization of $V(\lambda)$ as a direct summand
of $V(\epsp_1)^{\otimes t}\otimes S$ for some $t\in \Z_{\geq 0}$ and one-dimensional
$\UU$-modules $S$ of weight $a1_{m|1}$; since the factor $S$ will make no difference
with respect to the action of crystal operators, 
it suffices to prove this in the case $S$ is trivial
and $\lambda\in \tP^{++}$. Then we can freely identify the crystal with
the (signed) crystal of semistandard Young tableaux in the super alphabet 
$\set{1,\ldots, m}\cup \set{m+1}$;
we refer the reader to \cite[Sections 3.2 and 4]{BKK} for details and examples.

First, let us show that we don't need a signed basis.
To do this, note that the signed crystal $B(\lambda)\cup -B(\lambda)$
is isomorphic to a subcrystal of $V(\epsp_1)^{\otimes t}$.
It suffices to prove that $B(\epsilon_1)^{\otimes t}$ (which is unsigned) is closed
under the action of $\te_i$ and $\tf_i$. Clearly it is closed when $i\in I_0$, so we
only need to show this in the case $i=m$. This is easy, since
$\te_m$ and $\tf_m$ kill all boxes (i.e. elements of $B(\epsilon_1)$) other
than those colored by $m$ and $m-1$ respectively; in particular,
we only move them past boxes of even parity (i.e. those colored by $n\leq m-1$)
so we don't have signs appearing. This shows that $B(\lambda)$,
which is generated by an element of $B(\epsp_1)^t$,
is closed under $\te_i$ and $\tf_i$ and hence
itself a(n unsigned) basis of $L(\lambda)/qL(\lambda)$.

Next, we wish to show that $L(\lambda)$ and $B(\lambda)$ are actually generated from the highest
weight vector by applying sequences of $\tf_i$ for $i\in I$.
This is immediate if there is no 
``fake highest weight vector''; that is, no $b\in B(\lambda)$ such that
$\te_i b=0$ for all $i$ yet $b\neq v_\lambda+qL(\lambda)$.
Indeed, starting from any $b'\in B(\lambda)$, by weight considerations
there is some element $b=\te_{i_m}\ldots \te_{i_1} b'\in B(\lambda)$
such that $\te_j b=0$ for all $j\in I$. If there is no fake highest
weight vector, then $b=v_\lambda+qL(\lambda)$ hence 
$b'=\tf_{i_1}\ldots \tf_{i_m}v_\lambda+qL(\lambda)$.
In the case of $\gl(m|1)$, there is no fake highest weight vector as noted in
the proof of \cite[Theorem 4.8]{BKK}; for the sake of completeness,
we will prove that claim here.

Now to prove that there is no fake highest weight vectors, 
suppose $T$ is a tableaux such that $\te_i T=0$ for all $i$.
Since $T$ is semistandard, let $T'$ be the semistandard tableaux
in the alphabet $J_0$ obtained by deleting the boxes colored by $m+1$.
Note that by the definition of semistandard, each row of $T/T'$ contains 
at most one box. Now, suppose $k$ is the first
nonempty row in the skew diagram $T/T'$, or $k=\infty$ if $T/T'$ is empty.
Then since $e_i T=0$ for all $i\in I_0$, $T'$ must be a highest weight
tableaux with respect to $\gl(m)$. 
If $k=m+1$ or $k=\infty$, then $T$
is the genuine highest weight vector, so assume $k\leq m$.
Since $\te_m T=0$, in any admissible reading of $T$ the first box
read which contains $m$ or $m+1$ must contain $m$,
which implies that there is a box in $T'$ containing $m$ in the first $k-1\leq m-1$ rows.
But the highest weight tableaux with respect to $\gl(m)$
necessarily contain only $j$ in boxes in the $j$-th row, which is a contradiction.
\end{proof}

Now let us observe some more properties of the crystals $(L(\lambda),B(\lambda))$.

\begin{lem}\label{cor: hwt conditions}
Let $\lambda\in P^{++}$.
\begin{enumerate}
\item $\set{u\in L(\lambda)/qL(\lambda)\mid \te_i u=0 \text{ for all } i\in I}=\Q v_\lambda+qL(\lambda)$
\item $\set{u\in V(\lambda)\mid \te_i u\in L(\lambda) \text{ for all } i\in I}=L(\lambda)+V(\lambda)_\lambda$
\end{enumerate}
\end{lem}
\begin{proof}
(1) follows from the fact there are no fake highest weight vectors. For (2), suppose
$\te_i u\in L(\lambda)$ for all $i\in I$.
There is some minimal $t\in \N$ such that $u\in q^{-t}L(\lambda)$.
If $t=0$, then $u\in L(\lambda)$.
If $t>1$, then $\te_i (q^t u)\in q^tL(\lambda)$ for all $i\in I$,
so $\te_i (q^tu+qL(\lambda))=0$. But then by (1),
$u\in V(\lambda)_\lambda$.
\end{proof}

Let $\lambda,\mu\in P^{++}$. Then since $V(\lambda)$ and $V(\mu)$ have polarizable
crystal bases, $V(\lambda)\otimes V(\mu)$ is completely reducible.
In particular,
there exist unique $\UU$-module homomorphisms 
\begin{equation}\label{eq:G,S}
\begin{aligned}
S_{\lambda,\mu}:V(\lambda+\mu)\rightarrow V(\lambda)\otimes V(\mu),\qquad S_{\lambda,\mu}(v_{\lambda+\mu})=v_\lambda\otimes v_\mu;\\
G_{\lambda,\mu}:V(\lambda)\otimes V(\mu)\rightarrow V(\lambda+\mu),\qquad
G_{\lambda,\mu}(v_\lambda\otimes v_\mu)=v_{\lambda+\mu}.
\end{aligned}
\end{equation}
These maps satisfy $G_{\lambda,\mu}S_{\lambda,\mu}=1_{V(\lambda+\mu)}$;
and $(G_{\lambda,\mu}(x),y)=(x,S_{\lambda,\mu}(y))$
for all $x\in V(\lambda)\otimes V(\mu)$ and $y\in V(\lambda+\mu)$,
where $(\cdot,\cdot)$ denotes the polarization on each module.
We also define the $\Um$-linear map
\begin{equation}\label{eq:P}
P_{\lambda,\mu}:V(\lambda)\otimes V(\mu)\rightarrow V(\lambda),\qquad P_{\lambda,\mu}(w\otimes v)=(v_\mu,v)w.
\end{equation}

\begin{cor}
Let $\lambda,\mu\in P^{++}$.
\begin{enumerate}
\item If $b\in B(\lambda)$ and $b'\in B(\mu)$ such that 
$\te_i (b\otimes b')=0$ for any $i\in I$,
then $b=v_\lambda+qL(\lambda)$.
\item If $b\in B(\lambda)$ and $i\in I$,
then $\tf_i(b\otimes v_\mu)=(\tf_ib)\otimes v_\mu$ or 
$\tf_ib=0$.
\item $S_{\lambda,\mu}(L(\lambda+\mu))\subset L(\lambda)\otimes L(\mu)$ 
and $G_{\lambda,\mu}(L(\lambda)\otimes L(\mu))=L(\lambda+\mu)$
and the induced map $$\hat G_{\lambda,\mu}:(L(\lambda)\otimes L(\mu))/q(L(\lambda)\otimes L(\mu))\rightarrow L(\lambda+\mu)/qL(\lambda+\mu)$$
satisfies $\hat G_{\lambda,\mu}(B(\lambda)\otimes B(\mu))=B(\lambda+\mu)\cup\set  0$.
\item $P_{\lambda,\mu}(L(\lambda)\otimes L(\mu))=L(\lambda)$
and the induced map $$\hat P_{\lambda,\mu}:(L(\lambda)\otimes L(\mu))/q(L(\lambda)\otimes L(\mu))\rightarrow L(\lambda)/qL(\lambda)$$
satisfies $\hat P_{\lambda,\mu}\tf_i=\tf_i\hat P_{\lambda,\mu}$ for all $i\in I$.
\end{enumerate}
\end{cor}

\begin{proof}
For (1), first
note that if $\te_i(b\otimes b')=0$, then $\te_i(b)=0$ (since
$\te_i(b\otimes b')=b\otimes\te_i(b')$ would imply 
$\varepsilon_i(b')>\phi_i(b)\geq 0$ if $i\in I_0$ or $\ang{h_m,|b|}=0$ if $i=m$, 
hence in either case $\te_i(b')\neq 0$). 
Since $V(\lambda)$ has no fake highest weight
vectors, $\te_i(b)=0$ for any $i$ implies that $b=v_\lambda$.

For (2), note that if $\tf_i(b\otimes v_\mu)=b\otimes \tf_i v_\mu$,
then $\phi_i(b)\leq \varepsilon_i(v_\mu)=0$ if $i\in I_0$ and $\ang{h_m,|b|}=0$
if $i=m$, hence $\tf_i(b)=0$ in either case.

For (3), note that $S_{\lambda,\mu}$ and $G_{\lambda,\mu}$ 
are $\UU$-module homomorphisms,
hence they preserve the actions of $\te_i$ and $\tf_i$ for $i\in I$.
Moreover, recall that
$$S_{\lambda,\mu}(v_{\lambda+\mu})=v_{\lambda}\otimes v_{\mu},$$
$$G_{\lambda,\mu}(v_{\lambda}\otimes v_{\mu})=v_{\lambda+\mu}.$$
Then $G_{\lambda,\mu}((L(\lambda)\otimes L(\mu))_{\lambda+\mu})
=L(\lambda+\mu)_{\lambda+\mu}$.
In particular, assume $G_{\lambda,\mu}((L(\lambda)\otimes L(\mu))_{\lambda+\mu-\zeta})\subset L(\lambda+\mu)$ for $\height \zeta<l$ with $l>0$. 
Then note that if $\height \zeta=l$, 
$$\te_i G_{\lambda,\mu}((L(\lambda)\otimes L(\mu))_{\lambda+\mu-\zeta})
\subset G_{\lambda,\mu}((L(\lambda)\otimes L(\mu))_{\lambda+\mu-\zeta+\alpha_i})
\subset L(\lambda+\mu).$$
By Corollary \ref{cor: hwt conditions} (2) and weight space considerations, 
we must have
$G_{\lambda,\mu}((L(\lambda)\otimes L(\mu))_{\lambda+\mu-\zeta})
\subset L(\lambda+\mu)$.
Since $L(\lambda+\mu)$ is generated from $\cA v_{\lambda+\mu}$ by
the $\tf_i$'s,
the result follows from the tensor product rule.

For (4), note that for $x\otimes y\in L(\lambda)\otimes L(\mu)$,
$P_{\lambda,\mu}(x\otimes y)=(v_\mu,y)x\in L(\lambda)$ since $x\in L(\lambda)$
and $(v_\mu,y)\in \cA$. Moreover, $P_{\lambda,\mu}$ is easily seen to be surjective
since $P_{\lambda,\mu}(x\otimes v_\mu)=x$. Now $\hat P_{\lambda,\mu}$ commutes
with all $\tf_i$ for $i\in I_0$ by (2) and the definition of $P_{\lambda,\mu}$.
\end{proof}

Finally, let us recall the results of \cite{Kw14}. Note that, in general,
$K(\lambda)\notin \catOint$. Nevertheless, we can define a crystal structure on
$K(\lambda)$. To do this, note that $K(\lambda)$ is naturally still a weight 
$\UU_q(gl(m))$-module, hence the Kashiwara operators $\te_i$ and $\tf_i$ for $i\in I_0$
are well-defined on $K(\lambda)$. On the other hand, $\te_m$ is not a good choice
of crystal operator for $K(\lambda)$ in general, since in general it will not be an ``inverse
at $q=0$'' to $\tf_m$.

We can fix this by using the quantum differential from \eqref{eq:quandiff}. 
Indeed, note that in Lemma \ref{lem:Kac facts} (2), we trivially have
$e_m(I_\lambda)\subset I_\lambda$.
In particular, $e_m$ descends to a map on $K(\lambda)$, so for $x\in K(\lambda)$ we define
\[\tf_i^Kx=\tf_i x,\quad \te_i^K x=\te_i x\text{ for }i\in I_0;\]
\[\tf_m^Kx=\tf_m x=F_m x,\qquad \te_m^K x=e_m(x).\]

\begin{prop}[{\cite[Theorems 4.7-4.11]{Kw14}}]\label{prop:KwResults}
Let $\lambda\in P^{+}$. Then $K(\lambda)$ has a signed crystal basis
(with respect to the crystal operators $\te_i^K,\tf_i^K$ for $i\in I$) given by
$(L^K(\lambda), B^K(\lambda)\cup -B^K(\lambda))$, where 
$$L^K(\lambda)=\sum \cA \tilde x^K_{i_1}\ldots \tilde x^K_{i_t} k_\lambda,\quad \text{with the sum being over } t\geq 0, x\in\set{e,f}, i_1,\ldots, i_t\in I;$$
$$B^K(\lambda)=\set{\tilde x^K_{i_1}\ldots \tilde x^K_{i_t} k_\lambda+qL(\lambda)\mid t\geq 0, x\in\set{e,f}, i_1,\ldots, i_t\in I}\setminus 0.$$
Moreover, if $\lambda\in \tP^{+}$, then the projection 
$\pi_\lambda^{K,V}:K(\lambda)\rightarrow V(\lambda)$ given by $k_\lambda\mapsto v_\lambda$
satsifies \begin{enumerate}
\item $\pi_\lambda^{K,V}(L^K(\lambda))=L(\lambda)$, hence $\pi_\lambda^{K,V}$ induces
a projection 
\[\hat\pi_\lambda^{K,V}:L^K(\lambda)/qL^K(\lambda)\rightarrow L(\lambda)/qL(\lambda);\]
\item $\hat\pi_\lambda^{K,V}(B^K(\lambda))=B(\lambda)\cup \set{0}$; and
\item $\hat\pi_\lambda^{K,V}$ restricts to a weight-preserving bijection
between $\set{b\in B^K(\lambda)\mid \hat\pi_\lambda^{K,V}(b)\neq 0}$ and $B(\lambda)$.
\end{enumerate}
\end{prop}

\section{The crystal $L(\infty)$ and globalization}\label{sec:crystal}

We will now construct a crystal basis on $\Um$ which is compatible
with those on representations as described above.
Subsequently, we will construct a canonical basis on $\Um$ and the representations
by ``globalizing'' the crystal bases in the sense of Kashiwara \cite{Ka}.

\subsection{On the quantum differentials}

The differentials $e_i,\bar e_i$ on $\Um$ are closely tied to the multiplicative
structure of $\UU$ as follows.

\begin{lem}
Let $x\in\UU^-$ and $i\in I$. Then
\[E_ix-(-1)^{p(i)p(x)}xE_i=\frac{K_i\bar e_i(x)-K_i^{-1}e_i(x)}{q-q^{-1}}\]
\end{lem}
\begin{proof}
We proceed by induction on $|x|$. If $x\in\Qq$ or $|x|=-j$ for $j\in I$,
then this is trivial. Now suppose $x=uv$ for some homogeneous $u,v\in\Um$ with 
$|u|,|v|\in Q^-\setminus 0$.
Then by induction 
\begin{align*}
E_ix&=E_iuv=(-1)^{p(i)p(u)}uE_iv +\frac{K_i\bar e_i(u)v-K_i^{-1}e_i(u)v}{q-q^{-1}}\\
&=(-1)^{p(i)p(uv)}uvE_i +\frac{K_i\bar e_i(u)v-K_i^{-1}e_i(u)v}{q-q^{-1}}+(-1)^{p(i)p(u)}\frac{uK_i\bar e_i(v)-uK_i^{-1}e_i(v)}{q-q^{-1}}\\
&=(-1)^{p(i)p(x)}xE_i +
\frac{K_i\bar e_i(x)
-K_i^{-1}e_i(x)}{q-q^{-1}}.
\end{align*}
\end{proof}

\begin{lem}
Let $i\in I$ and $u\in \Um_{-\zeta}$ with $e_i'(u)=0$.
Let $M$ be a $\UU$-module and $\lambda\in P$.
Then for any $x\in M_\lambda$ with $E_ix=0$, we have
\[K_i^tE_i^tux=\frac{q^{t(2\ang{h_i,\lambda-\zeta}+(1-\delta_{im})(3t+1))}}{(q-q^{-1})^t}((\bar e_i)^t(u)x)\]
\end{lem}

\begin{proof}
For $i\in I_0$, the proof is the same as that of \cite[Cor. 3.4.6]{Ka}.
For $i=m$, note that both sides are zero if $t>1$, and for $t=1$
this trivially follows from the previous lemma.
\end{proof}

Finally, we note that the differentials $e_i$, together with the left-multiplication maps $f_i$,
generate a subalgebra of $\End_{\Qq}(\Um)$ which is a version of Kashiwara's boson algebra.
In particular, many of the results of \cite[Section 3.4]{Ka} generalize to our setting;
see \cite{Zou} for some details on this. 

\subsection{The $L(\infty)$ crystal, and a less grand loop}

For each $i\in I$, we define operators $\te_i$, $\tf_i$ on $\Um$ as follows.
Write $u=\sum F_i^{(t)} u_t$ with $e_i(u_t)=0$; for $i\in I_0$, the existence and uniqueness of such
a decomposition follows from \cite[Section 3.4]{Ka}, 
whereas for $i=m$, this follows from \cite[Lemma 4.5]{Kw14}. Then we set
$$\te_i u=\sum F_i^{(t-1)} u_t,\qquad \tf_i u =\sum F_i^{(t+1)} u_t.$$
Note that $\tf_m u = F_mu$, and $\te_m(u)=e_m(u)$.

\begin{dfn}
	We say that an $\cA$-submodule $L$ of $\Um$ is a {\em crystal
		lattice} of $\Um$ if:
	\begin{enumerate}
		\item $\Um=\Qq\otimes_\cA L$;
		\item $L=\bigoplus_{\lambda} (L^0_\lambda\oplus L^1_\lambda)$ 
		where $L^t_\lambda=L\cap M^t_\lambda$;
		\item $\te_i L\subset L$ and $\tf_i L\subset L$ for all $i\in I$.
	\end{enumerate}
	Suppose $L$ is a crystal lattice of $\Um$, and let $B\subset L/qL$.
	We say that $(L,B)$ is a {\em signed crystal basis} of $M$ if:
	\begin{enumerate}[resume]
		\item $B$ is a signed $\Q$-basis of $L/qL$ (that is, $B=B'\cup(-B')$ for some $\Q$-basis of
		$L/qL$);
		\item $B=\bigcup_{\lambda} (B^0_\lambda\cup B^1_\lambda)$ where 
		$B^t_\lambda=B\cap (L^t_\lambda/qL^t_\lambda)$;
		\item $\te_i B\subset B\cup\set{0}$ and $\tf_i B\subset B\cup \set{0}$ for all $i\in I$.
		\item If $b,b'\in B$ and $i\in I$, then $\te_i b=b'$ if and only if $\tf_i b'=b$.
	\end{enumerate}
	If $B$ satisfies (5)-(7) and is a $\Q$-basis of $L/qL$, 
	then we say $(L,B)$ is an unsigned crystal basis of $\Um$,
	or simply a crystal basis of $\Um$.
\end{dfn}

\begin{thm}\label{thm:crystal}
Let $L(\infty)=\sum \cA \tf_{i_1}\ldots\tf_{i_n} 1$
	and $B(\infty)=\set{\tf_{i_1}\ldots\tf_{i_n} 1\mid i_1,\ldots, i_n\in I}\setminus 0$.
	Then $(L(\infty),B(\infty))$ is a crystal basis of $\Um$. Moreover, for $\lambda\in P^{++}$,
	the projection $\pi_\lambda:\Um\rightarrow V(\lambda)$ satisfies $\pi_\lambda(L(\infty))=L(\lambda)$.
	Finally, the induced projection $\hat \pi_\lambda: L(\infty)/qL(\infty)\rightarrow L(\lambda)/qL(\lambda)$
	satisfies $\hat\pi_\lambda(B(\infty))=B(\lambda)\cup 0$ and induces a bijection
	\[\set{b\in B(\infty)\mid \pi_\lambda(b)\neq 0}\leftrightarrow B(\lambda).\]
\end{thm}

For the remainder of this subsection, we will prove the theorem
by induction on the weight grading.
We prove the following statements, which form a truncated
version of Kashiwara's grand loop \cite{Ka}.
\begin{enumerate}
\item[(S$1_l$)] For $\xi\in Q^+(l)$, $\te_i L(\infty)_{-\xi}\subset L(\infty)$.
\item[(S$2_l$)] For $\xi\in Q^+(l)$, $\pi_\lambda L(\infty)_{-\xi}\subset L(\lambda)_{\lambda-\xi}$ for $\lambda\in P^{++}$, so we have the induced
projection $\hat\pi_\lambda: L(\infty)_{-\xi}/qL(\infty)_{-\xi}\rightarrow L(\lambda)_{\lambda-\xi}/qL(\lambda)_{\lambda-\xi}$.
\item[(S$3_l$)] For $\zeta\in Q^+(l)$, $B(\infty)_{-\xi}$ is a basis
for $L(\infty)_{-\xi}/qL(\infty)_{-\xi}$.
\item[(S$4_l$)] For $\xi\in Q^+(l-1)$ and $\lambda\in P^{++}$,
$\tf_i(xv_\lambda)\in (\tf_ix)v_\lambda+qL(\lambda)_{\lambda-\xi}$ for $x\in L(\infty)_{-\xi}$.
\item[(S$5_l$)] For $\xi\in Q^+(l)$,
$\te_iB(\infty)_{-\xi}\subset B(\infty)\cup\set  0$.
\item[(S$6_l$)] For $\xi\in Q^+(l)$ and $\lambda\in P^{++}$,
$\hat\pi_\lambda$ induces a bijection between
$\set{b\in B(\infty)_{-\xi}\mid \hat\pi_\lambda(b)\neq 0}$ and
$B(\lambda)_{\lambda-\xi}$.
\item[(S$7_l$)] For $\xi\in Q^+(l)$ and $\lambda\in P^{++}$,
if $\hat\pi_\lambda(b)\neq 0$ for some $b\in B(\infty)_{-\xi}$,
then $\te_i\hat\pi_\lambda(b)=\hat\pi_\lambda(\te_i b)$.
\item[(S$8_l$)] For $\xi\in Q^+(l)$ (resp. $\xi\in Q^+(l-1)$) and for $b\in B(\infty)_{-\xi}$,
if $\te_ib\neq 0$ (resp. $\tf_i b\neq 0$) then $\tf_i\te_ib=b$
(resp. $\te_i\tf_ib=b$)..
\end{enumerate}

Assume (S$1_{l-1}$)-(S$8_{l-1}$) for $l\geq 2$ (since they are obvious for $l=0,1$).

\begin{lem}[S$4_l$]\label{lem:kash op commute with proj}
Let $i\in I$, $\xi\in Q^+(l-1)$ and $x\in \Um_{-\xi}$.
Suppose $\lambda\in P^{++}$.
Then $(\tf_ix)v_\lambda\equiv\tf_i(xv_\lambda)$ modulo $qL(\lambda)$. If we further assume that
$\ang{h_i,\lambda}\gg 0$, then
$(\te_i x)v_\lambda\equiv \te_i (xv_\lambda)$ modulo $qL(\lambda)$.
\end{lem}

\begin{proof}
If $i\in I_0$, this follows from the same proof as \cite[Lemma 4.4.1, Proposition 4.4.2]{Ka}.
Suppose $i=m$. Then $x=x_0+F_mx_1$ with $e_m(x_0)=e_m(x_1)=0$.
Since $\tf_m x=F_m x$ and $\tf_m(xv_\lambda)=F_mxv_\lambda$ by definition, 
obviously $\tf_m(xv_\lambda)=(\tf_m x)v_\lambda$.

For $\te_m$, it suffices to prove the statement when $x=x_0$ or $x=F_mx_1$.
If $x=x_0$, then $\te_m x=0$. On the other hand,
$$\te_m (xv_\lambda)=q^{-1}K_mE_mxv_\lambda
=\frac{q^{2\ang{h_m,\lambda-\xi}}}{q^2-1} \bar e_m(x)v_\lambda.$$ 
Then since $\bar e_m(x)\in q^{-N}L(\infty)$ for some $N$, and thus by (S$2_{l-1}$)
we have $\bar e_m(x)\in q^{-N}L(\lambda)$ for any $\lambda$,
 we see that 
$q^{2\ang{h_m,\lambda-\xi}}\bar e_m(x)v_\lambda\in qL(\lambda)$ for $\ang{h_m,\lambda}\gg 0$,
hence $\te_m(xv_\lambda)\equiv 0=(\te_m x)v_\lambda$ modulo $qL(\lambda)$
in that case.

If $x=F_mx_1$, then $\te_m x=x_1$. On the other hand,
$$\te_m (xv_\lambda)=q^{-1}K_mE_mxv_\lambda
=\frac{q^{2\ang{h_m,\lambda-\xi}}-1}{q^2-1}x_1v_\lambda-\frac{q^{2\ang{h_m,\lambda-\xi}}}{q^2-1} F_m\bar e_m(x_1)v_\lambda,$$ 
hence in particular if $\ang{h_m,\lambda}\gg 0$, then  \[\te_m(xv_\lambda)-(\te_mx)v_\lambda=\frac{q^{2\ang{h_m,\lambda-\xi}}+q^2}{q^2-1}x_1v_\lambda-\frac{q^{2\ang{h_m,\lambda-\xi}}}{q^2-1} F_m\bar e_m(x_1)v_\lambda\in qL(\lambda).\]
\end{proof}

\begin{cor}[S$2_l$ and S$3_l$]
For any $\lambda\in P^{++}$ and $\xi\in Q^+(l)$, we have
$\pi_\lambda(L(\infty)_{-\xi})=L(\lambda)_{\lambda-\xi}$.
Moreover, the induced map 
$$\hat\pi_\lambda:L(\infty)_{-\xi}/qL(\infty)_{-\xi}\rightarrow
L(\lambda)_{\lambda+\xi}/qL(\lambda)_{\lambda-\xi}$$
satisfies $\hat\pi_\lambda(B(\infty)_{-\xi})\setminus 0=B(\lambda)_{\lambda-\xi}$.
Finally, if $\lambda\gg \xi$,
then $\pi_\lambda$ restricts to an isomorphism 
$L(\infty)_{-\xi}\cong L(\lambda)_{\lambda-\xi}$
and $\hat\pi_\lambda$ restricts to a bijection 
$B(\infty)_{-\xi}\leftrightarrow B(\lambda)_{\lambda-\xi}$.
Consequently, $B(\infty)_{-\xi}$ is a basis for 
$L(\infty)_{-\xi}/qL(\infty)_{-\xi}$.
\end{cor}

\begin{proof}
By Lemma \ref{lem:kash op commute with proj} and (S$2_{l-1}$), $\pi_\lambda(L(\infty)_{-\xi})\subset L(\lambda)$ and
$L(\lambda)_{\lambda-\xi}\subset \pi_\lambda(L(\infty)_{-\xi})+qL(\lambda)_{\lambda-\xi}$,
hence $\pi_\lambda(L(\infty)_{-\xi})= L(\lambda)_{\lambda-\xi}$ by Nakayama's lemma.
Then $\hat\pi_\lambda$ commutes with $\tf_i$ by Lemma 
\ref{lem:kash op commute with proj}, hence 
$\hat\pi_\lambda(B(\infty)_{-\xi})\setminus 0=B(\lambda)_{\lambda-\xi}$.
The final statements follow from the fact that if $\lambda\gg \xi$, 
then $\Um_{-\xi}\cong V(\lambda)_{\lambda-\xi}$
and that $B(\lambda)$ is a basis for $L(\lambda)/qL(\lambda)$.
\end{proof}

\begin{cor}[S$1_l$ and S$5_l$]
For $\xi\in Q^+(l)$, we have 
\[\te_iL(\infty)_{-\xi}\subset L(\infty)\quad\text{and}\quad
\te_iB(\infty)_{-\xi}\subset B(\infty)\cup\set 0.\]
\end{cor}

\begin{proof}
Let $\lambda\in\tPp$ with $\lambda\gg \xi$ and $\ang{h_i,\lambda}\gg 0$.
Let $x\in L(\infty)$. Then $\pi_\lambda(\te_ix)\equiv \te_i (xv_\lambda)$
modulo $qL(\lambda)$ by Lemma \ref{lem:kash op commute with proj}. 
Since $xv_\lambda\in L(\lambda)$  by (S$2_l$), $\te_i(xv_\lambda)\in L(\lambda)$
hence $\pi_\lambda(\te_ix)\in L(\lambda)$. But since $\lambda\gg \xi$,
$\lambda\gg \xi-\alpha_i$ hence $\pi_\lambda$ is an isomorphism
$\Um_{-\xi+\alpha_i}\rightarrow V(\lambda)_{\lambda-\xi+\alpha_i}$ such that
$L(\infty)_{-\xi+\alpha_i}\cong L(\lambda)_{\lambda-\xi+\alpha_i}$.
In particular, $\te_ix \in L(\infty)$.

Now suppose $b=x+qL(\infty)\in B(\infty)$. Then applying $\hat\pi_\lambda$,
$b_\lambda=xv_\lambda+qL(\lambda)\in B(\lambda)_{\lambda-\xi}$,
so $\te_i b_\lambda\in B(\lambda)_{\lambda-\xi+\alpha_i}\cup\set  0$.
Moreover, by Lemma \ref{lem:kash op commute with proj} we have
$\hat\pi_\lambda(\te_i b)=\te_i b_\lambda$. Since $\lambda\gg \xi-\alpha_i$,
$\hat\pi_\lambda$ is a bijection 
$B(\infty)_{-\xi+\alpha_i}\cup\set  0\leftrightarrow
B(\lambda)_{\lambda-\xi+\alpha_i}\cup\set  0$
and thus $\te_i b\in B(\infty)\cup\set  0$.
\end{proof}

\begin{lem}[S$7_l$]
Let $\xi\in Q^-(l)$.
If $b\in B(\infty)_{-\xi}$ with $\te_ib\neq 0$, 
then $\pi_\lambda(\te_ib)=\te_i\pi_\lambda(b)$ for all $\lambda\in \tPp$.
\end{lem}
\begin{proof}
This is proved exactly as in \cite[\S 4.6]{Ka}.
Namely, observe that for any $\lambda,\mu\in P^{++}$ and $x\in L(\lambda)$,
we have $\te_i(x\otimes v_\mu)\equiv \te_i x\otimes v_\mu$
modulo $qL(\lambda)\otimes L(\mu)$; the proof in {\em loc. cit.} is valid
for $i\in I_0$, and for $i=m$ observe that $\te_m(b\otimes v_\lambda)=(q^{-1}K_m E_m x)\otimes v_\mu$. 

Now fix $\lambda\in P^{++}$, and pick a $\mu\in P^{++}$
such that $\lambda+\mu\gg \xi$. In $L(\lambda+\mu)$, $\te_i$ commutes with $\pi_{\lambda+\mu}$ modulo $q$ as in the proof of (S$1_l$).
Note that we can push this congruence to $L(\lambda)\otimes L(\mu)$ by applying 
$S_{\lambda,\mu}$, and then also to $L(\lambda)$ by applying $P_{\lambda,\mu}$.

In particular, suppose $b\in B(\infty)_{-\xi}$ is given by
$b=\tf_{i_1}\ldots\tf_{i_t} 1+qL(\infty)$.
It suffices to show 
$(\te_i\tf_{i_1}\ldots\tf_{i_t} 1)v_\lambda
\equiv\te_i(\tf_{i_1}\ldots\tf_{i_t} v_\lambda)$
modulo $qL(\lambda)$. Well, observe that 
$(\te_i \tf_{i_1}\ldots\tf_{i_t} v_\lambda)\otimes v_\mu\equiv\te_i\tf_{i_1}\ldots\tf_{i_t} (v_\lambda\otimes v_\mu)$.
Applying $S_{\lambda,\mu}$ to the congruence 
$\te_i\tf_{i_1}\ldots\tf_{i_t} v_{\lambda+mu}
\equiv  (\te_i\tf_{i_1}\ldots\tf_{i_t} 1)v_{\lambda+mu}$ modulo $qL(\lambda+\mu)$,
we have $\te_i\tf_{i_1}\ldots\tf_{i_t} (v_\lambda\otimes v_\mu)\equiv 
(\te_i\tf_{i_1}\ldots\tf_{i_t} 1) v_\lambda\otimes v_\mu$
modulo $qL(\lambda)\otimes L(\mu)$.
But then applying $P_{\lambda,\mu}$, we obtain the desired congruence.
\end{proof}

\begin{cor}[S$8_l$]
Let $\xi\in Q^+(l)$ (resp. $\xi\in Q^+(l-1)$) and $b\in B(\infty)_{-\xi}$.
Suppose $\te_i b\neq 0$ (resp. $\tf_i b\neq 0$).
Then $\tf_i\te_i b=b$ (resp. $\te_i\tf_i b=b$).
\end{cor}

\begin{proof}
Let $b\in B(\infty)_{-\xi}$.
Pick $\lambda\in \tPp$ with $\lambda\gg \xi$.
We consider the case $\xi\in Q^+(l)$ and $\te_i b\neq 0$, as 
the other case follows by a similar argument.
Then since $\hat \pi_\lambda$ is a bijection, 
$\hat \pi_\lambda(\te_i b)=\te_i \hat\pi_{\lambda}(b)\neq 0$,
and so $\hat\pi_{\lambda}(b)=\hat\pi_{\lambda}(b)=\hat\pi_{\lambda}(\tf_i\te_ib)$
thus $\tf_i\te_ib=b$. 
\end{proof}

\begin{cor}[S$6_l$]
Let $\xi\in Q^+(l)$.
Then for any $\lambda\in P^{++}$, $\hat\pi_\lambda$ induces a bijection between
$\set{b\in B(\infty)_{-\xi}\mid \pi_\lambda(b)\neq 0}$ and 
$B(\lambda)_{\lambda-\xi}$.
\end{cor}

\begin{proof}
Let $\lambda\in P^{++}$.
We already know that 
$\hat\pi_\lambda(B(\infty)_{-\xi})\setminus 0=B(\lambda)_{\lambda-\xi}$.
Now suppose $b,b'\in B(\infty)_{-\xi}$ such that 
$\hat\pi_{\lambda}(b)=\hat\pi_{\lambda}(b')\neq 0$.
Clearly, there is some $i\in I$ such that $\te_i b\neq 0$,
hence $\hat\pi_{\lambda}(\te_ib)=\hat\pi_{\lambda}(\te_ib')\neq 0$. 
But then by (S$6_{l-1}$), $\te_ib=\te_ib'$, so $b=\tf_i\te_ib=\tf_i\te_ib'=b'$.
\end{proof}

This finishes the induction, and thus Theorem \ref{thm:crystal} is proven.

\subsection{Some further properties of $L(\infty)$}

We can now deduce some properties of the crystal on $\Um$.
First, we note that $L(\infty)$ enjoys many favorable properties
with respect to the bilinear form on $\Um$.

\begin{prop}\label{prop:BilLat}
We have the following.
\begin{enumerate}
\item $(L(\infty),L(\infty))\subset A$ and hence $(-,-)$ descends to a 
$\Q$-valued bilinear form $(-,-)_0$ on $L(\infty)/qL(\infty)$.
\item $(\te_i u,v)_0=(u,\tf_iv)_0$ for $u,v\in L(\infty)/qL(\infty)$.
\item For any $b,b'\in B(\infty)$, $(b,b')_0=\delta_{b,b'}$.
In particular, $(-,-)_0$ is positive definite.
\item $L(\infty)=\set{x\in \Um\mid (x,L(\infty))\in A}=\set{x\in \Um\mid (x,x)\in A}$.
\item $\tau(L(\infty))=L(\infty)$.
\end{enumerate}
\end{prop}

\begin{proof}
First, let us prove $(L(\infty)_{-\xi},L(\infty)_{-\xi})\subset A$
for $\xi\in Q^+(l)$ by induction on $l$. In particular,
since $L(\infty)_{-\xi}=\sum \tf_i L(\infty)_{-\xi+\alpha_i}$,
it suffices to show that $(\tf_i u,v)\equiv (u,\te_i v)$ modulo $qL(\infty)$
for $u\in  L(\infty)_{-\xi+\alpha_i}$ and $v\in  L(\infty)_{-\xi}$.
We may further assume $u=F_i^{(x)}u_0$ and $v=F_i^{(y)}v_0$
for some $x,y\in \N$ and $u_0,v_0\in \Um$ with $e_i'(u_0)=e_i'(v_0)=0$.
Then the case $i\in I_0$ is virtually the same as in \cite{Ka}.
For $i=m$, it is even simpler: $(\tf_mu,v)=(F_mu,v)=(u,e_m'(v))=(u,\te_mv)$.
This proves (1) and (2).  Then (3) and (4) follow exactly as proven in \cite{Ka},
Finally, (5) follows from (4) and the $\tau$-invariance of $(-,-)$.
\end{proof}

Now let us compare the crystal on $\Um$ with the crystal on $K(\lambda)$.
To facilitate this, we need to refer to the odd PBW vectors. To that end,
let $I=\set{1<\ldots<m}$ be the standard
ordering on $I$. Then
for $\chi=\set{\alpha_1<\ldots<\alpha_k}\subset \Phi^+_1$,
let $F_\chi=F_{<,\alpha_1}\ldots F_{<,\alpha_k}$ and observe that $e_i'(F_\chi)=0$
for all $i\in I_0$.

In particular, let $\mathbf{W}$ be Kashiwara's Boson algebra for $\gl(m)$; that is, 
the $\Qq$-subalgebra of $\End_{\Qq}(\Um)$ generated by $e_i$, $f_i$ with $i\in I_0$.
Let us write $(L_{\gl(m)}(\infty), B_{\gl(m)}(\infty))$ for the crystal basis of $\UU_q^-(\gl(m))$.
Then by \cite[Remarks 3.4.10 and 3.5.1]{Ka}, $\Um$ is a direct sum of $\mathbf{W}$-modules isomorphic
to $\UU_q^-(\gl(m))$ and and $(L(\infty), B(\infty))$ is a direct sum of crystals
isomorphic to $(L_{\gl(m)}(\infty), B_{\gl(m)}(\infty))$. In fact, we can be more specific:
we see that as $\mathbf{W}$-modules, $\Um=\bigoplus_{\chi\subset \Phi^+_1} \Um_q(\gl(m))F_\chi$,
and with respect to this decomposition\footnote{Note that in \cite{Zou}, while a similar decomposition of the lattice is claimed,
the PBW vectors there are taken with respect to the opposite order. As a result,
$F_mF_{m-1}$ is not included in the lattice, causing it to not
be closed under all the operators. This is not a problem here, and as we shall see
any PBW basis maps onto the crystal basis modulo $q$.}
$L(\infty)=\bigoplus L_\chi$ where $L_\chi$ is the sublattice 
generated by $F_\chi$ under $\tf_i$ for $i\in I_0$.

\begin{thm}\label{thm:KacCompat}
	Let $\lambda\in P^+$, and let $\pi^K_\lambda:\Um\rightarrow K(\lambda)$ be
	the projection map. Then $\pi^K_\lambda(L(\infty))=L^K(\lambda)$.
	Moreover, let $\hat\pi^K_\lambda:L(\infty)/qL(\infty)\rightarrow L^K(\lambda)/qL^K(\lambda)$
	be the induced projection modulo $q$. Then $\hat \pi^K_\lambda$  induces a bijection between
	$\set{b\mid b\in B(\infty)\text{ such that } \pi^K_\lambda(b)\neq 0}$ and $B^K(\lambda)$.
\end{thm}

\begin{proof}
	Observe that since $\pi^K_\lambda(1)=k_\lambda$, if
	$\pi_\lambda^K(\tilde f_i y)\equiv \tilde f_i \pi_\lambda^K(y)$ 
	modulo $qL(\lambda)$ for $i\in I$ and $y\in \Um$ then 
	$\pi^K_\lambda(L(\infty))=L^K(\lambda)$.
	This is obvious from the definition
	when $i=m$. When $i\in I_0$, suppose $y\in \Um$.
	Without loss of generality, we can assume $y=xF_\chi\in \Um_q(\gl(m)) F_\chi$
	for some $\chi\subset \Phi_1^-$ and $x\in \Um_q(\gl(m))$. Then
	$\tf_i(xF_\chi)=(\tf_ix)F\chi$ hence 
	$\pi^K_\lambda(\tf_i(xF_\chi))=(\tf_ix)\pi^K_\lambda(F_\chi)$.
	
	Now, as a $U_q(\gl(m))$-module we have a decomposition into irreducibles
	$K(\lambda)\cong\bigoplus_{s=1}^n V(\lambda;\lambda'_s)$ 
	and by construction and the tensor product rule, we have corresponding
	decompositions of the crystal basis:
	$L^K(\lambda)\cong \bigoplus L(\lambda'_s)$ and
	$B^K(\lambda)=\coprod B(\lambda'_s)$.	
	Let $\pi_\lambda^s$ be the composition of $\pi_\lambda^K$ with projection
	on the $s$th component of this direct sum. 
	Then by \cite[($C_l.6$)]{Ka},
	 $(\tf_ix)\pi_\lambda^s(F_\chi)\equiv \tf_i\pi_\lambda^s(xF_\chi)$
	 modulo $qL(\lambda'_s)$ for each $s$. In particular,
	 $\pi^K_\lambda(\tf_i(xF_\chi))\equiv\tf_i\pi^K_\lambda(xF_\chi)$
	 modulo $qL^K(\lambda)$.
	 
	 This proves that $\pi_\lambda^K$ preserves the lattice.
	 It is then easy to see that $\hat\pi_\lambda^K$ induces the desired bijection on the bases,
	 as it does so on each $\gl(m)$-component of the crystals.
\end{proof}

Note that Theorem \ref{thm:KacCompat} implies a slight refinement of 
Proposition \ref{prop:KwResults}.

\begin{cor}
Let $\lambda\in P^+$. Then we have
$$L^K(\lambda)=\sum \cA \tilde f^K_{i_1}\ldots \tilde f^K_{i_t} k_\lambda,\quad \text{with the sum being over } t\geq 0, i_1,\ldots, i_t\in I;$$
$$B^K(\lambda)=\set{\tilde f^K_{i_1}\ldots \tilde f^K_{i_t} k_\lambda+qL(\lambda)\mid t\geq 0, i_1,\ldots, i_t\in I}\setminus 0.$$
Moreover, Proposition \ref{prop:KwResults} (1)-(3) also hold for 
$\lambda\in P^{++}$.
\end{cor}

\subsection{The integral form and globalization}

Finally, we turn to the construction of canonical bases. For this, we first need
to produce an integral form of our lattices.

Recall the integral form $\Uint$ defined in Section \ref{subsec:quagrp}.
Let $\Lint(\infty)=L(\infty)\cap \Umint$. Note that
$\Umint$ is closed under $e_i(x)$ for all $i\in I$; indeed, in the case
$i\in I_0$ this follows from properties of Kashiwara's boson algebra,
and if $i=m$ this is trivial to verify.
In particular, let $i\in I$  and write 
$u=\sum F_i^{(t)} u_t\in \Umint$ where $u_t\in \Um$ with $e_i'(u_t)=0$.
Then it is easy to verify that all $u_t\in \Umint$ as well, hence $\Umint$ and thus $\Lint(\infty)$ 
is closed under $\te_i$ and $\tf_i$
for all $i\in I$.

Now observe 
\begin{equation}
B(\infty)\subset \Lint(\infty)/q\Lint(\infty)\subset L(\infty)/qL(\infty).
\end{equation}
Recall $\cA_\Z$ is the $\Z$-subalgebra of $\Qq$ generated by $q$ and $\frac{1}{1-q^{2t}}$
for $t\geq 1$. Let $K_\Z$ be the subalgebra generated by $\cA_\Z$ and $q^{-1}$.
Then we have $\cA_\Z=\cA\cap K_\Z$. On the other hand,
$(\Umint,\Umint)\subset K_\Z$ hence $(\Lint(\infty),\Lint(\infty))\subset \cA_\Z$.
Therefore, since $\cA_\Z/q\cA_\Z=\Z$, we see that the specialization of the bilinear form
$(-,-)_0$ is $\Z$-valued on $\Lint(\infty)$.

Then the following lemma immediately follows from this discussion and
Proposition \ref{prop:BilLat}.

\begin{lem}
We have that $\Lint(\infty)/q\Lint(\infty)$ is a free $\Z$-module with basis $B(\infty)$.
Moreover, \[B(\infty)\cup -B(\infty)=\set{u\in \Lint(\infty)/q\Lint(\infty)\mid (u,u)_0=1}.\]
\end{lem}

Note that Lemma \ref{lem:CHW3PBW} immediately implies the following.

\begin{cor}\label{cor:prec cb}
Let $\prec$ be a total order on $I$.
Then $\cB(\prec)\subset \Lint(\infty)$, and moreover 
\[\cB(\prec)+q\Lint(\infty)\subset B(\infty)\cup -B(\infty).\]
\end{cor}

Now let $\lambda\in P^{++}$ and $\mu\in P^+$. Set $\Vint(\lambda)=\Umint v_\lambda$ and 
$\Kint(\mu)=\Umint k_\mu$.
Then $\Vint(\lambda)$ (resp. $\Kint(\mu)$) is a $\Uint$-submodule of $V(\lambda)$ (resp. $K(\mu)$). 
Furthermore, define
$\Lint(\lambda)=\Vint(\lambda)\cap L(\lambda)$ and
$\Lint^K(\mu)=\Kint(\mu)\cap L^K(\mu)$. We define a bar involution on $V(\lambda)$
(resp. $K(\mu)$)
by $\bar{uv_\lambda}=\bar u v_\lambda$ (resp. $\bar{uk_\mu}=\bar u k_\mu$) for $u\in \Um$. Then $\Kint(\mu)$ and 
$\Vint(\lambda)$ are stable under the bar involution.
Note that $\pi_\lambda(\Lint(\infty))=\Lint(\lambda)$ (resp. $\pi_\mu^K(\Lint(\mu))=\Lint^K(\mu)$)
and $B(\lambda)\subset \Lint(\lambda)/q\Lint(\lambda)$ (resp. $B^K(\mu)\subset \Lint^K(\mu)/q\Lint^K(\mu)$).

We now have the necessary ingredients to construct Kashiwara's globalization.

\begin{thm}\label{thm:Globalization}
For every $\xi\in Q^+$, $\mu\in P^+$, and $\lambda\in P^{++}$, there are isomorphisms
\[G:\Lint(\infty)_{-\xi}/q\Lint(\infty)_{-\xi}
\rightarrow \Umint_{-\xi}\cap\Lint(\infty)\cap\bar{\Lint(\infty)}.\]
\[G_{\mu}^K:\Lint^K(\mu)_{\mu-\xi}/q\Lint^K(\mu)_{\mu-\xi}
\rightarrow \Kint(\mu)_{\mu-\xi}\cap\Lint^K(\mu)\cap\bar{\Lint^K(\mu)}.\]
\[G_\lambda:\Lint(\lambda)_{\lambda-\xi}/q\Lint(\lambda)_{\lambda-\xi}
\rightarrow \Vint(\lambda)_{\lambda-\xi}\cap\Lint(\lambda)\cap\bar{\Lint(\lambda)}.\]
We call $\cB=G(B(\infty))$ (resp. $\cB(\lambda)=G_\lambda(B(\lambda))$, $\cB^K(\mu)=G^K_\mu(B^K(\mu))$) 
the canonical basis of $\Um$
(resp. the canonical basis of $V(\lambda)$). Furthermore, we observe that:
\begin{enumerate}
\item $\Umint_{-\xi}\cap \Lint(\infty)=\bigoplus_{b\in B(\infty)_{-\xi}}\Z[q]G(b)$
and $\Umint_{-\xi}=\bigoplus_{b\in B(\infty)_{-\xi}}\bbA G(b)$;
\item $\Kint(\mu)_{\mu-\xi}\cap \Lint^K(\mu)
=\bigoplus_{b\in B(\mu)_{\mu-\xi}}\Z[q]G_\mu(b)$
and $\Kint(\mu)_{\mu-\xi}
=\bigoplus_{b\in B(\mu)_{\mu-\xi}}\bbA G_\mu(b)$;
\item $\Vint(\lambda)_{\lambda-\xi}\cap \Lint(\lambda)
=\bigoplus_{b\in B(\lambda)_{\lambda-\xi}}\Z[q]G_\lambda(b)$
and $\Vint(\lambda)_{\lambda-\xi}
=\bigoplus_{b\in B(\lambda)_{\lambda-\xi}}\bbA G_\lambda(b)$;
\item $G_\lambda(\pi_\lambda(b))=\pi_\lambda(G(b))$, $G_\mu^K(\pi^K_\mu(b))=\pi^K_\lambda(G(b))$, and $G_\lambda(\pi^{K,V}_\lambda(b))=\pi^{K,V}_\lambda(G^K_\lambda(b))$ ;
\item $\bar{G(b)}=G(b)$ for any 
$b\in \Lint(\infty)/q\Lint(\infty)$.
\end{enumerate}
\end{thm}

\begin{proof}
Observe that for any total order $\prec$ on $I$, the $\prec$-canonical basis
generates a $\Z$-submodule $Z=\Z\cB(\prec)_{-\xi}$ of $\Umint_{-\xi}$ 
such that $\bbA Z=\Umint_{-\xi}$. 
Moreover, by Corollary \ref{cor:prec cb} and because 
$\cB(\prec)_{-\xi}$ is a bar-invariant basis of $\Umint_{-\xi}$,
clearly 
$Z\subset \Umint\cap \Lint(\infty)\cap \bar{\Lint(\infty)}$
and we see that the projections $Z\rightarrow \Lint(\infty)_{-\xi}/q\Lint(\infty)_{-\xi}$
and $Z\rightarrow \bar{\Lint(\infty)_{-\xi}}/q^{-1}\bar{\Lint(\infty)_{-\xi}}$
are injective. Then applying \cite[Lemma 7.1.1 (ii)]{Ka},
we obtain the isomorphism $G$. An entirely similar argument
applies to the modules, using instead the fact that the nonzero images of $\cB(\prec)$ under the
projection is
equal to the crystal basis modulo $q$, up to possible signs, and hence lifts to a basis
of the lattice, and thus the module, by Nakayama's lemma. The statements (1)-(5) follow
from the same arguments as in \cite[Section 7]{Ka}.
\end{proof}

\begin{rmk}
We remark that the result can be proven
independently of the results of \cite{CHW3} by mimicking the proof in \cite[Section 7]{Ka}.
In particular, we can also show directly that if $G(b)\equiv_q \tf_i^n b$, then $G(b)\in F_i^{(n)}\Um$.
However, this will not be necessary here (and indeed, will follow from the PBW realization
of the crystal and canonical bases).
\end{rmk}

\section{Braid Operators}\label{sec: braid}

We now aim to show that the canonical basis arising from the crystal structure on $\Um$
is compatible with the PBW bases of \cite{CHW3}. To do this, we need to reinterpret
the PBW bases using analogues of Lusztig's braid operators. As mentioned in
the introduction, this requires some shifting of perspective, as now these are
not necessarily automorphisms but rather a network of isomorphisms between quantum enveloping algebras
associated to different choices of simple roots. 

\subsection{Perspective}\label{subsec:perspective}

To clarify the definitions in the following sections, it helps to view Lusztig's braid operators 
in a particular way (which no doubt is well known to experts, but nevertheless we find to be
obscured in the standard definition of these maps).
To be concrete, let's take the example of $\UU_q(\gl(3))$. Usually, we define the
this algebra with generators $E_1,E_2,F_1,F_2, q^h$ for $h\in P^\vee$
subject to the usual relations, which one might call the {\em agnostic} 
presentation as we make no definitive choice of simple roots. 
We then define the braid operator $T_2$ as an automorphism lifting
the action of the reflection $s_2=s_{\alpha_2}$ on the weight data;
that is, it is an automorphism which sends the $\mu$-weight space
to the $s_2(\mu)$ weight space, and the $q^h$ to $q^{s_2(h)}$.
Explictly, we have e.g. 
\begin{equation}\label{eq:gl3braid}
T_2(E_2)=-K_2^{-1}F_2,\quad  T_2(E_1)=E_1E_2-q^{-1}E_2E_1,\quad T_1(q^{\alpha_2^\vee})=q^{-\alpha_2^\vee},\quad T_1(q^{\alpha_1^\vee})=q^{\alpha_1^\vee+\alpha_2^\vee}.
\end{equation} 
For the sake of argument
let us consider two formally different {\em gnostic} versions of this algebra
which commit to a choice of simple roots:
\begin{enumerate}
\item $\UU_1$ is the $\Qq$ algebra on generators 
\[X_{\epsp_1-\epsp_2},\quad X_{\epsp_2-\epsp_1},\quad  
X_{\epsp_2-\epsp_3},\quad  X_{\epsp_3-\epsp_2},\quad 
q^h \text{ for }h\in P^\vee,\]
satisfying the relations of $\UU_q(\gl(3))$ with the replacements 
$E_1\leftrightarrow X_{\epsp_1-\epsp_2}$, $E_2\leftrightarrow X_{\epsp_2-\epsp_3}$
$F_1\leftrightarrow X_{\epsp_2-\epsp_1}$, $F_2\leftrightarrow X_{\epsp_3-\epsp_2}$.
In particular, note that the weights satisfy $|X_{\alpha}|=\alpha$.
\item $\UU_2$ is the $\Qq$ algebra on generators 
\[X_{\epsp_1-\epsp_3},\quad X_{\epsp_3-\epsp_1},\quad  
X_{\epsp_3-\epsp_2},\quad  X_{\epsp_2-\epsp_3},\quad 
q^h \text{ for }h\in P^\vee,\]
satisfying the relations of $\UU_q(\gl(3))$ with the replacements 
$E_1\leftrightarrow X_{\epsp_1-\epsp_3}$, $E_2\leftrightarrow X_{\epsp_3-\epsp_2}$
$F_1\leftrightarrow X_{\epsp_3-\epsp_1}$, $F_2\leftrightarrow X_{\epsp_2-\epsp_3}$.
Again, $|X_{\alpha}|=\alpha$.
\end{enumerate}
Then we can think of $T_2$ as a weight-preserving isomorphism $\UU_1\rightarrow\UU_2$
which is the identity on $q^h$ for $h\in P^\vee$: translating \eqref{eq:gl3braid} 
into this notation, we see that
\[T_2(X_{\epsp_2-\epsp_3})=-q^{\epsp_2^\vee-\epsp_3^\vee}X_{\epsp_2-\epsp_3},\quad
T_2(X_{\epsp_1-\epsp_2})=X_{\epsp_1-\epsp_3}X_{\epsp_3-\epsp_2}-q^{-1}X_{\epsp_3-\epsp_2}
X_{\epsp_1-\epsp_3},\] 
\[
T_2(q^{\epsp_2^\vee-\epsp_3^\vee})=q^{-(\epsp_3^\vee-\epsp_2^\vee)}
=q^{\epsp_2^\vee-\epsp_3^\vee},\quad 
T_1(q^{\epsp_1^\vee-\epsp_2^\vee})=q^{\epsp_1^\vee-\epsp_3^\vee}q^{\epsp_3^\vee-\epsp_2^\vee}
=q^{\epsp_1^\vee-\epsp_2^\vee}.
\]

Now in the case of $\UU_q(\gl(m|1))$, the situation is similar. The crucial
difference is that we cannot formally identify different agnostic Chevalley-Serre-Yamane 
presentations of the algebra like we can in the classical case: indeed,
interpreting \eqref{eq:gl3braid} in terms of $\UU_q(\gl(2|1))$,
an obvious problem  is that $E_1^2$ is nonzero and yet
$(E_1E_2-q^{-1}E_2E_1)^2=0$. Instead, we must treat $T_2$ in this example
(and more generally, any odd braid operator) as an isomorphism
to a different presentation of the same quantum enveloping algebra: 
one with generators $E_1,E_2$ satisfying $E_1^2=E_2^2=0$,
coming from the choice of simple roots where both are isotropic.

In the following, we will make several simplifications for the sake of controlling
notation and to make the construction as similar to Lusztig's as possible.
To that end, we will work with the quantum enveloping algebras associated to $\fsl(m|1)$ 
rather than $\gl(m|1)$ and we will stick
to the agnostic presentations; this will allow us to avoid some tedious notation.
In particular, note that the different quantum enveloping algebras will have {\bf different} sets
of simple roots, despite the common labeling of generators, 
so it may be helpful to think of them as being different algebras rather
than different presentations of the same algebra.

\subsection{The Cartan data orbit}

To that end, we say an $I\times I$ matrix $A$ is a generalized Cartan matrix 
(or {\em GCM} for short) of type 
$\gl(m|1)$ if $A=[A_{ij}]_{i,j\in I}$ satisfies
\begin{enumerate}
\item $A_{ii}\in\set{0,2}$ for all $i\in I$;
\item There exists $i\in \set{0,\ldots, m}$ such that $A_{ii}=A_{i+1,i+1}=0$
	and $A_{jj}=2$ for all $j\in I$ with $j\neq i,i+1$. (Here, we include $a_{00}=0$ and $a_{m+1,m+1}=0$
	for convenience.)
\item $A_{ij}=A_{ji}=0$ for $j\neq i\pm 1$;
\item If $i,i+1\in I$, then 
$A_{i,i+1}=A_{i+1,i}=-1$ if $a_{i,i}=2$ or $A_{i+1,i+1}=2$; otherwise, $A_{i,i+1}=A_{i+1,i}=1$.
\end{enumerate}
Note that any such matrix $A$ is associated to a positive system of roots $\Phi^+(A)\subset \Phi$
and choice of simple roots $\Pi(A)\subset \Phi^+(A)$ in the sense that, writing $\Pi(A)=\set{\alpha_i^A\mid i\in I}$,
$(\alpha_i^A,\alpha_j^A)=A_{ij}$.
Associated to $A\in \cA$ is a parity function $p_A:I\rightarrow \Z/2\Z$, explicitly given
by $p_A(i)=p(\alpha_i^A)=1-A_{ii}/2$. In particular, we define
$I_{A,1}=\set{i\in I\mid A_{ii}=0}$ to be the odd roots relative to $A$.
Finally, observe that $I_{A,1}$ completely determines the matrix; indeed,
for $i,j,k\in I$ with $j=i\pm 1$ and $k\neq i,i\pm 1$, we have
\begin{equation}\label{eq:parity and matrix entries}
X_{ii}=1+(-1)^{p(i)},\quad X_{i}=(-1)^{1-p(i)p(j)},\quad X_{ik}=0.
\end{equation}

It will be convenient in the following to introduce the following shorthand notation.
We will say $i,j\in I$ are {\em connected} if $i=j\pm 1$,
and write $i\sim j$. Likewise, we say {\em not connected} 
if $i\neq j, j\pm 1$, and write $i\nsim j$. Note that given $i,j\in I$, we either
have $i\sim j$, $i\nsim j$, or $i=j$.

Let $\cA$ be the collection of GCMs of type $\gl(m|1)$.
We can write $\cA=\set{A^{t}\mid 0\leq t\leq m}$ where $A^t$ is the unique
GCM with $a_{tt}=a_{t+1,t+1}=0$; in particular, note that $A^m$ is the GCM
associated to the standard root system $\Pi$ defined in Section \ref{subsec:root data}.
There is an action of $F_2(I)$ on $\cA$ given by $i\cdot A^t=A^t$ for $i\neq t,t+1$,
$t\cdot A^t=A^{t-1}$, $(t+1)\cdot A^t=A^{t+1}$.
Note that this action satisfies 
\[i\cdot j\cdot A^t=j\cdot i\cdot A^t\]
\[i\cdot(i+1)\cdot i\cdot A^t=(i+1)\cdot i\cdot (i+1)\cdot A^t\]
for all $j\neq i,i+1$ and $0\leq t\leq m$.
In particular, the $F_2(I)$-action factors through $S_{m+1}$, the symmetric
group on $m+1$ letters.

\begin{example}
In the case $m=4$, then $\cA$ with its $S_5$-action is
described in Figure \ref{fig:gl(4|1) and S_5 action}.
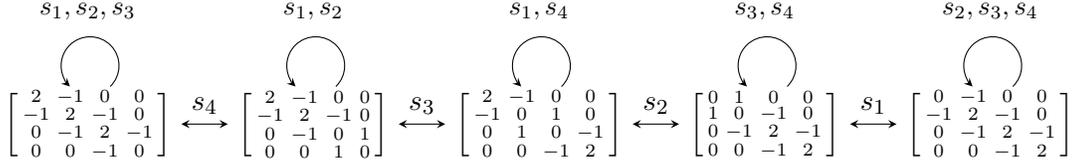
\begin{figure}
\begin{tikzpicture}
\draw (0,0) node (1) {\small 
	$\left[\begin{smallmatrix}2 &-1&0&0\\-1&2&-1&0\\0&-1&2&-1\\0&0&-1&0\end{smallmatrix}\right]$};
\draw (3,0) node (2) {\small 
	$\left[\begin{smallmatrix}2 &-1&0&0\\-1&2&-1&0\\0&-1&0&1\\0&0&1&0\end{smallmatrix}\right]$};
\draw (6,0) node (3) {\small 
	$\left[\begin{smallmatrix}2 &-1&0&0\\-1&0&1&0\\0&1&0&-1\\0&0&-1&2\end{smallmatrix}\right]$};
\draw (9,0) node (4) {\small 
	$\left[\begin{smallmatrix}0 &1&0&0\\1&0&-1&0\\0&-1&2&-1\\0&0&-1&2\end{smallmatrix}\right]$};
\draw (12,0) node (5) {\small 
	$\left[\begin{smallmatrix}0 &-1&0&0\\-1&2&-1&0\\0&-1&2&-1\\0&0&-1&2\end{smallmatrix}\right]$};
\draw[stealth-stealth] (1.east) -- (2.west) node[midway, above] {$s_4$};
\draw[stealth-stealth] (2.east) -- (3.west) node[midway, above] {$s_3$};
\draw[stealth-stealth] (3.east) -- (4.west) node[midway, above] {$s_2$};
\draw[stealth-stealth] (4.east) -- (5.west) node[midway, above] {$s_1$};
\draw[-stealth] (.3,.5) arc(-46:226:1em); 
\draw (0,1.5) node{\small $s_1,s_2,s_3$};
\draw[-stealth] (3.3,.5) arc(-46:226:1em); 
\draw (3,1.5) node{\small $s_1,s_2$};
\draw[-stealth] (6.3,.5) arc(-46:226:1em); 
\draw (6,1.5) node{\small $s_1,s_4$};
\draw[-stealth] (9.3,.5) arc(-46:226:1em); 
\draw (9,1.5) node{\small $s_3,s_4$};
\draw[-stealth] (12.3,.5) arc(-46:226:1em); 
\draw (12,1.5) node{\small $s_2,s_3,s_4$};
\end{tikzpicture}
\caption{The set $\cA$ for $\gl(4|1)$, with arrows indicating
	the nontrivial actions of the simple reflections of $S_{5}$. }\label{fig:gl(4|1) and S_5 action}	
\end{figure}

\end{example}

\begin{rmk}\label{rem:symm group action}
Essentially, the $S_{m+1}$-action on $\cA$ corresponds to the $S_{m+1}$-action
on $\Phi$ induced by $\sigma\cdot \epsilon_j=\epsilon_{\sigma(j)}$
as in the $\gl(m+1)$ case. The crucial difference between the $\gl(m+1)$ and $\gl(m|1)$ cases is 
that the permutations moving $\epsp_m$ change the GCM:
indeed, for the $m=2$ case, the transposition $(2,3)$ applied to the standard
simple roots $\epsp_1-\epsp_2,\epsp_2-\epsp_3$ yields the simple roots
$\epsp_1-\epsp_3,\epsp_3-\epsp_2$, both of which are isotropic!
\end{rmk}

Note that we can easily detect how this action changes parities, which
in turn allows us to determine the matrix entries.
Indeed, let $i,j,k\in I$ with $i$ connected to $j$, and $i$ not connected to $k$.
Then for any $X\in\cA$,
\begin{equation}\label{eq:parity and involutions}
p_{i\cdot X}(i)=p_X(i),\quad p_{i\cdot X}(j)=p_X(j)+p_X(i),\quad p_{i\cdot X}(k)=p_X(k).
\end{equation}
We also observe the following fact.

\begin{lem}
Recall that $A^m$ is the GCM with $A^m_{ii}=2-2\delta_{i,m}$.
Let $\sigma\in S_{m+1}$. Then $\sigma\cdot A^m=A^t$ for $t<m$ if and only if
$l(s_{t+1}\sigma)<l(\sigma)$.
\end{lem}

\begin{proof}
We can restate the problem as follows. To each $A^t$, we associate it's $\epsilon$-$\delta$ sequence (cf. \cite[Section 1.3.2]{CW});
that is, a word $w_t=a_1\ldots a_{m+1}$ in the letters $\set{\epsilon,\delta}$ such that
$a_{t+1}=\epsp$ and $a_s=\delta$ for all $s\neq t+1$.
Let $\mathcal{ED}$ be the set of these sequences. 
The symmetric group $S_{m+1}$ acts naturally on words in $\mathcal{ED}$ via
$\sigma\cdot a_1\ldots a_{m+1}=a_{\sigma(1)}\ldots a_{\sigma(m+1)}$;
and note that for each $w\in \mathcal{ED}$, $\sigma\cdot w\in \mathcal ED$.
In particular, we see that as $S_{m+1}$-sets, $\mathcal{ED}$ is isomorphic to $\cA$
via $w_t\leftrightarrow A^{t}$.

On the other hand, note that any $\sigma\in S_{m+1}$ can be written
in the form $\sigma=s_is_{i+1}\ldots s_m\varsigma$ where $\varsigma\in S_m$
and $s_i=(i,i+1)$ are the simple reflections.
In particular, since $S_m={\rm Stab}_{S_{m+1}}(w_m)$, observe that $\sigma(w_{m})=w_i$ with $i<m$ if and only if it has a reduced
expression with $s_{i+1}$ as a left factor;
in other words, if $\sigma(w_{m})=w_i$, then $l(s_{i+1}\sigma)<l(\sigma)$.

Translating back to $\cA$, we see that if $\sigma\cdot A^m=A^t$, then $l(s_{t+1}\sigma)<l(\sigma)$.
\end{proof}

\begin{cor}\label{cor:red words and odd roots}
Let $\sigma\in S_{m+1}$
such that $l(\sigma s_{i+1})>l(\sigma)$ for some $i\in I$.
Set $X=\sigma^{-1}\cdot A^m$. Then $I_{X,1}\neq\set{i,i+1}$.
\end{cor}
\begin{proof}
  By the previous lemma, $X\neq A^{i}$, hence the result follows by definition.
\end{proof}
Lastly, it will be helpful in several places to note that if $i$ and $j$ are connected, then
\begin{equation}\label{eq:qint to parity}
\frac{q^{-A_{ij}}-q^{A_{ij}}}{q-q^{-1}}=(-1)^{p_A(i)p_A(j)}
\end{equation}

\subsection{A family of quantum enveloping algebras}
Now for $X\in \cA$, we will now associate a quantum enveloping algebra $\UU(X)$
with generators $E_i=E_{X,i}$, $F_i=F_{X,i}$, and $K_i^{\pm 1}=K_{X,i}^{\pm 1}$ for $i\in I$,
and have parity $p=p_X$ given by $p(E_i)=p(F_i)=p_X(i)$ and $p(K_i)=0$.
(For ease of reading, we will drop the $X$ subscript when the $X$ is clear from context.)
These generators satisfy the relations
\begin{align} \label{eq:UX Ks}
&K_{i}K_{j}=K_{j}K_{i},\qquad K_{i}K_{i}^{-1}=1
&&\text{ for } i,j\in I;\\
\label{eq:UX weights}
&K_{i}E_{j}K_{i}^{-1}=q^{X_{ij}}E_{j},\quad
K_{i}F_{j}K_{i}^{-1}=q^{-X_{ij}}F_{j}
&&\text{ for } i,j\in I;\\
\label{eq:UX commutator}
&E_{i}F_{j}-(-1)^{p_X(i)p_X(j)}F_{j}E_{i}=\delta_{ij}\frac{K_{i}-K_{i}^{-1}}{q-q^{-1}},
&&\text{ for } i,j\in I\\
\label{eq:UX nilpotent}
&E_{i}^2= F_{i}^2=0
&&\text{ if } i\in I_1;\\
\label{eq:UX disc serre}
&E_{i}E_{j}=E_{j}E_{i},\quad F_{i}F_{j}=F_{j}F_{i}
&& \text{ if } i\nsim j\in I;\\
\label{eq:UX even serre E}
&E_{i}^2E_{j}-(q+q^{-1})E_{i}E_{j}E_{i}+E_{j}E_{i}^2=0
&& \text{ if } i\sim j\in I \text{ and } p(i)=0;\\
\label{eq:UX even serre F}
&F_{i}^2F_{j}-(q+q^{-1})F_{i}F_{j}F_{i}+F_{j}F_{i}^2=0
&& \text{ if } i\sim j\in I \text{ and } p(i)=0;\\
\label{eq:UX odd serre E}
&\begin{aligned}
[2]E_{i}E_{j}E_{k}E_{i}&=(-1)^{p(j)}(E_{i}E_{k}E_{i}E_{j}+E_{j}E_{i}E_{k}E_{i})\\
&\hspace{1em}+(-1)^{p(k)}(E_{k}E_{i}E_{j}E_{i}+E_{i}E_{j}E_{i}E_{k})
\end{aligned}
&& \begin{aligned}&\text{ if } i,j,k\in I\text{ such that } j\sim i,\\
&\ i\sim k, j\neq k, \text{ and }p(i)=1;
\end{aligned}\\
\label{eq:UX odd serre F}
&\begin{aligned}
[2]F_{i}F_{j}F_{k}F_{i}&=(-1)^{p(j)}(F_{i}F_{k}F_{i}F_{j}+F_{j}E_{i}F_{k}F_{i})\\
&\hspace{1em}+(-1)^{p(k)}(F_{k}F_{i}F_{j}F_{i}+F_{i}F_{j}F_{i}F_{k})
\end{aligned}
&& \begin{aligned}&\text{ if } i,j,k\in I\text{ such that } j\sim i,\\
&\ i\sim k, j\neq k, \text{ and }p(i)=1.
\end{aligned}
\end{align}

Note that $\UU(X)$ has a natural $Q$-grading given by $|E_i|=\alpha_i^X=-|F_i|$, $|K_i|=0$.
We note that $(\alpha_i^X,\alpha_i^Y)=X_{ij}$, so $K_i u K_i^{-1}=q^{(\alpha_i^X,|u|)_X} u$. 

The algebra $\UU(X)$ has several standard properties of quantum enveloping algebras.
In particular, it has a triangular decomposition 
\[\UU(X)=\Um(X)\otimes \Uz(X)\otimes \Up(X)=\Up(X)\otimes \Uz(X)\otimes \Um(X),\]
where $\Um(X)$ (resp. $\Uz(X)$, $\Up(X)$) is the subalgebra generated by the $F_i$ ($K_i$, $E_i$) for $i\in I$.
We also define $\UU^{\geq 0}(X)=\Uz(X)\Up(X)$ and $\UU^{\leq 0}=\Uz(X)\Um(X)$. Moreover, $\Um(X)$
has quantum differentials $e_i=e_{X,i}, \bar e_i=\bar e_{X,i}=\barmap\circ e_i\circ \barmap:\Um(X)\rightarrow \Um(X)$ 
for $i\in I$ satisfying
\[e_i(F_j)=\delta_{ij},\qquad e_i(xy)=e_i(x)y+(-1)^{p_X(i)p_X(x)}q^{(\alpha_i^X,|x|)}xe_i(y),\]
\[E_{i}x-(-1)^{p(i)p(x)}xE_i=\frac{K_i\bar e_i(x)-K_i^{-1}e_i(X)}{q-q^{-1}}\]

We also observe that there are a number of important automorphisms of $\UU(X)$.
Indeed, we define maps $\tau,\barmap$ satisfying
\[\tau(E_i)=E_i,\quad \tau(F_i)=F_i,\quad \tau(K_i)=(-1)^{p(i)}K_i^{-1},\quad \tau(xy)=\tau(y)\tau(x).\]
\[\rho(E_i)=F_i,\quad \rho(F_i)=E_i,\quad \rho(K_i)=K_i,\quad \rho(xy)=\rho(y)\rho(x).\]
\[\bar{E_i}=E_i,\quad \bar{F_i}=F_i,\quad \bar{K_i}=K_i^{-1},\quad \bar q=q^{-1}\quad \bar{xy}=\bar x\ \bar y.\]
Equipped with this family of algebras, we now define analogues of Lusztig's braid operators.

\begin{thm}\label{thm:braids}
	Let $i\in I$ and $X\in\cA$, $e=\pm 1$, and set $Y=i\cdot X$.
	There exists a $\Qq$-linear algebra isomorphisms 
	$T_{i,e}', T''_{i,e}: \UU(X)\rightarrow \UU(Y)$ satisfying
	\[T_{i,e}'(E_{X,j})=\begin{cases}-(-1)^{p_Y(i)}K_{Y,i}^{-e}F_{Y,i}&\text{ if }j=i;\\
	E_{Y,j}E_{Y,i}-(-1)^{p_Y(i)p_Y(j)}q^{eY_{ij}}E_{Y,i}E_{Y,j}&\text{ if }j \sim i;\\
	E_{Y,j}&\text{otherwise;}
	\end{cases}\]
	\[T_{i,e}'(F_{X,j})=\begin{cases}-(-1)^{p_Y(i)}E_{Y,i}K^e_{Y,i}&\text{ if }j=i;\\
	F_{Y,i}F_{Y,j}-(-1)^{p_Y(i)p_Y(j)}q^{-eY_{ij}} F_{Y,j}F_{Y,i}&\text{ if }j \sim i;\\
	F_{Y,j}&\text{otherwise;}
	\end{cases}\]
	\[T_{i,e}'(K_{X,j})=\begin{cases}(-1)^{p_Y(i)}K_{Y,i}^{-1}&\text{ if }j=i;\\
	(-1)^{p_Y(i)p_Y(j)}K_{Y,i}K_{Y,j}&\text{ if }j \sim i;\\
	K_{Y,j}&\text{otherwise;}
	\end{cases}\]
	and
	\[T_{i,e}''(E_{X,j})=
	\begin{cases}-F_{Y,i}K_{Y,i}^{e}&\text{ if }j=i;\\
	E_{Y,i}E_{Y,j}-(-1)^{p_Y(i)p_Y(j)}q^{eY_{ij}}E_{Y,j}E_{Y,i}&\text{ if }j  \sim i;\\
	E_{Y,j}&\text{otherwise;}
	\end{cases}\]
	\[T_{i,e}''(F_{X,j})=
	\begin{cases}-K^{-e}_{Y,i}E_{Y,i}&\text{ if }j=i;\\
	F_{Y,j}F_{Y,i}-(-1)^{p_Y(i)p_Y(j)}q^{-eY_{ij}} F_{Y,i}F_{Y,j}&\text{ if }j \sim i;\\
	F_{Y,j}&\text{otherwise;}
	\end{cases}\]
	\[T_{i,e}''(K_{X,j})=
	\begin{cases}(-1)^{p_Y(i)}K_{Y,i}^{-1}&\text{ if }j=i;\\
	(-1)^{p_Y(i)p_Y(j)}K_{Y,i}K_{Y,j}&\text{ if }j \sim i;\\
	K_{Y,j}&\text{otherwise;}
	\end{cases}\]
\end{thm}

\begin{rmk}\label{rmk:braidthm}
	Let us make two observations about this theorem.
	\begin{enumerate}
		\item Note that the maps $T_{i,e}'$ and $T_{i,e}''$ are parity-preserving. In fact, they are also weight-preserving
		in the sense of Section \ref{subsec:perspective}. This will be convenient to keep in mind for transporting coefficients 
		between the algebras associated to different $X$'s.
		\item We note that Theorem \ref{thm:braids} and the proof can be generalized to
		$\gl(m|n)$; 
		one only has to check a few additional cases. However, we don't need this generality,
		and as the proof for our special case is already quite involved,
		we do not do this here.
	\end{enumerate}
\end{rmk}

In the remainder of this subsection, we will prove Theorem \ref{thm:braids} and then verify
that the maps satisfy the braid relations in Lemmas \ref{lem:disc braid rel} and \ref{lem:conn braid rel}.
To do so succinctly, first observe that
\begin{equation}\label{eq:braids and invols}
T_{i,e}'=\tau T_{i,e}''\tau,\qquad T_{i,-e}'=\barmap\ T_{i,e}'\ \barmap=\rho T_{i,e}' \rho,\qquad T_{i,e}''=(T_{i,e}')^{-1};
\end{equation}
it is easy to verify these identities on the generators (for the $j$ connected
to $i$ case, see Lemma \ref{lem:root transpositions}).\
In particular, to prove the desired results in the case of $T_{i,1}'$,
as all the remaining cases follow immediately from \eqref{eq:braids and invols}

We need to verify that the images of the generators of $\UU(X)$ 
under $T_{i,1}'$ satisfy the relations 
\eqref{eq:UX Ks}-\eqref{eq:UX odd serre F}.
In the case of \eqref{eq:UX Ks}, this is trivial.
The relation \eqref{eq:UX weights} essentially follows from the construction,
but in any case it is easy to verify directly.
In order to keep the proof of the remaining relations digestible,
we will break up the verification into lemmas.

Since the calculations
get quite involved, we will drop the subscripts on the generators of $\UU(Y)$ 
for the sake of readability. 
First, we check the relation \eqref{eq:UX nilpotent}.
	
\begin{lem}
If $i,j\in I$ such that $p_X(j)=1$, then 
\[T_{i,1}'(E_{X,j})^2=T_{i,1}'(F_{X,j})^2=0.\]
\end{lem}
\begin{proof}	
	This is trivial if $i$ is not connected to $j$, so suppose $i$ is connected to $j$.
	In this case, \eqref{eq:parity and involutions} implies 
	$p(j)=1+p_X(i)=1+p(i)$ hence $p(i)p(j)=0$ and $Y_{ij}=-1$. Then
	$T_{i,1}'(E_{X,j})=E_{j}E_{i}-q^{Y_{ij}}E_{i}E_{j}$, and so
	\begin{align*}
	T_{i,1}'(E_{X,j})^2&=(E_{j}E_{i}-q^{-1}E_{i}E_{j})^2\\
	&=E_{j}E_{i}E_{j}E_{i}-q^{-1}E_{j}E_{i}^2E_{j}-q^{-1}E_{i}E_{j}^2E_{i}+q^{-2}E_{i}E_{j}E_{i}E_{j}.
	\end{align*}
	Now observe that either $E_{j}^2=0$ or $E_{i}^2=0$; let's assume $E_{i}^2=0$ as the
	computation is similar in the other case.
	Then using the identity $E_{i}E_{j}E_{i}E_{j}=E_{i}E_{j}^{(2)}E_{i}$
	(which follows from the Serre relation for $i,j$)
	\begin{align*}
	T_{i,1}'(E_{X,j})^2&=E_{i}E_{j}^{(2)}E_{i}-q^{-1}(q+q^{-1})E_{i}E_{j}^{(2)}E_{i}+q^{-2}E_{i}E_{j}^{(2)}E_{i}=0.
	\end{align*}
	The statement for $F_{X,j}$ is proved similarly.
\end{proof}

Next, we verify the commutation relation \eqref{eq:UX commutator} in two steps.
\begin{lem}
If $i,j,k\in I$ with $j\neq k$, then
\[T_{i,1}'(E_{X,j})T_{i,1}'(F_{X,k})=(-1)^{p_X(j)p_X(k)}T_{i,1}'(F_{X,k})T_{i,1}'(E_{X,j}).\]
\end{lem}	
\begin{proof}
	Let $c_{jk}=T_{i,1}'(E_{X,j})T_{i,1}'(F_{X,k})-(-1)^{p_X(j)p_X(k)}T_{i,1}'(F_{X,k})T_{i,1}'(E_{X,j})$.
	We want to prove that $c_{jk}=0$ for all $j\neq k$.
	First, observe that if either $j$ or $k$ is not connected to $i$, then
	the statement is trivially true; indeed, if $j$ is not connected to $i$,
	then $T_{i,1}'(E_{X,j})=E_j$ and $p(j)=p_X(j)$. On the other hand, 
	$T_{i,1}'(F_{X,k})$ is a polynomial in the elements $K_i$, $E_i$, $F_k$, and $F_i$
	with $p(T_{i,1}'(F_{X,k}))=p_X(k)$. Since $E_j$ supercommutes with all of these
	elements, the statement in this case follows.
	
	The remaining cases are when $j$ and $k$ are both either equal to or connected to $i$.
		
	First, suppose $k$ is connected to $j=i$.
	Then we have
	\begin{align*}
	c_{ik}&=(-K_{i}^{-1}F_{i})(F_{i}F_{k}-(-1)^{p(i)p(k)}q^{-Y_{ik}}F_{k}F_{i})\\
	&\hspace{1em}-(-1)^{p_X(i)p_X(k)}(F_{i}F_{k}-(-1)^{p(i)p(k)}q^{-Y_{ik}}F_{j}F_{i})(-K_{i}^{-1}F_{i})\\
	&=-K_{i}^{-1}(F_{i}^2F_{k}-(-1)^{p(i)p(k)}q^{-Y_{ik}}F_{i}F_{k}F_{i}\\
	&\hspace{1em}-(-1)^{p(i)p(k)+p(i)}q^{-Y_{ik}-Y_{ii}}F_{i}F_{k}F_{i}+(-1)^{p(i)}q^{-2Y_{ik}-Y_{ii}}F_{k}F_{i}^2
	\end{align*}
	Now there are two subcases. If $p_X(i)=p(i)=1$, then $Y_{ii}=0$ and $F_{Y,i}^2=0$ whence
	\[c_{ik}=(-1)^{p(k)}K_{i}^{-1}(q^{-Y_{ik}}-q^{-Y_{ii}-Y_{ik}})F_{i}F_{k}F_{i}=0.\]
	If $p_X(i)=p(i)=0$, then $Y_{ii}=2$ and $Y_{ik}=-1$, so we see that
	\[c_{ik}=K_{i}^{-1}(F_{i}^2F_{k}-[2]F_{i}F_{k}F_{i}+F_{k}F_{i}^2)=0.\]
	We note the case $j$ is connected to $k=i$ is entirely similar.
	
	Now suppose $j$ and $k$ are connected to $i$. Note that then $p_X(j)p_X(k)=p(j)p(k)=0$;
	moreover, if $p(i)=1$, then either $p(j)=1$ or $p(k)=1$, hence
	$p(i)p(j)+p(i)p(k)=p(i)$.
	Then we have
	\begin{align*}
	c_{ik}
	&=(E_{j}E_{i}-(-1)^{p(i)p(j)}q^{Y_{ij}}E_{i}E_{j})(F_{i}F_{k}-(-1)^{p(i)p(k)}q^{-Y_{ik}}F_{k}F_{i})\\
	&\hspace{1em}-(F_{i}F_{k}-(-1)^{p(i)p(k)}q^{-Y_{ik}}F_{k}F_{i})(E_{j}E_{i}-(-1)^{p(i)p(j)}q^{Y_{ij}}E_{i}E_{j})\\
	&=E_j(E_iF_i-(-1)^{p(i)}F_iE_i)F_k-q^{Y_{ij}}(E_iF_i-(-1)^{p(i)}F_iE_i)E_jF_k\\
	&\hspace{1em}-q^{-Y_{ik}}E_jF_k(E_iF_i-(-1)^{p(i)}F_iE_i)+q^{Y_{ij}-Y_{ik}}F_k(E_iF_i-(-1)^{p(i)}F_iE_i)E_j\\
	&=\frac{1}{q-q^{-1}}E_j((1-q^{2Y_{ij}}-1+q^{2Y_{ij}})K_i-(1-1-q^{-2Y_{ik}}+q^{-2Y_{ik}})K_i^{-1})F_k =0.
	\end{align*}
	
This finishes the proof.
\end{proof}

Next, we check \eqref{eq:UX commutator} in the case $j=k$.

\begin{lem}
If $i,j\in I$, then
\[T_{i,1}'(E_{X,j})T_{i,1}'(F_{X,j})-(-1)^{p_X(j)}T_{i,1}'(F_{X,j})T_{i,1}'(E_{X,j})
=\frac{T_{i,1}'(K_{X,j})-T_{i,1}'(K_{X,j})^{-1}}{q-q^{-1}}.\]
\end{lem}	
\begin{proof}
	Observe that the statement is trivially true for $j$ not connected or equal to $i$,
	and if $j=i$ then the statement is easy to verify.
	Therefore, let us assume that $j$ is connected to $i$.
	
	Let $c_{jj}=T_{i,1}'(E_{X,j})T_{i,1}'(F_{X,j})-(-1)^{p_X(j)p_X(j)}T_{i,1}'(F_{X,j})T_{i,1}'(E_{X,j})$.
	Then the statement follows by verifying that $c_{jj}=(-1)^{p(i)p(j)}\frac{K_iK_j-K_i^{-1}K_j^{-1}}{q-q^{-1}}$.
	To that end, we first do some preliminary computations.
	First, recall that $p_X(j)=p_Y(i)+p_Y(j)$. Then
	\begin{align*}
	z_1&=E_{j}E_{i}F_{i}F_{j}-(-1)^{p_X(j)}F_{i}F_{j}E_{j}E_{i}\\
	&=E_{j}F_{j}\frac{q^{-Y_{ij}}K_{i}-q^{Y_{ij}}K_{i}^{-1}}{q-q^{-1}}+(-1)^{p_Y(i)}F_{i}E_{i}
	\frac{q^{Y_{ij}}K_{j}-q^{-Y_{ij}}K_{j}^{-1}}{q-q^{-1}}
	\end{align*}
	\begin{align*}
	z_2&=E_{j}E_{i}F_{j}F_{i}-(-1)^{p_X(j)}F_{j}F_{i}E_{j}E_{i}\\
	&=(-1)^{p_Y(i)p_Y(j)}E_{j}F_{j}\frac{K_{i}-K_{i}^{-1}}{q-q^{-1}}+(-1)^{p_Y(i)+p_Y(i)p_Y(j)}F_{i}E_{i}
	\frac{K_{j}-K_{j}^{-1}}{q-q^{-1}}
	\end{align*}
	\begin{align*}
	z_3&=E_{i}E_{j}F_{i}F_{j}-(-1)^{p_X(j)}F_{i}F_{j}E_{i}E_{j}\\
	&=(-1)^{p_Y(i)p_Y(j)}E_{i}F_{i}\frac{K_{j}-K_{j}^{-1}}{q-q^{-1}}+(-1)^{p_Y(j)+p_Y(i)p_Y(j)}F_{j}E_{j}
	\frac{K_{i}-K_{i}^{-1}}{q-q^{-1}}
	\end{align*}
	\begin{align*}
	z_4&=E_{i}E_{j}F_{j}F_{i}-(-1)^{p_X(j)}F_{j}F_{i}E_{i}E_{j}\\
	&=E_{i}F_{i}\frac{q^{-Y_{ij}}K_{j}-q^{Y_{ij}}K_{j}^{-1}}{q-q^{-1}}+(-1)^{p_Y(j)}F_{j}E_{j}
	\frac{q^{Y_{ij}}K_{i}-q^{-Y_{ij}}K_{i}^{-1}}{q-q^{-1}}
	\end{align*}
	Now observe that
	\begin{align*}
		c_{jj}&=z_1-(-1)^{p_Y(i)p_Y(j)}q^{-Y_{ij}}z_2-(-1)^{p_Y(i)p_Y(j)}q^{Y_{ij}}z_3+z_4\\
		&=E_{j}F_{j}\frac{q^{-Y_{ij}}-q^{Y_{ij}}}{q-q^{-1}}K_{i}^{-1}
		-(-1)^{p_Y(j)}F_{j}E_{j}\frac{q^{-Y_{ij}}-q^{Y_{ij}}}{q-q^{-1}}K_{i}^{-1}\\
		&+E_{i}F_{i}\frac{q^{-Y_{ij}}-q^{Y_{ij}}}{q-q^{-1}}K_{j}
		-(-1)^{p_Y(i)}F_{i}E_{i}\frac{q^{-Y_{ij}}-q^{Y_{ij}}}{q-q^{-1}}K_{j}\\
		&=(-1)^{p_Y(i)p_Y(j)}\frac{K_{j}K_{i}-K_{j}^{-1}K_{i}^{-1}}{q-q^{-1}}.
	\end{align*}
	This finishes the proof.
\end{proof}
	
Now we need to verify the various Serre relations.
The next lemma checks the relation \eqref{eq:UX disc serre}.

\begin{lem}
If $i,j,k\in I$ such that $j\nsim k$, then
\[T_{i,1}'(E_{X,j})T_{i,1}'(E_{X,k})=T_{i,1}'(E_{X,k})T_{i,1}'(E_{X,j}).\]
\end{lem}
\begin{proof} Again, as the $E$ and $F$ cases
	are similar, we will prove this for $E$. We need to verify that $T_{i,1}'(E_{X,j})T_{i,1}'(E_{X,k})=T_{i,1}'(E_{X,k})T_{i,1}'(E_{X,j})$
	when $k\nsim j$. Well, if neither $j$ nor $k$ is connected to $i$, this is trivial,
	so suppose $j\sim i\sim k$.
	Moreover, observe that if $p_X(i)=0$, then one or less of $j,k$ are odd, 
	and so the calculation is formally identical to the same identity for
	$U_q(\fsl(4))$; therefore, we may as well assume $p_X(i)=1$.
	Then either $p_X(j)=0$ or $p_X(k)=0$, and as the roles of $j$ and $k$ are symmetric,
	we can assume $p_X(j)=0$ without loss of generality.
	
	Then note that $p_X(k)=1$, and so $p(j)=1=p(i)$ and $p(k)=0$. 
	Therefore we see that it suffices to show that
	$T_{i,1}'(E_{X,j})=E_{j}E_i+qE_{i}E_{j}$ and
	$T_{i,1}'(E_{X,k})=E_{k}E_i-q^{-1}E_{i}E_{k}$ commute. 
	Well,
	\[(E_{j}E_i+qE_{i}E_{j})(E_{k}E_i-q^{-1}E_{i}E_{k})
	=E_{j}E_iE_{k}E_i+qE_iE_{j}E_{k}E_i-E_{i}E_{j}E_{i}E_{k}\]
	\[(E_{k}E_i-q^{-1}E_{i}E_{k})(E_{j}E_i+qE_{i}E_{j})
	=E_{k}E_iE_{j}E_i-q^{-1}E_iE_{j}E_{k}E_i-E_{i}E_{k}E_{i}E_{j}\]
	The difference of these two equations is zero by \eqref{eq:UX odd serre E}.
	
\end{proof}	

Next up are \eqref{eq:UX even serre E} and \eqref{eq:UX even serre F}.
\begin{lem}
If $i,j,k\in I$ such that $j\sim k$ and $p_X(j)=0$. Then
\[T_{i,1}'(E_{X,j})^{2}T_{i,1}'(E_{X,k})-[2]T_{i,1}'(E_{X,j})T_{i,1}'(E_{X,k})T_{i,1}'(E_{X,j})+T_{i,1}'(E_{X,k})T_{i,1}'(E_{X,j})^{2}=0,\]
\[T_{i,1}'(F_{X,j})^{2}T_{i,1}'(F_{X,k})-[2]T_{i,1}'(F_{X,j})T_{i,1}'(F_{X,k})T_{i,1}'(F_{X,j})+T_{i,1}'(F_{X,k})T_{i,1}'(F_{X,j})^{2}=0.\]
\end{lem}

\begin{proof}
	Again, due to similarity of the arguments, we will prove this only for the $E$'s.
	If neither $j$ nor $k$ is connected to $i$, this is trivial.
	Otherwise, we must check case by case.
	Suppose first that $j\sim k\sim i$ and $j\neq i$. Then observe that  
	$T_{i,1}'(E_{X,j})=E_j$ and $T_{i,1}'(E_{X,k})=E_{k}E_{i}-(-1)^{p_Y(i)p_Y(k)}q^{Y_{k,i}}E_{i}E_k$.
	Then the Serre relation follows by observing that $E_j$ and $E_{i}$ commute,
	so if $S_{jk}=E_j^2E_k-[2]E_jE_kE_j+E_{k}E_j^2=0$, then
	\begin{align*}
	T_{i,1}'(E_j)^{2}T_{i,1}'(E_k)&-[2]T_{i,1}'(E_j)T_{i,1}'(E_k)T_{i,1}'(E_j)+T_{i,1}'(E_k)T_{i,1}'(E_j)^{2}\\
	&=S_{jk}E_{i}-(-1)^{p_Y(i)p_Y(k)}q^{Y_{k,k+1}}E_{i}S_{jk}=0
	\end{align*}
	
	The remaining cases are when we have $k=i$ or $j=i$. It will be convenient to
	observe the following commutation relation: if $\ell\sim i$, then 
	using \eqref{eq:UX commutator} and \eqref{eq:qint to parity},
	we observe that
	\begin{align*}
	T_{i,1}'(E_{X,\ell})T_{i,1}'(E_{X,i})
	&=q^{Y_{ii}+Y_{i\ell}}(-1)^{p(i)+p(i)p(\ell)}T_{i,1}'(E_{X,i})T_{i,1}'(E_{X,\ell})-(-1)^{p(i)+p(i)p(\ell)}q^{Y_{ii}+Y_{i\ell}}E_{\ell}.
	\end{align*}
	
	First, suppose $k=i$.
	Note that either $p_X(i)=0$ and hence $p(j)=p(i)=0$, or
	$p_X(i)=1=p(i)=p(j)$. The first case is a $\UU_q(\fsl(3))$ calculation, which is known. 
	Therefore we may assume $p(i)=p(j)=1$, which implies that $Y_{ii}=0$, and $Y_{ij}=1$.
	Note that
	\begin{align*}
	T_{i,1}'(E_{X,j})E_j&=q^{-1}E_j T_{i,1}'(E_{X,j})
	\end{align*}
	Then we compute that 
	\begin{align*}
	T_{i,1}'(E_{X,j})^2T_{i,1}'(E_{X,i})
	&=qT_{i,1}'(E_{X,j})T_{i,1}'(E_{X,i})T_{i,1}'(E_{X,j})
	-E_{j}T_{i,1}'(E_{X,j})\\
	&=q^{2}T_{i,1}'(E_{X,i})T_{i,1}'(E_{X,j})^2
	-(q^{2}+1)E_{j}T_{i,1}'(E_{X,j})\\
	\end{align*}
	\begin{align*}
	T_{i,1}'(E_{X,j})T_{i,1}'(E_{X,i})T_{i,1}'(E_{X,j})
	&=qT_{i,1}'(E_{X,i})T_{i,1}'(E_{X,j})^2-qE_{j}T_{i,1}'(E_{X,j}).
	\end{align*}
	It is now easy to verify that \[T_{i,1}'(E_{X,j})^2T_{i,1}'(E_{X,i})-(q+q^{-1})T_{i,1}'(E_{X,j})T_{i,1}'(E_{X,i})T_{i,1}'(E_{X,j})+T_{i,1}'(E_{X,i})T_{i,1}'(E_{X,j})^2=0.\]
	
	The last case is when $j=i$; in particular, note that
	$p(i)=0$ hence $Y_{ii}=2$ and $Y_{ik}=-1$.
	Now we observe that
	\begin{align*}
	T_{i,1}'(E_{X,i})E_k&=qE_k T_{i,1}'(E_{X,i})
	\end{align*}
	In this case,
	\begin{align*}
	T_{i,1}'(E_{X,k})T_{i,1}'(E_{X,i})^2
	&=qT_{i,1}'(E_{X,i})T_{i,1}'(E_{X,k})T_{i,1}'(E_{X,i})
	-T_{i,1}'(E_{X,i})E_{k}\\
	&=q^{2}T_{i,1}'(E_{X,i})^2T_{i,1}'(E_{X,k})
	-(q^{2}+1)T_{i,1}'(E_{X,i})E_{k}\\
	\end{align*}
	\begin{align*}
	T_{i,1}'(E_{X,i})T_{i,1}'(E_{X,k})T_{i,1}'(E_{X,i})
	&=qT_{i,1}'(E_{X,i})^2T_{i,1}'(E_{X,j})-qT_{i,1}'(E_{X,i})E_{k}.
	\end{align*}
	Then we see that
	\[T_{i,1}'(E_{X,i})^2T_{i,1}'(E_{X,k})-(q+q^{-1})T_{i,1}'(E_{X,i})T_{i,1}'(E_{X,k})T_{i,1}'(E_{X,i})
	+T_{i,1}'(E_{X,k})T_{i,1}'(E_{X,i})^2=0.\]

\end{proof}	

Finally, we must verify \eqref{eq:UX odd serre E} and \eqref{eq:UX odd serre F}.
As these relations are so long to state, it will help to introduce the following notation.
For $t\in \set{0,1}$,
let $S_t(x_1,x_2, x_3)\in \Qq\ang{x_1,x_2,x_3}$ be the polynomial in three non-commuting
variables given by 
\[S_t(x_1,x_2,x_3)=[2]x_2x_1x_3x_2-(-1)^t(x_2x_3x_2x_1+x_1x_2x_3x_2
-x_2x_1x_2x_3-x_3x_2x_1x_2).\]
Then for instance, we see that \eqref{eq:UX odd serre E} is the statement that
$S_{p(i)}(E_{i},E_j,E_{k})=0$ for $i\sim j\sim k$ with $i\neq k$.
\begin{lem}
Let $i,j,k,\ell\in I$ with $j\sim k\sim \ell$, $j\neq \ell$, and $p_X(k)=1$. Then
\[S_{p_X(j)}(T_{i,1}'(E_{X,j}),T_{i,1}'(E_{X,k}),T_{i,1}'(E_{X,\ell}))=
S_{p_X(j)}(T_{i,1}'(F_{X,j}),T_{i,1}'(F_{X,k}),T_{i,1}'(F_{X,\ell}))=0.\]
\end{lem}
\begin{proof}
Again, the $E$ and $F$ cases are similar, so we will only prove the $E$ case.
We also note that necessarily $p(j)p(\ell)=0$.
Furthermore, if none of $j,k,\ell$ are connected or equal to $i$, then
this is trivial to verify.

First suppose $k$ is not connected or equal to $i$,
but $j$ or $\ell$ is connected to $i$; without loss of generality, we can assume $j$ is connected
to $i$. Since $p_X(k)=p(k)=1$, in this case we necessarily have $p(i)=0$ and $Y_{ij}=-1$.
Then $T_{i,1}'(E_{k,X})=E_k$, $T_{i,1}'(E_\ell)=E_\ell$, $T_{i,1}'(E_{X,j})=E_jE_i-q^{-1}E_iE_j$.
Since $E_i$ commutes with $E_k,E_\ell$, we see that
\[S_{p(j)}(T_{i,1}'(E_{X,j}),T_{i,1}'(E_{X,k}),T_{i,1}'(E_{X,\ell}))
=S_{p(j)}(E_{j},E_k,E_{\ell})E_i-q^{-1}E_iS_{p(j)}(E_{j},E_k,E_{\ell})=0.\]

Next, suppose $k\sim i$ and assume wlog that $j=i$.
Observe that since $p(k)=1+p(i)$, $p(i)p(k)=0$ and $Y_{ik}=-1$.
Then $T_{i,1}'(E_{k,X})=E_kE_i-q^{-1}E_iE_k$, $T_{i,1}'(E_\ell)=E_\ell$, 
and $T_{i,1}'(E_{X,i})=-(-1)^{p(i)}K_i^{-1}F_i$.
Note the identities
\[T_{i,1}'(E_{X,i})T_{i,1}'(E_{X,\ell})=T_{i,1}'(E_{X,\ell})T_{i,1}'(E_{X,i}),\]
\begin{align*}
T_{i,1}'(E_{X,k})T_{i,1}'(E_{X,i})&=(-1)^{p(i)}q^{Y_{ii}-1}T_{i,1}'(E_{X,i})T_{i,1}'(E_{X,k})
-(-1)^{p(i)}q^{Y_{ii}-1}E_{j}.
\end{align*}

Now let $x_{abcd}=T_{i,1}'(E_{X,a})T_{i,1}'(E_{X,b})T_{i,1}'(E_{X,c})T_{i,1}'(E_{X,d})$.
Then
\[x_{k\ell ki}
=(-1)^{p(i)}q^{Y_{ii}-1}x_{ki\ell k}-(-1)^{p(i)}q^{Y_{ii}-1}T_{i,1}'(E_{X,k})E_{\ell}E_{k}.\]
\[x_{ik\ell k}
=(-1)^{p(i)}q^{1-Y_{ii}}x_{ki\ell k}+E_{j}E_{\ell}T_{i,1}'(E_{X,k}).\]
\[x_{kik\ell}
=T_{i,1}'(E_{X,k})E_kE_\ell.\]
\[x_{\ell kik}
=E_{\ell}T_{i,1}'(E_{X,j})E_j\]
We want to show that 
$x_{k\ell ki}+x_{ik\ell k}-x_{kik\ell}-x_{\ell kik}=(-1)^{p(i)}[2]x_{ki\ell k}$.
Now the calculation proceeds slightly differently depending on $p(i)$,
so we split into subcases.
Suppose first $p(i)=0$, and so $p(k)=1$, $Y_{ii}=2$, and $Y_{ik}=-1$.
Then we have

\[x_{k\ell ki}
=qx_{ki\ell k}-qE_kE_iE_{\ell}E_{k}+E_iE_kE_{\ell}E_k.\]
\[x_{ij\ell j}
=q^{-1}x_{ki\ell k}+E_{k}E_{\ell}E_kE_i-q^{-1}E_kE_iE_\ell E_k.\]
\[x_{kik\ell}
=E_kE_iE_kE_\ell.\]
\[x_{\ell kik}
=E_{\ell}E_kE_iE_k\]
In particular, we see that
\[x_{k\ell ki}+x_{ik\ell k}-x_{kik\ell}-x_{\ell kik}=[2]x_{ki\ell k}-S_{p(i)}(E_i, E_k, E_\ell)=[2]x_{ki\ell k}.\]

Now suppose that $p(i)=1$, in which case $p(k)=0$, $Y_{ii}=0$, and $Y_{ik}=-1$.
Then we have
\begin{align*}
x_{k\ell ki}
&=-q^{-1}x_{ki\ell k}+q^{-1}E_kE_iE_{\ell}E_{k}-q^{-2}E_iE_kE_{\ell}E_k.\\
x_{ik\ell k}
&=-qx_{ki\ell k}+E_{k}E_{\ell}E_kE_i-q^{-1}E_kE_iE_\ell E_k.\\
x_{kik\ell}
&=E_kE_iE_kE_\ell-q^{-1}E_iE_k^2E_\ell.\\
x_{\ell kik}
&=E_{\ell}E_kE_iE_k-q^{-1}E_iE_\ell E_k^2
\end{align*}
Now observe that using the Serre relation \eqref{eq:UX even serre E} repeatedly,
\[E_{k}E_{\ell}E_kE_i-E_\ell E_kE_iE_k-E_{k}E_{i}E_kE_\ell=-E_iE_k^{(2)}E_\ell -E_iE_\ell E_k^{(2)}=-E_iE_kE_\ell E_k;\]
Then in particular,
\[x_{k\ell ki}+x_{ik\ell k}-x_{kik\ell}-x_{\ell kik}
=-[2]x_{ki\ell k}-q^{-1}E_i(E_k^{2}E_\ell-[2]E_kE_\ell E_k+E_\ell E_k^{2})=-[2]x_{ki\ell k}.\]

Finally, suppose $k=i$. We can assume without loss of generality that 
$p(j)=0=p_X(\ell)$ and $p(\ell)=1=p_X(j)$, hence
$Y_{ii}=0$, $Y_{ij}=-1$, and $Y_{i\ell}=1$.
Then 
$T_{i,1}'(E_{X,j})=E_jE_i-q^{-1}E_iE_j$, 
$T_{i,1}'(E_\ell)=E_\ell E_i+qE_iE_\ell$, 
and $T_{i,1}'(E_{X,i})=K_i^{-1}F_i$, 
and we have the identities
\begin{align*}
T_{i,1}'(E_{X,j})T_{i,1}'(E_{X,i})&=-q^{-1}T_{i,1}'(E_{X,i})T_{i,1}'(E_{X,j})
+q^{-1}E_{j}.
\end{align*}
\begin{align*}
T_{i,1}'(E_{X,\ell})T_{i,1}'(E_{X,i})&=qT_{i,1}'(E_{X,i})T_{i,1}'(E_{X,\ell})
-qE_{\ell}.
\end{align*}
\begin{align*}
E_lT_{i,1}'(E_{X,i})&=qT_{i,1}'(E_{X,i})E_l
\end{align*}
\begin{align*}
E_jT_{i,1}'(E_{X,i})&=q^{-1}T_{i,1}'(E_{X,i})E_j
\end{align*}
Again, let $x_{abcd}=T_{i,1}'(E_{X,a})T_{i,1}'(E_{X,b})T_{i,1}'(E_{X,c})T_{i,1}'(E_{X,d})$.
Then
\begin{align*}
x_{i\ell ij}
&=K_i^{-1}F_i(-qE_jE_\ell E_i+ E_\ell E_iE_j),\\
x_{ji\ell i}
&=-qT_{i,1}'(E_{X,j})T_{i,1}'(E_{X,i})E_\ell=K_i^{-1}F_i(E_jE_iE_\ell-q^{-1}E_iE_jE_\ell)-E_jE_\ell,\\
x_{iji\ell}
&=K_i^{-1}F_i(q^{-1}E_jE_lE_i+E_jE_iE_\ell),\\
x_{\ell iji}
&=q^{-1}T_{i,1}'(E_{X,\ell})T_{i,1}'(E_{X,i})E_j=K_i^{-1}F_i(E_\ell E_i E_j+qE_i E_j E_\ell)-E_jE_\ell,\\
x_{ij\ell i}
&=qx_{iji\ell}-qK_i^{-1}F_i(E_jE_i-q^{-1}E_iE_j)E_\ell=K_i^{-1}F_i(E_jE_\ell E_i+E_iE_jE_\ell).
\end{align*}
Then we compute that
\begin{align*}
x_{i\ell ij}+x_{ji\ell i}-x_{iji\ell}-x_{\ell iji}&=K_i^{-1}F_i((-q-q^{-1})E_jE_\ell E_i+(-q-q^{-1}E_iE_jE_\ell)\\
&=-[2]x_{ji\ell j}.
\end{align*}
This finishes the proof
\end{proof}

We have now finished the proof that the maps $T_{i,e}'$ and $T_{i,e}''$ are algebra
isomorphisms for all $i\in I$ and $e=\pm 1$. Finally, we will show
that they satisfy the braid relations of type $A$.

\begin{lem}\label{lem:disc braid rel}
Let $i\nsim j\in I$, and $X\in\cA$.
Let $Y=i\cdot X$, $Y'=j\cdot X$, and $Z=j\cdot Y=i\cdot Y'$.
Then as maps $\UU(X)\rightarrow \UU(Z)$, $T_{i,e}'T_{j,e}'=T_{j,e}'T_{i,e}'$
and $T_{i,e}''T_{j,e}''=T_{j,e}''T_{i,e}''$.
\end{lem}

\begin{proof}
Again, it suffices to prove this for $T_{i,1}'$, as the other cases follow from this one.
To verify that $T_{i,1}'T_{j,1}'(u)=T_{j,1}'T_{i,1}'(u)$ for all $u\in \UU(X)$,
it suffices to check in the cases $u=E_{X,k}, F_{X,k}, K_{X,k}$ for some $k\in I$.

Let $k\in I$. Note that if $k$ is not connected to $i$ or $j$,
then the statement is obvious. Suppose that $k$ is connected or equal to exactly
one of $i$ or $j$; without loss of generality, let's suppose
$k$ is connected or equal to $j$. Then for $L\in\set{E,F,K}$,
note that $T_{j,1}'(L_k)$ is a polynomial in $K_k^{\pm 1}$, $E_k$, $F_k$, $E_j$, and $F_j$,
all of which are ``fixed'' (i.e. mapped to the corresponding generator) 
by $T_{i,1}'$, hence $T_{i,1}'T_{j,1}'(L_k)=T_{j,1}'(L_k)=T_{j,1}'T_{i,1}'(L_k)$.

Finally, suppose $k$ is connected to both $i$ and $j$.
Then 
\begin{align*}
T_{i,1}'T_{j,1}'(E_k)&=T_{i,1}'(E_kE_j-(-1)^{p_{Y'}(j)p_{Y'}(k)}q^{Y'_{jk}}E_jE_k)\\
&=E_kE_iE_j-(-1)^{p_Z(i)p_Z(k)}q^{Z_{ik}}E_iE_kE_j-(-1)^{p_{Y'}(j)p_{Y'}(k)}q^{Y'_{jk}}E_jE_kE_i\\
&+(-1)^{p_{Y'}(j)p_{Y'}(k)+p_Z(i)p_Z(k)}q^{Y'_{jk}+Z_{ik}}E_iE_jE_k
\end{align*}
\begin{align*}
T_{j,1}'T_{i,1}'(E_k)&=T_{j,1}'(E_kE_i-(-1)^{p_{Y}(i)p_{Y}(k)}q^{Y_{ik}}E_iE_k)\\
&=E_kE_iE_j-(-1)^{p_Z(j)p_Z(k)}q^{Z_{jk}}E_jE_kE_i-(-1)^{p_{Y}(i)p_{Y}(k)}q^{Y_{ik}}E_iE_kE_j\\
&+(-1)^{p_{Y}(i)p_{Y}(k)+p_Z(j)p_Z(k)}q^{Y_{ik}+Z_{jk}}E_iE_jE_k
\end{align*}
Now, observe that $p_Y(i)=p_Z(i)$, $p_Y(k)=p_Z(k)+p_Z(j)$, and $p_Z(i)p_Z(j)=0$. 
Then one finds that $p_{Y}(i)p_{Y}(k)=p_{Z}(i)p_{Z}(k)$ and hence $Y_{ik}=Z_{ik}$.
Replacing $i$ with $j$ and $Y$ with $Y'$, we find $p_{Y'}(j)p_{Y'}(k)=p_Z(j)p_Z(k)$
and $Y'_{jk}=Z_{jk}$, and thus $T_{i,1}'T_{j,1}'(E_k)=T_{j,1}'T_{i,1}'(E_k)$. Similar arguments
prove that $T_{i,1}'T_{j,1}'(F_k)=T_{j,1}'T_{i,1}'(F_k)$ and $T_{i,1}'T_{j,1}'(K_k)=T_{j,1}'T_{i,1}'(K_k)$.
\end{proof}

\begin{lem}\label{lem:root transpositions}
	Let $i\sim j\in I$. 
	Then \[T_{i,e}'T_{j,e}'(E_{X,i})=T_{i,e}''T_{j,e}''(E_{X,i})=E_{Z,j},\] \[T_{i,e}'T_{j,e}'(F_{X,i})=T_{i,e}''T_{j,e}''(F_{X,i})=F_{Z,j},\text{ and}\]  \[T_{i,e}'T_{j,e}'(K_{X,i})=T_{i,e}''T_{j,e}''(K_{X,i})=K_{Z,j}.\]
\end{lem}

\begin{proof}
	We have $T_{j,1}'(E_{X,i})=E_{Y,i}E_{Y,j}-(-1)^{p_Y(j)p_Y(i)}q^{Y_{ij}}E_{Y,j}E_{Y,i}$, so
	\begin{align*}
		T_{i,1}'T_{j,1}'(E_{X,i})&=(-(-1)^{p_Z(i)}K_{Z,i}^{-1}F_{Z,i})(E_{Z,j}E_{Z,i}-(-1)^{p_Z(i)p_Z(j)}q^{Z_{ij}}E_{Z,i}E_{Z,j})
		\\
		&\hspace{1em}+(-1)^{p_Z(i)p_Z(j)}q^{-Z_{ii}-Z_{ij}}(E_{Z,j}E_{Z,i}-(-1)^{p_Z(i)p_Z(j)}q^{Z_{ij}}E_{Z,i}E_{Z,j})(K_{Z,i}^{-1}F_{Z,i})\\
		&=(-1)^{p_Z(i)p_Z(j)}K_{Z,i}^{-1}(E_{Z,j}(E_{Z,i}F_{Z,i}-(-1)^{p_Z(i)}F_{Z,i}E_{Z,i})\\
		&\hspace{1em}-q^{Z_{ij}}(E_{Z,i}F_{Z,i}-(-1)^{p(i)}F_{Z,i}E_{Z,i})E_{Z,j})\\
		&=(-1)^{p_Z(i)p_Z(j)}K_{Z,i}^{-1}(\frac{q^{-Z_{ij}}K_{Z,i}-q^{Z_{ij}}K_{Z,i}^{-1}}{q-q^{-1}}
		-q^{Z_{ij}}\frac{K_{Z,i}-K_{Z,i}^{-1}}{q-q^{-1}})E_{Z,j}\\
		&=(-1)^{p_Z(i)p_Z(j)}\frac{q^{-Z_{ij}}-q^{Z_{ij}}
		}{q-q^{-1}}E_{Z,j}=E_{Z,j}
	\end{align*}
	Similarly, we have $T_{j,1}'(F_{X,i})=F_{Y,j}F_{Y,i}-(-1)^{p_Y(j)p_Y(i)}q^{-Y_{ij}}F_{Y,i}E_{F,j}$, so
	\begin{align*}
		T_{i,1}'T_{j,1}'(F_{X,i})&=(F_{Z,i}F_{Z,j}-(-1)^{p_Z(i)p_Z(j)}q^{-Z_{ij}}F_{Z,j}F_{Z,i})(-(-1)^{p_X(i)}E_{Z,i}K_{Z,i})
		\\
		&\hspace{1em}+(-1)^{p_Z(i)p_Z(j)}q^{Z_{ii}+Z_{ij}}(E_{Z,i}K_{Z,i})(F_{Z,i}F_{Z,j}-(-1)^{p_Z(i)p_Z(j)}q^{-Z_{ij}}F_{Z,j}F_{Z,i})\\
		&=(-1)^{p_Z(i)p_Z(j)}((E_{Z,i}F_{Z,i}-(-1)^{p_Z(i)}F_{Z,i}E_{Z,i})F_{Z,j}\\
		&\hspace{1em}-q^{-Z_{ij}}F_{Z,j}(E_{Z,i}F_{Z,i}-(-1)^{p(i)}F_{Z,i}E_{Z,i}))K_{Z,i}\\
		&=(-1)^{p_Z(i)p_Z(j)}F_{Z,j}(\frac{q^{-Z_{ij}}K_{Z,i}-q^{Z_{ij}}K_{Z,i}^{-1}}{q-q^{-1}}
		-q^{-Z_{ij}}\frac{K_{Z,i}-K_{Z,i}^{-1}}{q-q^{-1}})K_{Z,i}\\
		&=(-1)^{p_Z(i)p_Z(j)}\frac{q^{-Z_{ij}}-q^{Z_{ij}}
		}{q-q^{-1}}F_{Z,j}=F_{Z,j}
	\end{align*}
	Finally, note that we have $T_{i,1}'T_{j,1}'(K_{X,i})=(-1)^{p_{Y}(i)p_Y(j)+p_Z(i)+p_Z(i)p_Z(j)}K_{Y,j}$,
	and observe that $p_{Y}(i)p_Y(j)=p_Z(i)+p_Z(i)p_Z(j)$ since $Y=i\cdot Z$.
\end{proof}

\begin{lem}\label{lem:conn braid rel}
	Let $i,j\in I$ with $j=i\pm 1$ and $W\in\cA$, and set $X=i\cdot W$ and $X'=j\cdot W$,
	$Y=j\cdot X$ and $Y'=i\cdot X'$, and $Z=i\cdot Y=j\cdot Y'$.
	Then as maps $\UU(W) \rightarrow \UU(Z)$, we have
	$T_{i,e}'T_{j,e}'T_{i,e}'=T_{j,e}'T_{i,e}'T_{j,e}'$
	and $T_{i,e}''T_{j,e}''T_{i,e}''=T_{j,e}''T_{i,e}''T_{j,e}''$.
\end{lem}

\begin{proof}
	Note that it suffices to check this on the generators.
	To reduce the clutter of notation in the proof, we will drop the $W,X,Y,Z$ subscripts
	on the generators whenever the ambient space of the elements is clear from context.
	We will also freely use the identities given by Lemma \ref{lem:root transpositions}
	in the following computations.
	
	First observe that if $k$ is not connected to $i$ or $j$, 
	$T_{i,1}'T_{j,1}'T_{i,1}'$ and $T_{j,1}'T_{i,1}'T_{j,1}'$ map $E_k\mapsto E_k$,
	$F_k\mapsto F_k$, and $K_k\mapsto K_k$ hence they agree on these generators.
	
	Now suppose that $k\notin\set{i,j}$ but is connected to $i$ or $j$;
	without loss of generality, assume $k$ is connected to $j$ but not $i$.
	Then using \eqref{eq:parity and involutions}
	we note the coincidences
	\begin{equation*}
	\tag{$\star$}
	\begin{aligned}
	p_{X'}(j)p_{X'}(k)&=p_Z(i)p_Z(j)+p_Z(i)p_Z(k)=p_Z(i)p_Z(j),\\
	p_{Y}(j)p_{Y}(k)&=p_Z(i)p_Z(k)+p_Z(j)p_Z(k)=p_Z(j)p_Z(k).
	\end{aligned}
	\end{equation*}
	Using \eqref{eq:parity and matrix entries}, this immediately implies
	\[\tag{$\star\star$}X'_{jk}=Z_{ij},\quad Y_{jk}=Z_{jk}.\]
	First, let us compare the images of $K_k$. We compute that
	\[T_{i,1}'T_{j,1}'T_{i,1}'(K_k)=T_{i,1}'T_{j,1}'(K_k)=T_{i,1}'((-1)^{p_{Y}(j)p_Y(k)}K_jK_k)
	=(-1)^{p_{Z}(i)p_Z(j)+p_{Y}(j)p_Y(k)}K_iK_jK_k,\]
	\[T_{j,1}'T_{i,1}'T_{j,1}'(K_k)=(-1)^{p_{X'}(j)p_{X'}(k)}T_{j,1}'T_{i,1}'(K_j)T_{j,1}'T_{i,1}'(K_k)
	=(-1)^{p_{X'}(j)p_{X'}(k)+p_{Z}(j)p_{Z}(k)}K_iK_jK_k.\]
	Comparing in view of ($\star$), we see that
	$T_{i,1}'T_{j,1}'T_{i,1}'(K_k)=T_{j,1}'T_{i,1}'T_{j,1}'(K_k)$.
	Next, we consider the images of $E_k$. We compute
	\begin{align*}
	T_{i,1}'T_{j,1}'T_{i,1}'(E_k)&=T_{i,1}'T_{j,1}'(E_k)\\
	&=E_k(E_jE_i-(-1)^{p_{Z}(i)p_{Z}(j)}q^{Z_{ij}}E_iE_j)\\
	&\hspace{2em}-(-1)^{p_{Y}(j)p_{Y}(k)}q^{Y_{jk}}(E_jE_i-q^{Z_{ij}}(-1)^{p_{Z}(i)p_{Z}(j)}E_iE_j)E_k\\
	&=E_kE_jE_i-(-1)^{p_{Z}(i)p_{Z}(j)}q^{Z_{ij}}E_iE_kE_j-(-1)^{p_{Y}(j)p_{Y}(k)}q^{Y_{jk}}E_jE_iE_k\\
	&\hspace{2em}
	+(-1)^{p_{Y}(j)p_{Y}(k)+p_{Z}(i)p_{Z}(j)}q^{Y_{jk}+Z_{ij}}E_iE_jE_k.
	\end{align*} 
	On the other hand,
	\begin{align*}
	T_{j,1}'T_{i,1}'T_{j,1}'(E_k)&=T_{j,1}'T_{i,1}'(E_k)T_{j,1}'T_{i,1}'(E_j)-q^{X'_{jk}}(-1)^{p_{X'}(j)p_{X'}(k)}T_{j,1}'T_{i,1}'(E_j)T_{j,1}'T_{i,1}'(E_k)\\
	&=T_{j,1}'(E_k)E_i-q^{X'_{jk}}(-1)^{p_{X'}(j)p_{X'}(k)}E_iT_{i,1}'(E_k)\\
	&=(E_kE_j-(-1)^{p_Z(j)p_Z(k)}q^{Z_{jk}}E_jE_k)E_i\\
	&\hspace{2em}-q^{X'_{jk}}(-1)^{p_{X'}(j)p_{X'}(k)}E_i(E_kE_j-(-1)^{p_Z(j)p_Z(k)}q^{Z_{jk}}E_jE_k)\\
	&=E_kE_jE_i-(-1)^{p_Z(j)p_Z(k)}q^{Z_{jk}}E_jE_iE_k-(-1)^{p_{X'}(j)p_{X'}(k)}q^{X'_{jk}}E_iE_kE_j\\
	&\hspace{2em}+(-1)^{p_{X'}(j)p_{X'}(k)+p_{Z}(j)p_{Z}(k)}q^{X'_{jk}+Z_{jk}}E_iE_jE_k
	\end{align*}
	Comparing coefficients in view of the above coincidences, we see
	that \[T_{i,1}'T_{j,1}'T_{i,1}'(E_k)=T_{j,1}'T_{i,1}'T_{j,1}'(E_k).\] A similar computation
	proves 
	\[T_{i,1}'T_{j,1}'T_{i,1}'(F_k)=T_{j,1}'T_{i,1}'T_{j,1}'(F_k).\]
	
	Finally, the last case is for $k=i,j$.
	We will prove the case $k=i$, as the $k=j$ case follows from reversing
	the roles of $i$ and $j$ in the following arguments.
	First, note that
	\[T_{i,1}'T_{j,1}'T_{i,1}'(K_i)=(-1)^{p_X(i)}T_{i,1}'T_{j,1}'(K_i^{-1})=(-1)^{p_X(i)}K_j^{-1}.\]
	\[T_{j,1}'T_{i,1}'T_{j,1}'(K_i)=T_{j,1}'(K_j)=(-1)^{p_Z(j)}K_j^{-1}.\]
	Note that \eqref{eq:parity and involutions} implies $p_X(i)=p_Y(i)+p_Y(j)=p_Z(j)$,
	thus $T_{i,1}'T_{j,1}'T_{i,1}'(K_i)=T_{j,1}'T_{i,1}'T_{j,1}'(K_i)$.
	Next, let us verify that $T_{i,1}'T_{j,1}'T_{i,1}'(E_i)= T_{j,1}'T_{i,1}'T_{j,1}'(E_i)$.
	Well, by the previous lemma, \[T_{i,1}'T_{j,1}'(E_i)=E_{j},\text{ hence }
	T_{j,1}'T_{i,1}'T_{j,1}'(E_i)=-(-1)^{p_{Z}(j)}K_{j}^{-1}F_{j}.\] 
	On the other hand, $T_{i,1}'(E_i)=-(-1)^{p_{X}(i)}K_i^{-1}F_i$, and so
	\[T_{i,1}'T_{j,1}'T_{i,1}'(E_i)=-T_{i,1}'T_{j,1}'(K_i^{-1})T_{i,1}'T_{j,1}'(F_i)=-(-1)^{p_{X}(i)}K_j^{-1}F_j.\]
	Then as before, $p_X(i)=p_Z(j)$ hence $T_{j,1}'T_{i,1}'T_{j,1}'(E_i)=T_{i,1}'T_{j,1}'T_{i,1}'(E_i)$. 
	Proving the identity $T_{i,1}'T_{j,1}'T_{i,1}'(F_i)= T_{j,1}'T_{i,1}'T_{j,1}'(F_i)$ proceeds similarly.
\end{proof}

\subsection{Constructing PBW bases}

Now we turn to the problem of constructing PBW bases using our braid operators.
To that end, let us now fix $T_i=T_{i,1}''$ and $T_{i}^{-1}=T_{i,1}'$.

Let $\omega_0\in S_{m+1}$ be the longest element and let $\bI$ be the set of
$\bi=(i_1,\ldots, i_N)\in I^N$ such that
$\omega_0=s_{i_1}\ldots s_{i_N}$ is a reduced expression.
There is an associated convex order on $\Phi^+(X)=\set{\beta_1<\ldots<\beta_N}$
where $\beta_t=s_{i_1}\ldots s_{i_{t-1}}(\alpha_{i_t}^X)$
(here, the action is as in the classical case, 
as described in Remark \ref{rem:symm group action}).
In particular, let $X_{\bi;t}=i_{t-1}\cdot \ldots\cdot i_2\cdot i_1\cdot X$ (so $X_{\bi;1}=X$).
Define the root vectors $F_{\bi; \beta_t}= T_{i_1}\ldots T_{i_{t-1}}(F_{X_{\bi;t},i_t})$.
The we define the set of PBW vectors (relative to $\bf i$) to be
\begin{equation}
\bB_\bi=\set{F_{\bi; \beta_1}^{(a_1)}\ldots F_{\bi;\beta_N}^{(a_N)}\mid a_1,\ldots, a_n
\in\N, a_s<2\text{ if }p(\beta_s)=1}.
\end{equation}

\begin{lem}\label{lem:braid moves and root vectors}
Let $\bi=(i_1,\ldots, i_N)\in \bI$.
\begin{enumerate}
\item Suppose there is an $1< t\leq N$ such that $i_{t-1}$ and $i_{t}$ are not connected.
Let $\bj=(j_1,\ldots, j_N)$ where $j_s=i_s$ for $s\neq t-1,t$, $j_t=i_{t-1}$ and $j_{t-1}=i_t$.
Let $\sigma=(t-1,t)\in S_{N+1}$.
Then $\bj\in \bI$, and if $F_{\bi;\beta_s}=F_{\bj;\beta_{\sigma(s)}}$ for $1\leq s\leq N$,
and thus $F_{\bi;\beta}=F_{\bj;\beta}$ for all $\beta\in \Phi^+$.
Moreover, $F_{\bi;\beta_{t-1}}F_{\bi;\beta_{t}}=F_{\bi;\beta_{t}}F_{\bi;\beta_{t+1}}$
so $\bB_\bi=\bB_\bj$.
\item Suppose there is an $1<t<N$ such that $i_t=i$ and $i_{t+1}=i_{t-1}=j$ connected to $i$.
Let $\bj=(j_1,\ldots, j_N)$ where $j_s=i_s$ for $s\neq t-1,\ t,\ t+1$, $j_t=j$ and $j_{t\pm 1}=i$.
Let $\sigma=(t-1,t+1)\in S_{N+1}$.
Then $\bj\in \bI$ and for $s\neq t$, $F_{\bi;\beta_s}=F_{\bj;\beta_{\sigma(s)}}$ and thus
$F_{\bi;\beta}=F_{\bj;\beta}$ for all $\beta\neq\beta_t$.
Furthermore, $\beta_{t-1}+\beta_{t+1}=\beta_t$.
and setting $Z=X_{\bi;t-1}$, we have
\[F_{\bi;\beta_t}=F_{\bi;\beta_{t+1}}F_{\bi;\beta_{t-1}}-(-1)^{p(\beta_{t-1})p(\beta_{t+1})}q^{(\beta_{t-1},\beta_{t+1})}F_{\bi;\beta_{t-1}}F_{\bi;\beta_{t+1}},\]
\item For any $i\in I$, $F_{\bi;\alpha_i^X}=F_i$.
\item For any $\beta\in \Phi^+$, $F_\beta\in \UU^-(X)$, the subalgebra
generated by the $F_{i}$.
\end{enumerate}
\end{lem}

In the cases of (1) and (2) of the lemma, we will say $\bi$ and $\bj$ are braid-connected.
It is well-known that in any Coxeter group, any reduced expression can be obtained
from a given one using braid moves. In particular, note that between
any $\bi,\bj\in \bI$, there is a sequence $\bi_0=\bi,\bi_1,\ldots, \bi_t=\bj$ of $\bi_s\in \bI$
such that $\bi_{s}$ and $\bi_{s+1}$ are braid connected.

\begin{proof}

For (1) and (2), it is well known that $\bj\in \bI$,
and the fact that the braid operators respect the braid relations
proves that the root vectors are the same for most roots. The remaining
statements are easily proved by applying $T_{i_{t-2}}^{-1}\ldots T_{i_1}^{-1}$
and observing that the claims follow from elementary rank 2 calculations.
In particular, the coefficient $(-1)^{p(\beta_{t-1})p(\beta_{t+1})}q^{(\beta_{t-1},\beta_{t+1})}$
can be expressed thusly from the observation that the maps $T_i$ are parity-preserving
and ``weight-preserving'' in the sense of Remark \ref{rmk:braidthm} and Section \ref{subsec:perspective}.

In particular, note that (2) implies that if $\bi$ and $\bj$ are braid
connected, then $F_{\bi;\alpha_i}=F_{\bj;\alpha_i}$ since $\alpha_i$
cannot be written as a sum of positive roots.
Then (3) immediately follows from (2), as there is a $\bj\in \bI$ 
with $\bj=(i,j_2,\ldots, j_N)$ which is connected to $\bi$ by a sequence of braid moves,
and $F_{\bi;\alpha_i}=F_{\bj;\alpha_i}=F_i$ by definition.

Finally, note that (4) follows by induction on the height of $\beta$ exactly
as in \cite[Lemma 3.2]{Tin}
\end{proof}

\begin{cor} 
The set $\bB_\bi$ is a basis of $\Um(X)$.
\end{cor}

\begin{proof}
This follows essentially the same proof as \cite[Lemma 3.4]{Tin}.
However, note that we essentially prove by simultaneous induction for the entire
family of algebras $\set{\UU(X)}_{X\in\cA}$.
\end{proof}

\begin{lem}
Let $\bi=(i_1,\ldots i_N)\in \bI$.
Let $1\leq r<s\leq N$. Then $T_{i_s}^{-1}T_{i_{s-1}}^{-1}\ldots T_{i_1}^{-1}(F_{\beta_r})\in \UU^{\geq 0}(X_{\bt;s+1})$.
\end{lem}
\begin{proof}
	Let $Y=X_{\bt;s+1}$.
	Note that $T_{i_r}^{-1}T_{i_{r-1}}^{-1}\ldots T_{i_1}^{-1}(F_{\beta_r})=-(-1)^{p(i_r)}E_{i_r}K_{i_r}=-\omega(F_{i_r})K_{i_r}$, where
	$\omega=\barmap\circ\rho\circ\tau$. Then
	\[T_{i_{s}}^{-1}\ldots T_{i_{r+1}}^{-1}\omega(F_i)=\omega(T_{i_s}\ldots T_{i_{r+1}}(F_{i_r})
	=\omega(F_{\bj; \beta_{s-r}}),\]
	where $\bj=(i_s,\ldots, i_1, i_N,\ldots, i_{s+1})\in \bI$;
	In particular, since Lemma \ref{lem:braid moves and root vectors}
	(4) says $F_{\bj; \beta_{s-r}}\in\Um(Y)$,
	it follows that $\sigma(F_{\bj; \beta_{s-r}})\in\Up(Y)$.
	Then the result follows by observing that $T_i^{\pm 1}(\Uz(Z))=\Uz(i\cdot Z)$ for all $i\in I$ 
	and $Z\in\cA$, and because the $T_i^{\pm 1}$ are algebra isomorphisms.
\end{proof}

\newcommand{\ba}{\mathbf{a}}
\newcommand{\bk}{\mathbf{k}}
\begin{cor} \label{cor:levsoib}
Let $\bi\in \bI$ and $1\leq r<s\leq N$.
 Write $F_{\bi;\beta_s}F_{\bi;\beta_r}=\sum c_\ba F_\bi^\ba$,
where $\ba=(a_1,\ldots, a_N)\in \N^N$ with $a_s<2$ if $p(\beta_s)=1$,
$c_\ba\in \Qq$,
and $F_\bi^\ba=F_{\bi;\beta_1}^{(a_1)}\ldots F_{\bi;\beta_N}^{(a_N)}$.
Then $c_{\ba}=0$ unless $a_t=0$ for $t<r$ or $t>s$.
Moreover, the coefficient of $F_{\bi;\beta_r}F_{\bi;\beta_s}$ is $(-1)^{p(\beta_r)p(\beta_s)}q^{-(\beta_r,\beta_s)}$.
\end{cor}

\begin{proof}
This follows from a similar proof as in \cite[Lemma 3.5]{Tin};
namely, observe that from the proof of the previous lemma,
\[\tag{a}T_{i_t}^{-1}\ldots T_{i_1}^{-1}(F_\bi^\ba)=
\omega(F_{\bj;\beta_{t}}^{(a_1)}\ldots F_{\bj;\beta_1}^{(a_t)})f_{a_1,\ldots, a_{t}}(K_i;i\in I)F_{\bk;\beta_1}^{(a_{t+1})}\ldots F_{\bk;\beta_{N-t}}^{(N)},\]
where $f_{a_1,\ldots, a_{t}}$ is some polynomial in $|I|$ variables
with $f_{0,\ldots, 0}=1$
Furthermore, observe that the
$\omega(F_{\bj;\beta_{t}}^{(a_1)}\ldots F_{\bj;\beta_1}^{(a_t)})$
are linearly independent vectors in $\Up(X_{\bi;t+1})$.

By Lemma \ref{lem:braid moves and root vectors}(4), 
\[\tag{b}x=T_{i_{r-1}}^{-1}\ldots T_{i_1}^{-1}(F_{\bi;\beta_s}F_{\bi;\beta_r})\in \Um(X_{\bi;r}).\]
On the other hand, applying (a) to the right-hand side of $F_{\bi;\beta_s}F_{\bi;\beta_r}=\sum c_\ba F_\bi^\ba$ and using the triangular decomposition $\UU(X_{\bi;r})\cong\Up(X_{\bi;r})\otimes \Uz(X_{\bi;r})\otimes \Um(X_{\bi;r})$, 
we see that
\[\tag{c}x=\sum c_\ba
\omega(F_{\bj;\beta_1}^{(a_t)}\ldots F_{\bj;\beta_{t}}^{(a_1)})\otimes f_{a_1,\ldots, a_{r-1}}(K_i;i\in I)\otimes F_{\bk;\beta_1}^{(a_{t+1})}\ldots F_{\bk;\beta_{N-t}}^{(N)}.\]
In particular, the only way we can have $c_\ba\neq 0$ given (b) and
(c) is if $a_1=\ldots= a_{r-1}=0$. A similar argument proves
that $c_\ba\neq 0$ only if $a_{s+1}=\ldots=a_{N}=0$.

Lastly, let $c$ be the coefficient of $F_{\bi;\beta_r}F_{\bi;\beta_s}$. 
Then by the conditions on when $c_\ba\neq 0$, 
$z=T_{i_{r}}^{-1}\ldots T_{i_1}^{-1}(F_{\bi;\beta_s}F_{\bi;\beta_r}-cF_{\bi;\beta_r}F_{\bi;\beta_s})\in\Um(X_{\bi;r+1})$.
On the other hand, we have that $y=T_{i_{r}}^{-1}\ldots T_{i_1}^{-1}(F_{\bi;\beta_s})\in \Um(X_{\bi;r+1})$
and $T_{i_{r}}^{-1}\ldots T_{i_1}^{-1}(F_{\bi;\beta_r})=-K_{i_r}^{-1} E_{i_r}$, so
\begin{align*}\tag{d}
z
=cK_{i_r}^{-1} E_{i_r}y-yK_{i_r}^{-1} E_{i_r}
&=c\frac{\bar e_{i_r}(y)-K_i^{-2}e_{i_r}(y)}{q-q^{-1}}+((-1)^{p(i_r)p(y)}c-q^{(i_r,|y|)})K_{i_r}^{-1}y E_{i_r}.
\end{align*}
In particular, applying the triangular decomposition and the fact that $z\in \Um(X_{\bi;r+1})$, we see that
$e_{i_r}(y)=0$ and $c=(-1)^{p(i_r)p(y)}q^{(i_r,|y|)}$. Since the maps $T_i$ are parity- and weight-preserving
the result follows.
\end{proof}
Note that we extract the following corollary from the above argument.

\begin{cor}\label{cor:deriv and root}
	If $\bi=(i_1,\ldots, i_N)\in \bI$, then $e_{i_1}'(F_{\bi;\beta_r})=0$ for all $r>1$.
\end{cor}

Lastly, let us note that the PBW bases constructed here agree with those constructed via quantum shuffles.

\begin{lem}
Let $\bi\in I$ and let $<=<_\bi$ be the associated total order on $I$ induced
by the convex order on $\Phi^+$. Let $F_{<;\beta}$ be the root vector
defined with respect to $<$ as in \cite{CHW3}.
Then $F_{<;\beta}=F_{\bi;\beta}$.
\end{lem}

\begin{proof}
This follows from a similar proof to that in \cite{Lec}. To wit, first note that
$F_{<;\alpha_i^X}=F_i=F_{\bi;\alpha_i^X}$.  More generally, 
suppose $\beta=\beta_r+\beta_s$ such that $\beta_r$ is maximal, and assume that by induction on the height,  $F_{\bi;\beta_r}=F_{<;\beta_r}$ and 
$F_{\bi;\beta_s}=F_{<;\beta_s}$. Then on one hand, letting $l_1=l(\beta_r)$ and $l_2=l(\beta_s)$ be the Lyndon
words associated to the roots, as in \cite[Theorem 28]{Lec} , we deduce that $l_1l_2=l(\beta)$ is the costandard factorization,
and thus by definition (cf. \cite[Proposition 4.11 and (5.1)]{CHW3}) \[F_{<;\beta}=F_{<;\beta_s}F_{<;\beta_r}-(-1)^{p(\beta_r)p(\beta_s)}q^{-(\beta_r,\beta_s)}F_{<;\beta_r}F_{<;\beta_s}.\]
On the other hand, Corollary \ref{cor:levsoib} together with the maximality of $\beta_r$ implies that 
\begin{align*}F_{\bi;\beta_r+\beta_s}&=F_{\bi;\beta_s}F_{\bi;\beta_r}-(-1)^{p(\beta_r)p(\beta_s)}q^{-(\beta_r,\beta_s)}F_{\bi;\beta_r}F_{\bi;\beta_s}\\
&=F_{<;\beta_s}F_{<;\beta_r}-(-1)^{p(\beta_r)p(\beta_s)}q^{-(\beta_r,\beta_s)}F_{<;\beta_r}F_{<;\beta_s}=F_{<;\beta}.
\end{align*} 
\end{proof}

\subsection{Comparison to the lattice}
Now let us consider the case of $X=A^m$, the unique GCM with $X_{mm}=0$,
which corresponds to the standard Dynkin diagram for $\gl(m|1)$.
In this case, $\UU(X)$ is a subalgebra of $\UU$ as defined in \S \ref{sec:gl(m|1)},
but $\Um(X)=\Um$.

\begin{lem}\label{lem:same q span}
Let $\bi,\bj\in \bI$ be connected by a sequence of braid moves. Then
$\Z[q]\bB_\bi=\Z[q]\bB_\bj$.
Moreover, for any $b_{1}\in \bB_\bi$, there is a $b_2\in \bB_\bj$ 
such that $b_1\equiv b_2$ modulo $q\Z[q]\bB_{\bi}$.
\end{lem}
\begin{proof}
	It suffices to prove this in the case that $\bi$ and $\bj$ differ by a single braid move.
	If it is a braid move of the form $i,j\mapsto j,i$, then there is nothing to say
	by Lemma \ref{lem:braid moves and root vectors} (1). Suppose it is a braid move of the form
	$i_{t-1}i_ti_{t+1}\mapsto i_{t+1}i_ti_{t-1}$, where $i_{t\pm1}=i$ and $i_t=j$;
	Then \ref{lem:braid moves and root vectors} (2) implies
	that the only root vectors that change are the ones corresponding to the braid move,
	and applying $T_{i_{t-2}}^{-1}\ldots T_{i_{1}}^{-1}$ 
	it suffices to prove this in the rank 2 case; that is, to prove that in 
	$\UU(\sigma^{-1}\cdot A^m)$ where $\sigma=s_{i_1}\ldots s_{i_{t-2}}$, we have
	\[\tag{$\star$}\sum \Z[q]F_i^{(a)}T_i(F_j)^{(b)}T_iT_j(F_j)^{(c)}
	=\sum \Z[q]F_j^{(a)}T_j(F_i)^{(b)}T_jT_i(F_j)^{(c)}.\]
	Well, first observe that since $\bi=(i_1,\ldots, i_{t-2},i,j,i,\ldots)$ 
	and $\bj=(i_1,\ldots, i_{t-2},j,i,j,\ldots)$ 
	correspond to a reduced expressions
	of the longest element of $S_{m+1}$, $l(\sigma s_i)=l(\sigma s_j)>l(\sigma)$.
	Then by Corollary \ref{cor:red words and odd roots},
	at least one of $i$ or $j$ must be even, so without loss of generality
	we can assume that $i$ is even. If $j$ is also even, this follows
	from \cite[42.1.5]{LuBook}. Otherwise, this follows easily from \cite[\S 8]{CHW3};
	indeed, if we take the PBW basis with the opposite ordering, it is easy to see that
	it's $\Z[q]$ lattice agrees with the $\Z[q]$ lattice of the canonical basis therein
	by applying the anti-involution $\tau$ .
\end{proof}

\begin{thm}\label{thm:CBs equal}
	For any $\bi\in \bI$, $\bB_\bi+q\cL=B(\infty)$, thus in particular
	the canonical bases defined in \cite{CHW3} coincide and are equal to $\cB$.
\end{thm}
\begin{proof}
	It suffices to show that $\bB_\bi+q\cL$ is closed under the action of the $\tf_i$,
	since $1+q\cL\in \bB_{\bi}+q\cL$. For this, fix $i\in I$ and note that $i$ appears
	in $\bi$. If $\bi=(i_1,\ldots, i_N)$ and $i_1=i$, then by Corollary \ref{cor:deriv and root}
	it follows that $\tf_i F_\bi^{(a_1,\ldots, a_N)}=F_\bi^{(a_1+1,\ldots, a_N)}$.
	Otherwise, there is some $\bj=(j_1,\ldots, j_N)$ connected to $\bi$ by a sequence of braid moves
	such that  $j_1=i$. Then for any $b_1\in \bB_\bi$, there is a $b_2\in \bB_\bj$ 
	such that $b_1\equiv b_2\in \bB_\bj$, and thus $\tf_i b_1\equiv \tf_ib_2\in \bB_i+q\cL$.
\end{proof}

Note that the anti-involution $\tau$ maps any PBW basis to the PBW basis associated to the
opposite ordering on simple roots, hence in particular we have the following 
straightforward consequence of the theorem.

\begin{cor}
The canonical basis $\cB(\infty)$ is invariant under the anti-involution $\tau$.
\end{cor}

\section{Examples}\label{sec:examples}

\subsection{Canonical bases for standard quantum $\gl(2|1)$}\label{subsec:CB(2|1)}
Let us consider our construction in the special case of $U_q^-(\gl(2|1))$.
First, we should compare the crystal lattice constructed here to the one in \cite{Zou}.
Therein, the author constructs a partial crystal structure on $U_q^-(\gl(m|1))$.
To do this, essentially they use the even Kashiwara operators to construct
a lattice by descending from the nilpotent PBW elements.
However, while this lattice is closed under $\te_m$, it is not closed
under $\tf_m$. (An upper crystal basis is also constructed which is closed under $\tf_m$, but not $\te_m$.)
We observe that the lattice defined therein when $m=2$ case is 
$(\bigoplus_t \cA F_1^{(t)})\oplus (\bigoplus_t \cA F_1^{(t)}F_2)\oplus (\bigoplus_t q\cA F_2F_1^{(t)})
\oplus (\bigoplus_t \cA F_2F_1^{(t)}F_2)$.

Now let us relate our crystal basis to the canonical 
basis from \cite{CHW3}. This canonical basis is constructed
from the PBW bases, as described in the following proposition.
\begin{prop}    \label{prp:CBU+}
$\Um$ admits the following canonical basis:
$$
b_{1^r}=F_1^{(r)}, \quad b_{1^{r}2}=F_1^{(r)}F_2, \quad b_{21^{r}}=F_2F_1^{(r)}, \quad
b_{21^{r}2}=F_2F_1^{(r+1)} F_2  \qquad (\forall
r\ge 0).
$$
Here, the subscripts are words in the alphabet $\set{1,2}$ and for $r\in \Z$, we set $1^r=\underbrace{1\ldots 1}_r$.
\end{prop}
Let us now consider how this canonical basis relates to the crystals constructed in this section.
Computing the action of the Kashiwara operators on the canonical basis, one easily
sees that the canonical basis lies in the crystal lattice (hence spans it). Indeed, we compute
that for $r\geq 0$, the action of the Kashiwara operators on the canonical basis elements of height $r$
are given by
\begin{align*}
\tf_1 b_{1^r}&=b_{1^{r+1}};&
\tf_2 b_{1^r}&=b_{21^r};\\
\te_1 b_{1^r}&=b_{1^{r-1}};& 
\te_2 b_{1^{r}}&=0;\\
\tf_1 b_{1^{r-1}2}&=b_{1^{r}2}; &
\tf_2 b_{1^{r-1}2}&=b_{21^{r-1}2};\\
\te_1 b_{1^{r-1}2}&=b_{1^{r-2}2}; &
\te_2 b_{1^{r-1}2}&=q^{r-1}b_{1^{r-1}};\\
\tf_1 b_{21^{r-1}}&=b_{21^{r}}+(q^{r-1}-q^r)b_{1^{r}2};&
\tf_2 b_{21^{r-1}}&=0;\\
\te_1 b_{21^{r-1}}&=b_{21^{r-2}}+(q^r-q^{r-1})b_{1^{r-2}2};&
\te_2 b_{21^{r-1}}&=b_{1^{r-1}};\\
\tf_1b_{21^{r-2}2}&=b_{21^{r-1}2};&
\tf_2 b_{21^{r-2}2}&=0;\\
\te_1 b_{21^{r-2}2}&=b_{21^{r-3}2};&
\te_2 b_{21^{r-2}2}&=b_{1^{r-2}2}-q^{r-2}b_{21^{r-2}}.
\end{align*}

\begin{figure}
\begin{tikzpicture}[scale=.7]
\draw (0,0) node{$b_{\emptyset}$};
\draw (-1,-2) node{$b_{1}$};
\draw (1,-2) node{$b_{2}$};
\draw (-2,-4) node{$b_{11}$};
\draw (0,-4) node{$b_{21}$};
\draw (2,-4) node{$b_{12}$};
\draw (-3,-6) node{$b_{111}$};
\draw (-1,-6) node{$b_{211}$};
\draw (1,-6) node{$b_{112}$};
\draw (3,-6) node{$b_{212}$};
\draw (-3,-8) node{$b_{1111}$};
\draw (-1,-8) node{$b_{2111}$};
\draw (1,-8) node{$b_{1112}$};
\draw (3,-8) node{$b_{2112}$};
\draw (-3,-10) node{$b_{11111}$};
\draw (-1,-10) node{$b_{21111}$};
\draw (1,-10) node{$b_{11112}$};
\draw (3,-10) node{$b_{21112}$};

\draw[blue,->, thick] (-.3,-.3) -- (-.9,-1.6);
\draw[red, dashed, ->, thick] (.2,-.3) -- (.8,-1.6);

\draw[blue,->, thick] (-1.3,-2.3) -- (-1.9,-3.6);
\draw[red, dashed, ->, thick] (-.8,-2.3) -- (-.2,-3.6);
\draw[blue,->, thick] (1.3,-2.3) -- (1.9,-3.6);

\draw[blue,->, thick] (-2.3,-4.3) -- (-2.9,-5.6);
\draw[red, dashed, ->, thick] (-1.8,-4.3) -- (-1.2,-5.6);
\draw[blue,->, thick] (-.3,-4.3) -- (-.9,-5.6);
\draw[blue,->, thick] (1.7,-4.3) -- (1.3,-5.6);
\draw[red, dashed, ->, thick] (2.3,-4.3) -- (2.9,-5.6);

\draw[blue,->, thick] (-3,-6.3) -- (-3,-7.7);
\draw[red, dashed, ->, thick] (-2.8,-6.3) -- (-1.4,-7.6);
\draw[blue,->, thick] (-1,-6.3) -- (-1,-7.7);
\draw[blue,->, thick] (1,-6.3) -- (1,-7.7);
\draw[red, dashed, ->, thick] (1.4,-6.3) -- (2.8,-7.6);
\draw[blue,->, thick] (3,-6.3) -- (3,-7.7);

\draw[blue,->, thick] (-3,-8.3) -- (-3,-9.7);
\draw[red, dashed, ->, thick] (-2.8,-8.3) -- (-1.4,-9.6);
\draw[blue,->, thick] (-1,-8.3) -- (-1,-9.7);
\draw[blue,->, thick] (1,-8.3) -- (1,-9.7);
\draw[red, dashed, ->, thick] (1.4,-8.3) -- (2.8,-9.6);
\draw[blue,->, thick] (3,-8.3) -- (3,-9.7);

\end{tikzpicture}
\caption{The crystal of $\UU^-_q(\gl(2|1))$ truncated to weights of height at most 5.}
\label{fig: gl21 crys}
\end{figure}
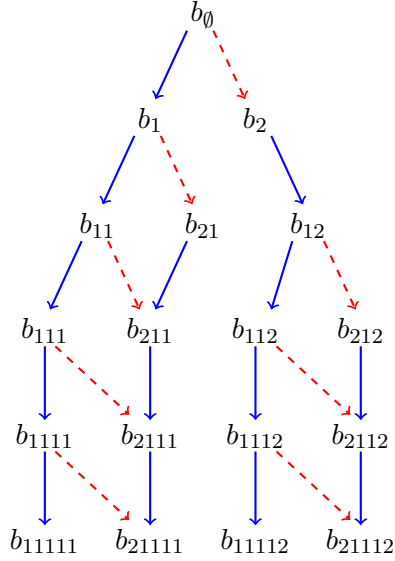
Furthermore, let us comment on the canonical basis of modules in this case.
Let $\lambda=a\epsp_1+b\epsp_2+c\epsp_3\in P^+$, so $a,b,c\in \Z$ with $a\geq b$.
Then there is a unique finite-dimensional simple module $V(\lambda)$ 
of highest weight $\lambda$. The compatibility of the canonical basis
with these modules is considered in \cite[\S 8]{CHW3}.
Specifically, for $\lambda\in P^+$, let $\cB(\lambda)=\set{uv_\lambda\mid u\in \cB}\setminus 0$;
note that for $\lambda\in\tPp$, this agrees with the definition
in Theorem \ref{thm:Globalization}.
Then in {\em loc. cit.}, it is shown that
\begin{enumerate}
\item $\cB(\lambda)$ is a basis for $V(\lambda)$ when $\lambda$ is typical;
\item $\cB(\lambda)$ is a basis for $V(\lambda)$ when $b=-c$;
\item $\cB(\lambda)$ is a basis for $V(\lambda)$ when $a=-c-1=b$; and
\item $\cB(\lambda)$ is linearly dependent in $V(\lambda)$ when $a=-c-1>b$.
\end{enumerate}
First, note that $K(\lambda)=V(\lambda)$ whenever $\lambda$ is typical,
so (1) follows from Theorem \ref{thm:Globalization} for $K(\lambda)$.
For (2), note that $V(\lambda)\in \catOint$ so the result
follows from Theorem \ref{thm:Globalization} for $V(\lambda)$ in this case.

It remains to comment on (3) and (4). 
Let $n=\ang{h_1,\lambda}=a-b$, and note that $\ang{h_2,\lambda}=b+c=b-a-1=-(n+1)$.
Then as observed in \cite{CHW3},
\[\cB(\lambda)=\set{F_1^{(r)}v_\lambda,\quad F_1^{(r+1)}F_2v_\lambda,\quad
F_2F_1^{(r)}v_\lambda\mid 0\leq r\leq n}\] 
is a spanning set, but we have the linear
dependence $F_2F_1^{(r)}v_\lambda=\frac{[n+1-r]}{[n+1]}F_1^{(r)}F_2v_\lambda$ 
for $1\leq r\leq n$ if $n>0$. However, if we consider the lattice $L=\cA \cB(\lambda)$,
since $\frac{[n+1-r]}{[n+1]}\in q^{r}\cA$, we observe that we can canonically
pick a basis for $V(\lambda)$; that is, the vectors
\[\cB'(\lambda)=\set{uv_\lambda\mid u\in \cB\text{ such that }uv_\lambda\notin qL}=\set{F_1^{(r)}v_\lambda, F_2v_\lambda, F_1^{(r+1)}F_2v_\lambda\mid 0\leq r\leq n}.\]

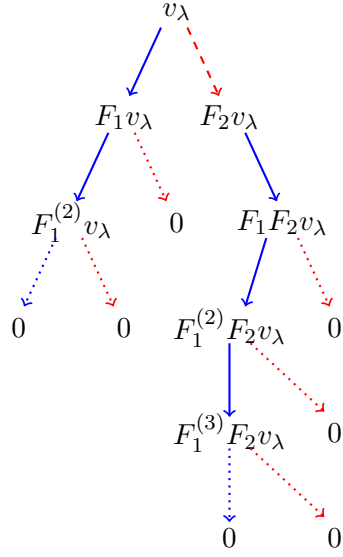
\begin{figure}[t]
\begin{tikzpicture}[scale=.7]
\draw (0,0) node{$v_\lambda$};
\draw (-1,-2) node{$F_1v_\lambda$};
\draw (1,-2) node{$F_2v_\lambda$};
\draw (-2,-4) node{$F_1^{(2)}v_\lambda$};
\draw (0,-4) node{$0$};
\draw (2,-4) node{$F_1F_2v_\lambda$};
\draw (-3,-6) node{$0$};
\draw (-1,-6) node{$0$};
\draw (1,-6) node{$F_1^{(2)}F_2v_\lambda$};
\draw (3,-6) node{$0$};
\draw (1,-8) node{$F_1^{(3)}F_2v_\lambda$};
\draw (3,-8) node{$0$};
\draw (1,-10) node{$0$};
\draw (3,-10) node{$0$};

\draw[blue,->, thick] (-.3,-.3) -- (-.9,-1.6);
\draw[red, dashed, ->, thick] (.2,-.3) -- (.8,-1.6);

\draw[blue,->, thick] (-1.3,-2.3) -- (-1.9,-3.6);
\draw[red, dashed, ->, thick, dotted] (-.8,-2.3) -- (-.2,-3.6);
\draw[blue,->, thick] (1.3,-2.3) -- (1.9,-3.6);

\draw[blue,->, thick, dotted] (-2.3,-4.3) -- (-2.9,-5.6);
\draw[red, dashed, ->, thick, dotted] (-1.8,-4.3) -- (-1.2,-5.6);
\draw[blue,->, thick] (1.7,-4.3) -- (1.3,-5.6);
\draw[red, dashed, ->, thick, dotted] (2.3,-4.3) -- (2.9,-5.6);

\draw[blue,->, thick] (1,-6.3) -- (1,-7.7);
\draw[red, dashed, ->, thick, dotted] (1.4,-6.3) -- (2.8,-7.6);

\draw[blue,->, thick, dotted] (1,-8.3) -- (1,-9.7);
\draw[red, dashed, ->, thick, dotted] (1.4,-8.3) -- (2.8,-9.6);

\end{tikzpicture}
\caption{The crystal of the atypical $\UU^-_q(\gl(2|1))$-module $V(2\epsp_1-3\epsp_4)$.}
\label{fig: gl21 crys atypical}
\end{figure}

We cannot define a crystal structure on $L$ as we have above.
Indeed, note that $L$ is not closed under the operator $\te_2$:
\[\te_2 F_1^{(r)}F_2v_\lambda=q^{-1}K_2E_2F_1^{(r)}F_2v_\lambda=-q^{-n-2+r}[n+1]F_1^{(r)}v_\lambda\in q^{r-2n-2}L.\] 
However, if we define $\te_2'=-q^{2(n+1)}\te_2$, then we see that $(L,\cB'(\lambda)+qL)$
is a crystal lattice under $\te_1, \tf_1, \te_2', \tf_2$.
Moreover this crystal is compatible with the one on $\Um$
modulo $q$; see Figure \ref{fig: gl21 crys atypical} and compare to Figure \ref{fig: gl21 crys}.

\subsection{Canonical bases for standard quantum $\gl(3|1)$}\label{subsec:CB(3|1)}
Next, let us explicitly construct the canonical basis in the case $m=3$.
We also want to produce some examples of canonical bases on atypical modules.
Both of these tasks are made easier by comparing canonical basis elements
to their corresponding PBW vectors modulo $q$,
hence we will write $x\equiv_q y$ if $x-y\in q\cL$.
On the other hand, each task is better suited by a particular PBW basis: either
that associated to the standard order $I=\set{1<2<3}$, or the opposite order
$I=\set{3<^{\rm op}2<^{\rm op}1}$.
To that end, we recall the associated PBW vectors
\[F_{12}=F_2F_1-qF_1F_2,\qquad F_{23}=F_3F_2-qF_2F_3;\qquad F_{123}=F_{3}F_{12}-qF_{12}F_3=F_{23}F_1-qF_1F_{23},\]
with respect to the standard order, whereas the opposite order yields the PBW vectors
\[F_{21}=F_1F_2-qF_2F_1,\qquad F_{32}=F_2F_3-qF_3F_2;\qquad F_{321}=F_{21}F_{3}-qF_3F_{21}=F_1F_{32}-qF_{32}F_1.\]
Note that
\[F_{21}=\tau(F_{12}),\quad F_{32}=\tau(F_{23}),\quad F_{321}=\tau(F_{123}),\]
and indeed in general we have $\BB(<)=\tau(\BB(<^{\rm op}))$.
It will be useful to note the identities
\[F_i^{(a)}F_jF_i^{(b)}=\bbinom{a+b-1}{b}F_i^{(a+b)}F_j+\bbinom{a+b-1}{b}F_jF_i^{(a+b)}\qquad i\in I_0,\ j=i\pm 1;\]
\[F_3F_2^{(a)}\equiv_q F_2^{(a-1)}F_{23};\quad F_2F_{123}=F_{123}F_2
\quad F_3F_{12}^{(a)}\equiv_qF_{12}^{(a-1)}F_{123};\quad F_3F_{23}=-qF_{23}F_3.\]
We also observe that $F_2^{(b)}F_3F_2F_3 F_2^{(a)}=F_3F_2^{(a+b+1)}F_3$,
and $F_3F_2F_1F_3F_2F_3$ is central in $\Um$.

\begin{thm}\label{thm:CB(3|1)}
For $x,y,z\in \set{0,1}$ and $a,b,c\in \Z_{\geq 0}$, let $u=u(x,y,z,a,b,c)\in \cB$
be the unique element equal to the PBW vector 
$F_3^xF_{32}^yF_{321}^zF_2^{(a)}F_{21}^{(b)}F_1^{(c)}$
modulo $q$.
Then
\[u=
\begin{cases}
F_3^x\ F_2^{(a+y)}\ F_1^{(b+c+z)}\ F_3^y\ F_2^{(b+z)}\ F_3^z
&\text{ if } c\geq a,\\
F_3^x F_2^{y} F_1^{(b+z)} F_3^y F_2^{(a+b+z)} F_3^z F_1^{(c)} 
&\text{ if } a> c\text{ and }y\leq x,\\
\displaystyle\sum_{t=0}^{b} (-1)^t \bbinom{a-c-1+t}{t}F_2^{(a+1+t)}F_1^{(b+c+z)}F_3F_2^{(b+z-t)}F_3^{z}
&\text{otherwise.}
\end{cases}\]

\end{thm}

Note that, in this theorem, the canonical basis elements are being indexed by PBW vectors
corresponding to the opposite order; this is the most useful description for determining
canonical basis elements for Kac modules. Of course, applying $\tau$ 
produces a similar description in terms of the standard PBW basis,
and it turns out that this is the easiest PBW basis for 
constructing the canonical basis elements.

\begin{proof}

We consider elements of weight $\mu=-i\alpha_1-j\alpha_2-k\alpha_3$
for $i,j,k\in \N$. Note that the $\mu$ weight space is nonzero only if $k\leq 3$. 
This makes it convenient to describe the canonical basis
on a case-by-case basis in terms of $k$. Furthermore, observe
that for $k\leq 1$, we can formally identify $\Um_\mu$ 
and $\UU_q^-(\fsl(4))_\mu$ 
(or more precisely, with the quotient $\parens{\UU_q^-(\fsl(4))/(F_3^2)}_\mu$). In particular,
for $k=0$ we have
$$\cB_\mu=\begin{cases}
\set{F_1^{(r)}F_2^{(j)}F_1^{(s)}\equiv_q F_1^{(r)}F_{12}^{(s)}F_2^{(j-s)}\mid j\geq r+s=i}&
\text{if } j\geq i;\\
\set{F_2^{(s)}F_1^{(i)}F_2^{(r)}\equiv_q F_1^{(i-s)}F_{12}^{(s)}F_2^{(r)}\mid i\geq r+s=j}&
\text{if } j\leq i;
\end{cases}$$
and for $k=1$, from essentially the same calculations as in 
\cite{Xi} we see that
$\cB_\mu$ contains the elements
\begin{align}
F_1^{(r)}F_2^{(j)}F_1^{(s)}F_3&\equiv_q F_1^{(r)}F_{12}^{(s)}F_2^{(j-s)}F_3,&\quad j\geq r+s=i\\
F_3F_1^{(r)}F_2^{(j)}F_1^{(s)}&\equiv_q 
\begin{cases}
F_1^{(r)}F_{12}^{(s-1)}F_{123}F_2^{(j-s)}\\
F_1^{(i)}F_2^{(j-1)}F_{23}\\
F_3
\end{cases}&
\begin{aligned}\text{ if }j\geq r+s=i\text{ and }s\geq 1,\\\text{ if }j\geq i=r\text{ and }j>s=0,\\\text{ if }j=i=r=s=0.
\end{aligned}\\
F_2^{(s)}F_1^{(i)}F_2^{(r)}F_3&\equiv_q F_1^{(i-s)}F_{12}^{(s)}F_2^{(r)}F_3, &i\geq r+s=j\\
F_3F_2^{(s)}F_1^{(i)}F_2^{(r)}&\equiv_q 
\begin{cases}
F_1^{(i-s)}F_{12}^{(s-1)}F_{123}F_2^{(r)}\\
F_1^{(i)}F_2^{(j-1)}F_{23}\\
F_3
\end{cases}&
\begin{aligned}\text{ if }i\geq r+s=j\text{ and }s\geq 1,\\\text{ if }i\geq j=s\text{ and }j>r=0,\\\text{ if }j=i=r=s=0.
\end{aligned}\\
F_2^{(s)}F_1^{(i)}F_3F_2^{(r)}&\equiv_q
F_1^{(i-s)}F_{12}^{(s)}F_2^{(r-1)}F_{23}&\text{ if } i+1\geq r+s= j \text{ and }r\geq 1
\end{align}
\begin{align}\begin{aligned}
\sum_{t=0}^{s} (-1)^t \bbinom{j-i-2+t}{t}F_2^{(s-t)}F_1^{(i)}&F_3F_2^{(j-s+t)}\equiv_q
F_1^{(i-s)}F_{12}^{(s)}F_2^{(j-s-1)}F_{23}\\
&\text{ if } j>i+1 \text{ and } 1\leq s\leq i
\end{aligned}\end{align}

Now consider the case $k=2$. 
Observe that applying $\tf_3$ to the $k=1$ canonical basis elements, we obtain the families
\begin{align}
F_3F_1^{(r)}F_2^{(j)}F_1^{(s)}F_3&\equiv_q \begin{cases}
F_1^{(r)}F_{12}^{(s-1)}F_{123}F_2^{(j-s)}F_3\\
F_1^{(i)}F_2^{(j-1)}F_{23}F_3
\end{cases}&
\begin{aligned}\text{ if }j\geq r+s=i\text{ and }s\geq 1,\\\text{ if }j\geq i=r\text{ and }j>s=0.
\end{aligned}\\
F_3F_2^{(s)}F_1^{(i)}F_2^{(r)}F_3&\equiv_q 
\begin{cases}
F_1^{(i-s)}F_{12}^{(s-1)}F_{123}F_2^{(r)}F_3\\
F_1^{(i)}F_2^{(j-1)}F_{23}F_3
\end{cases}&
\begin{aligned}\text{ if }i\geq r+s=j\text{ and }s\geq 1,\\\text{ if }i\geq j=s\text{ and }j>r=0.
\end{aligned}\\
F_3F_2^{(s)}F_1^{(i)}F_3F_2^{(r)}&\equiv_q 
F_1^{(i-s)}F_{12}^{(s-1)}F_{123}F_2^{(r-1)}F_{23}&\text{ if }i+1\geq r+s=j\text{ and }r,s\geq 1
\end{align}
\begin{align}\begin{aligned}
\sum_{t=0}^{s} (-1)^t \bbinom{j-i-2+t}{t}F_3F_2^{(s-t)}F_1^{(i)}&F_3F_2^{(j-s+t)}\equiv_q
F_1^{(i-s)}F_{12}^{(s-1)}F_{123}F_2^{(j-s-1)}F_{23}\\
&\text{ if } j>i+1 \text{ and } 1\leq s\leq i
\end{aligned}
\end{align}
This accounts for most of the canonical basis elements, except for most of those whose corresponding PBW elements
modulo $q$ has both $F_{23}$ and $F_3$ as factors.
However, since $F_{23}F_3=F_3F_2F_3$, it is clear that if $v$ is a canonical basis element from the $k=0$ case,
then $vF_3F_2F_3$ is a canonical basis element; indeed, note that
if $v=\sum_{\bf a} c_{\bf a}F^{\bf a}$ is the expression for $v$ in terms of the PBW basis,
then we obtain the PBW expansion $vF_3F_2F_3=\sum_{\bf a} c_{\bf a}F^{\bf a} F_{23}F_3$,
and the coefficients satisfy the same constraints. Therefore, we obtain the families
\begin{align}
F_1^{(r)}F_2^{(j-1)}F_1^{(s)}F_3F_2F_3&\equiv_q F_1^{(r)}F_{12}^{(s)}F_2^{(j-1-s)}F_{23}F_3& 
\text{ if }j-1\geq i=r+s,\\
F_2^{(s)}F_1^{(i)}F_3F_2^{(r)}F_3
&\equiv_q F_1^{(i-s)}F_{12}^{(s)}F_2^{(r-1)}F_{23}F_3,&\text{if } i+1\geq r+s=j.
\end{align}

Finally, for the case $k=3$, observe that $F_{123}F_{23}F_3=F_3F_2F_1F_3F_2F_3$,
which is central in $\Um$.
Since $F_2$ and $F_{123}$ commute, we again note that multiplying $k=0$ canonical basis
elements by $F_{123}F_{23}F_3$ on the right produces canonical basis elements, and thus 
the canonical basis elements are
\begin{align}
F_1^{(r)}F_2^{(j-2)}F_1^{(s)}F_3F_2F_1F_3F_2F_3&\equiv_q F_1^{(r)}F_{12}^{(s)}F_{123}F_2^{(j-1-s)}F_{23}F_3& 
\text{ if }j-2\geq r+s=i-1,\\
F_2^{(s)}F_1^{(i-1)}F_2^{(r)}F_3F_2F_1F_3F_2F_3
&\equiv_q F_1^{(i-1-s)}F_{12}^{(s)}F_{123}F_2^{(r)}F_{23}F_3,&\text{if } i-1\geq r+s=j-2.
\end{align}
This finishes the construction of the canonical basis elements. To obtain the description
in the statement of the theorem, we apply $\tau$ and reinterpret $r,s,i,j,k$ in
terms of the powers of the root vectors.
\end{proof}

\begin{figure}[h]
	\begin{tikzpicture}[scale=.7]
	\draw (0,-1+0) node (0){$\bullet$};
	\draw (1.2, -1+0.3) node (1) {$\bullet$};
	\draw (1.4, -1-0.3) node (2) {$\bullet$};
	\draw (2.5, -1+.9) node (11) {$\bullet$};
	\draw (2.7, -1+0.3) node (12) {$\bullet$};
	\draw (2.9, -1-0.3) node (21) {$\bullet$};
	\draw (3.1, -1-.9) node (22) {$\bullet$};
	\draw (3.7, -1+1.5) node (111) {$\bullet$};
	\draw (3.9, -1+.9) node (121) {$\bullet$};
	\draw (4.1, -1+0.3) node (212) {$\bullet$};
	\draw (4.3, -1-0.3) node (211) {$\bullet$};
	\draw (4.5, -1-.9) node (221) {$\bullet$};
	\draw (4.7, -1-1.5) node (222) {$\bullet$};
	\draw (4.9, -1+2.4) node (1111) {$\bullet$};
	\draw (5.1, -1+1.8) node (1211) {$\bullet$};
	\draw (5.3, -1+1.2) node (1212) {$\bullet$};
	\draw (5.5, -1+.6) node (2212) {$\bullet$};
	\draw (5.7, -1+0.0) node (2121) {$\bullet$};
	\draw (5.9, -1-.6) node (2111) {$\bullet$};
	\draw (6.1, -1-1.2) node (2211) {$\bullet$};
	\draw (6.3, -1-1.8) node (2221) {$\bullet$};
	\draw (6.5, -1-2.4) node (2222) {$\bullet$};
	
	\draw (-1, 4+0) node (3) {$\bullet$};
	\draw (.3, 4+.3) node (32) {$\bullet$};
	\draw (.5, 4-.3) node (31) {$\bullet$};
	\draw (1.5, 4+.9) node (322) {$\bullet$};
	\draw (1.7, 4+.3) node (321) {$\bullet$};
	\draw (1.9, 4-.3) node (312) {$\bullet$};
	\draw (2.1, 4-.9) node (311) {$\bullet$};
	\draw (2.7, 4+1.5) node (3222) {$\bullet$};
	\draw (2.9, 4+.9) node (3221) {$\bullet$};
	\draw (3.1, 4+.3) node (3211) {$\bullet$};
	\draw (3.3, 4-.3) node (3212) {$\bullet$};
	\draw (3.5, 4-.9) node (3121) {$\bullet$};
	\draw (3.7, 4-1.5) node (3111) {$\bullet$};
	
	\draw (.5, 6.5) node (323) {$\bullet$};
	\draw (1.8, 6.8) node (3232) {$\bullet$};
	\draw (2, 6.2) node (3231) {$\bullet$};
	
	\draw (5.7, 3) node (123) {$\bullet$};
	\draw (6.9, 3.3) node (1231) {$\bullet$};
	\draw (7.1, 2.7) node (2123) {$\bullet$};
	
	\draw (7.3, 1+0) node (23) {$\bullet$};
	\draw (8.5, 1+.3) node (231) {$\bullet$};
	\draw (8.7, 1-.3) node (232) {$\bullet$};
	\draw (9.7, 1+.9) node (2311) {$\bullet$};
	\draw (9.9, 1+.3) node (2312) {$\bullet$};
	\draw (10.1, 1-.3) node (2321) {$\bullet$};
	\draw (10.3, 1-.9) node (2322) {$\bullet$};
	
	\draw (6, 5.5) node (3123) {$\bullet$};
	
	
	
	\fill[color=white!80!black, path fading=east] (-.5,-1) -- (5,1.9) -- (5,-3.2) -- cycle;
	\draw[color=white!70!black, dashed] (-.5,-1) -- (5,1.9);
	\draw[color=white!70!black, dashed] (-.5,-1) -- (7,-4);
	
	\draw[blue,->, thick] (0.east) -- (1.west);
	\draw[green!50!black, densely dotted, ->, thick] (0.east) -- (2.west);
	begin
	\draw[blue,->, thick] (1.east) -- (11.west);
	\draw[green!50!black, densely dotted,->, thick] (1.east) -- (12.west);
	\draw[blue,->, thick] (2.east) -- (21.west);
	\draw[green!50!black, densely dotted,->, thick] (2.east) -- (22.west);
	
	\draw[blue,->, thick] (11.east) -- (111.west);
	\draw[green!50!black, densely dotted,->, thick] (11.east) -- (121.west);
	\draw[blue,->, thick] (12.east) -- (121.west);
	\draw[green!50!black, densely dotted,->, thick] (12.east) -- (212.west);
	\draw[blue,->, thick] (21.east) -- (211.west);
	\draw[green!50!black, densely dotted,->, thick] (21.east) -- (221.west);
	\draw[blue,->, thick] (22.east) -- (221.west);
	\draw[green!50!black, densely dotted,->, thick] (22.east) -- (222.west);
	
	\draw[blue,->, thick] (111.east) -- (1111.west);
	\draw[green!50!black, densely dotted,->, thick] (111.east) -- (1211.west);
	\draw[blue,->, thick] (121.east) -- (1211.west);
	\draw[green!50!black, densely dotted,->, thick] (121.east) -- (1212.west);
	\draw[blue,->, thick] (211.east) -- (2111.west);
	\draw[green!50!black, densely dotted,->, thick] (211.east) -- (2121.west);
	\draw[blue,->, thick] (212.east) -- (2121.west);
	\draw[green!50!black, densely dotted,->, thick] (212.east) -- (2212.west);
	\draw[blue,->, thick] (221.east) -- (2211.west);
	\draw[green!50!black, densely dotted,->, thick] (221.east) -- (2221.west);
	\draw[blue,->, thick] (222.east) -- (2221.west);
	\draw[green!50!black, densely dotted,->, thick] (222.east) -- (2222.west);

	\draw[red, dashed, ->, thick] (0.north) -- (3.south);
	\draw[red, dashed, ->, thick] (1.north) -- (31.south);
	\draw[red, dashed, ->, thick] (11.north) -- (311.south);
	\draw[red, dashed, ->, thick] (111.north) -- (3111.south);
	\draw[red, dashed, ->, thick] (212.north) -- (2123.south);
	\draw[red, dashed, ->, thick] (121.north) -- (1231.south);
	\draw[red, dashed, ->, thick] (12.north) -- (123.south);

	
	\fill[color=white!80!black, path fading=east] (-1.5,4) -- (4,6.5) -- (4,2) -- cycle;
	\draw[color=white!70!black, dashed] (-1.5,4) -- (4,6.5);
	\draw[color=white!70!black, dashed] (-1.5,4) -- (4,2);less than

	\draw[red, dashed, ->, thick] (0.north) -- (3.south);
	\draw[red, dashed, ->, thick] (1.north) -- (31.south);
	\draw[red, dashed, ->, thick] (11.north) -- (311.south);
	\draw[red, dashed, ->, thick] (111.north) -- (3111.south);
	
	\draw[blue,->, thick] (3.east) -- (31.west);
	\draw[green!50!black, densely dotted,->, thick] (3.east) -- (32.west);
	\draw[blue,->, thick] (31.east) -- (311.west);
	\draw[green!50!black, densely dotted,->, thick] (31.east) -- (312.west);
	\draw[blue,->, thick] (32.east) -- (321.west);
	\draw[green!50!black, densely dotted,->, thick] (32.east) -- (322.west);
	\draw[blue,->, thick] (311.east) -- (3111.west);
	\draw[green!50!black, densely dotted,->, thick] (311.east) -- (3121.west);
	\draw[blue,->, thick] (312.east) -- (3121.west);
	\draw[green!50!black, densely dotted,->, thick] (312.east) -- (3212.west);
	\draw[blue,->, thick] (321.east) -- (3211.west);
	\draw[green!50!black, densely dotted,->, thick] (321.east) -- (3221.west);
	\draw[blue,->, thick] (322.east) -- (3221.west);
	\draw[green!50!black, densely dotted,->, thick] (322.east) -- (3222.west);
	
	\draw[red, dashed, ->, thick] (32.east) -- (323.south);
	\draw[red, dashed, ->, thick] (321.east) -- (3231.south);
	\draw[red, dashed, ->, thick] (322.east) -- (3232.south);
	
	\draw[red, dashed, ->, thick] (312.east) -- (3123.west);
	
	\fill[color=white!80!black, path fading=east] (5.5,5.5) -- (6.5,5) -- (6.5,6) -- cycle;
	\draw[color=white!70!black, dashed] (5.5,5.5) -- (6.5,6);
	\draw[color=white!70!black, dashed] (5.5,5.5) -- (6.5,5);
	
	\fill[color=white!80!black, path fading=east] (5.2,3) -- (7.5,3.9) -- (7.5,2.4) -- cycle;
	\draw[color=white!70!black, dashed] (5.2,3) -- (7,3.7);
	\draw[color=white!70!black, dashed] (5.2,3) -- (7.5,2.4);

	\draw[blue,->, thick] (123.east) -- (1231.west);
	\draw[green!50!black, densely dotted,->, thick] (123.east) -- (2123.west);

	\draw[red, dashed, ->, thick] (2.east) -- (23.south west);
	\draw[red, dashed, ->, thick] (21.east) -- (231.south west);
	\draw[red, dashed, ->, thick] (22.east) -- (232.south west);
	\draw[red, dashed, ->, thick] (211.east) -- (2311.south west);
	\draw[red, dashed, ->, thick] (221.east) -- (2321.south west);
	\draw[red, dashed, ->, thick] (222.east) -- (2322.south west);
	
	\fill[color=white!80!black, path fading=east] (7,1) -- (11,2.7) -- (11,-.7) -- cycle;
	\draw[color=white!70!black, dashed] (7,1) -- (11,2.7);
	\draw[color=white!70!black, dashed] (7,1) -- (11,-.7);

	\draw[blue,->, thick] (23.east) -- (231.west);
	\draw[green!50!black, densely dotted,->, thick] (23.east) -- (232.west);
	\draw[blue,->, thick] (231.east) -- (2311.west);
	\draw[green!50!black, densely dotted,->, thick] (231.east) -- (2312.west);
	\draw[blue,->, thick] (232.east) -- (2321.west);
	\draw[green!50!black, densely dotted,->, thick] (232.east) -- (2322.west);
	
	\fill[color=white!80!black, path fading=east] (0,6.5) -- (2.3,7.4) -- (2.3,5.9) -- cycle;
	\draw[color=white!70!black, dashed] (0,6.5) -- (1.8,7.2);
	\draw[color=white!70!black, dashed] (0,6.5) -- (2.3,5.9);
	
	\draw[blue,->, thick] (323.east) -- (3231.west);
	\draw[green!50!black, densely dotted,->, thick] (323.east) -- (3232.west);

	\draw (0) node {$\bullet$};
	\draw (1) node {$\bullet$};
	\draw (2) node {$\bullet$};
	\draw (11) node {$\bullet$};
	\draw (12) node {$\bullet$};
	\draw (21) node {$\bullet$};
	\draw (22) node {$\bullet$};
	\draw (111) node {$\bullet$};
	\draw (121) node {$\bullet$};
	\draw (212) node {$\bullet$};
	\draw (211) node {$\bullet$};
	\draw (221) node {$\bullet$};
	\draw (222) node {$\bullet$};
	\draw (1111) node {$\bullet$};
	\draw (1211) node {$\bullet$};
	\draw (1212) node {$\bullet$};
	\draw (2212) node {$\bullet$};
	\draw (2121) node {$\bullet$};
	\draw (2111) node {$\bullet$};
	\draw (2211) node {$\bullet$};
	\draw (2221) node {$\bullet$};
	\draw (2222) node {$\bullet$};
	
	\draw (3) node {$\bullet$};
	\draw (31) node {$\bullet$};
	\draw (32) node {$\bullet$};
	\draw (311) node {$\bullet$};
	\draw (312) node {$\bullet$};
	\draw (321) node {$\bullet$};
	\draw (322) node {$\bullet$};
	\draw (3111) node {$\bullet$};
	\draw (3121) node {$\bullet$};
	\draw (3212) node {$\bullet$};
	\draw (3211) node {$\bullet$};
	\draw (3221) node {$\bullet$};
	\draw (3222) node {$\bullet$};
	
	\draw (323) node {$\bullet$};
	\draw (3231) node {$\bullet$};
	\draw (3232) node {$\bullet$};
	\draw (123) node {$\bullet$};
	\draw (1231) node {$\bullet$};
	\draw (2123) node {$\bullet$};
	
	\draw (23) node {$\bullet$};
	\draw (231) node {$\bullet$};
	\draw (232) node {$\bullet$};
	\draw (2311) node {$\bullet$};
	\draw (2312) node {$\bullet$};
	\draw (2321) node {$\bullet$};
	\draw (2322) node {$\bullet$};
	
	\draw (3123) node {$\bullet$};
	\end{tikzpicture}
	
	\caption{The crystal of $\UU^-_q(\gl(3|1))$ truncated to weights of height at most 4,
		arranged to highlight the copies of the crystal for $\UU^-_q(\fsl(3))$. }
\end{figure}
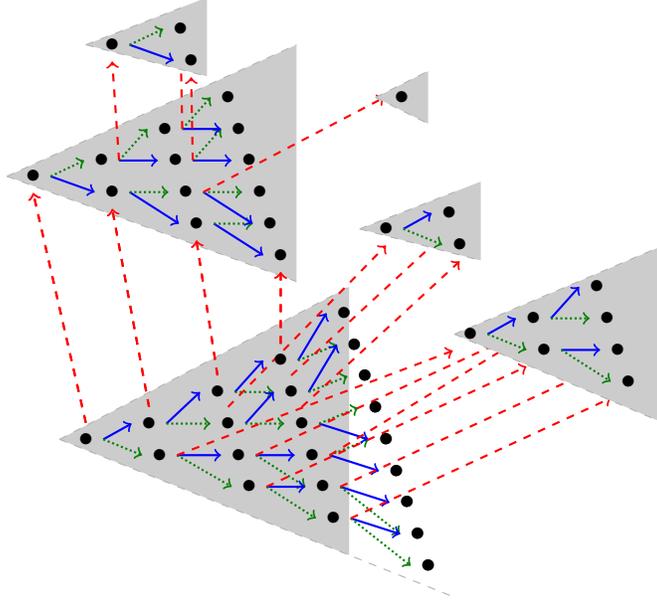

With this explicit description in hand, we can consider the compatibility of this basis
with the irreducible finite-dimensional modules $V(\lambda)$ for $\lambda\in P^+$. 
Of course, this is answered for $\lambda$ typical or fully dominant by 
Theorem \ref{thm:Globalization}, 
so it remains to consider the case when $\lambda$ is a dominant atypical weight which
is not fully dominant; explicitly, those weights $\lambda=a\epsp_1+b\epsp_2+c\epsp_3+d\epsp_4$
such that $a\geq b\geq c$ and $d=-b-1$ or $d=-a-2$.
At present, we do not know of a way to do this systematically,
but we will explicitly work out some examples.

\begin{example}
Let $\lambda=n\epsp_1-(n+2)\epsp_4$ for $n\in \Z_{\geq 0}$.
Then $\dim V_\VV(\lambda)=\binom{n+2}{2}$ has the canonical basis given by $F_2^{(a)}F_1^{(b)}v_{\lambda}$,
where $0\leq a\leq b\leq n$,
and we deduce that $K(\lambda)$ has the canonical basis elements
\[F_2^{(a)}F_1^{(b)}1_\lambda,\qquad F_3F_2^{(a)}F_1^{(b)}1_\lambda,
\qquad F_2^{(a+1)}F_1^{(b)}F_31_\lambda,\qquad F_2^{(a)}F_1^{(b+1)}F_2F_31_\lambda\]
\[F_3F_2^{(a+1)}F_1^{(b)}F_31_\lambda,\qquad
F_3F_2^{(a)}F_1^{(b+1)}F_2F_31_\lambda,\qquad
F_2^{(a+1)}F_1^{(b+1)}F_3F_2F_31_\lambda,\]
\[F_2^{(a)}F_1^{(b)}F_3F_2F_1F_3F_2F_31_\lambda,\]
where again $0\leq a\leq b\leq n$.
On the other hand, one can compute that $\dim V(\lambda)=2n^2+8n+7$ (cf. \cite{VdJ}
for the formula in the classical limit\footnote{This formula still holds in the
quantum case by arguments similar to those in \cite[Section 2.6]{Kw14};
to wit, it makes sense to take a classical limit of a Kac module to get the Kac module
of $\gl(m|1)$, and in this limit the maximal submodule of $K(\lambda)$ is still a submodule.
This gives us an upper bound $2n^2+4n+1$ on the dimension of this submodule, and on the
other hand we shall produce $2n^2+4n+1$ linearly independent vectors in this submodule.}). 
In particular,
we have $\dim N(\lambda)=2n^2+4n+1$ basis vectors
to remove to get a basis of $V(\lambda)$. 

Well, first one finds that we have the linear dependencies
\[F_3F_2^{(a+1)}F_1^{(b+1)}v_\lambda=\frac{[n+1-a]}{[n+2]}F_2^{(a+1)}F_1^{(b+1)}F_3v_\lambda-\frac{[n-b]}{[n+2]}F_2^{(a)}F_1^{(b+1)}F_2F_3v_\lambda\]
for $0\leq a\leq b\leq n-1$.
Consequently, we deduce that
\[F_3F_2^{(a+1)}F_1^{(b+1)}F_3v_\lambda=\frac{[n-b]}{[n+1-a]}F_3F_2^{(a)}F_1^{(b+1)}F_2F_3v_\lambda\]
for $0\leq a\leq b\leq n-1$. On the other hand, observe that 
\[F_3F_2^{(a+1)}F_1^{(b+1)}F_3v_\lambda=F_2^{(a)}F_3F_2F_1^{(b+1)}F_3v_\lambda=\frac{[n-b]}{[n+1]}F_2^{(a)}F_1^{(b+1)}F_3F_2F_3v_\lambda,\]
hence
\[F_3F_2^{(c+1)}F_1^{(c)}F_3v_\lambda=\frac{[n+1-c]}{[n+1]}F_2^{(c)}F_1^{(c)}F_3F_2F_3v_\lambda,\]
\[F_3F_2^{(a)}F_1^{(b+1)}F_2F_3v_\lambda=\frac{[n+1-a]}{[n+1]}F_2^{(a)}F_1^{(b+1)}F_3F_2F_3v_\lambda\]
for $1\leq a\leq b\leq n$ and $1\leq c\leq n$ (the case $b=n$ being a similar computation).
We also note that $F_2^{(a)}F_1^{(b)}F_3F_2F_1F_3F_2F_3v_\lambda=0$
for $0\leq a\leq b\leq n$.

In particular, note that we have shown that $3\binom{n+1}{2}+\binom{n+2}{2}+n=2n^2+4n+1$ of the canonical
basis vectors in $K(\lambda)$ are $q\cA$-linear combinations of canonical basis elements, so we have the canonical
basis
\begin{align*}
\cB(\lambda)&=\set{bv_\lambda\mid b\in\cB,\quad bv_\lambda\notin q\cA \cB v_\lambda}\\
&=\set{
	\begin{array}{c}
	F_2^{(a)}F_1^{(b)}v_\lambda,\quad
	F_1^{(c)}F_3v_\lambda, \\
	F_2^{(a+1)}F_1^{(b)}F_3v_\lambda,\quad
	F_2^{(a)}F_1^{(b+1)}F_2F_3v_\lambda,\\
	F_1^{(d)}F_3F_2F_3v_\lambda,\quad
	F_2^{(a+1)}F_1^{(b+1)}F_3F_2F_3v_\lambda
	\end{array}
	\mid \begin{array}{c}0\leq a\leq b\leq n\\0\leq c\leq n\\0\leq d\leq n+1\end{array}}
\end{align*}
We note that this basis seems compatible with the crystal structure on $K(\lambda)$,
though as in the $\gl(2|1)$ case one would need to modify the operator $\te_3$ to obtain
an actual crystal structure; for instance, if $L=\cA\cB(\lambda)$ and $B=\cB(\lambda)+qL$,
then $(L,B)$ is a crystal basis of $V(\lambda)$ where we replace $\te_3$ with
$\te_3'=-q^{-2\epsp_3^\vee} \te_3$.
See Figure \ref{fig:gl(3|1)vecind} for an
example in the case $n=1$.
\end{example}

\begin{example}
Consider the weight $\mu=n\epsp_1-\epsp_4$, in which case we have $\dim V(\mu)=3n^2+8n+4$. 
By similar considerations to the previous example,
we can show that $F_3F_2F_31_\mu$ is singular and generates $N(\mu)$, 
and we have the canonical basis
\begin{align*}
\cB(\mu)&=\set{bv_\mu\mid b\in\cB,\quad bv_\mu\neq 0}\\
&=\set{
	\begin{array}{c}
	F_2^{(a)}F_1^{(b)}v_\mu,
	F_3F_2^{(a)}F_1^{(b)}v_\mu,\\
	F_2^{(a+1)}F_1^{(b)}F_3v_\mu,
	F_2^{(a)}F_1^{(b+1)}F_2F_3v_\mu,\\
	F_3F_2^{(c+1)}F_1^{(d)}F_3v_\mu,\\
	F_3F_2^{(e)}F_1^{(f+1)}F_2F_3v_\mu.
	\end{array}
	\mid
	\begin{array}{c} 0\leq a\leq b\leq n, \\ 0\leq c\leq d\leq n,\quad d>0,\\
	1\leq e\leq f\leq n.
	\end{array}}.
\end{align*} 
We note that $E_3 F_3F_2F_1v_\mu=0$ and thus $\te_3 \tf_3 F_2F_1v_\mu=0$, 
so once again we fail to have a crystal basis with respect to the previously defined
Kashiwara operators.

\end{example}

\begin{figure}[t]
\begin{tikzpicture}
	\draw (0,0) node (0){\small $1_\lambda$};
	\draw (0,-1.2) node (1){\small $F_11_\lambda$};
	\draw (0,-2.4) node (21){\small $F_2F_11_\lambda$};
	\draw (0,-3.6) node[draw, inner sep=3pt] (321){\small $F_3F_2F_11_\lambda$};
	
	\draw (1.5,-0) node (3){\small $F_31_\lambda$};
	\draw (1.5,-1.2) node (13){\small $F_1F_31_\lambda$};
	\draw (1.5,-2.4) node (213){\small $F_2F_1F_31_\lambda$};
	\draw (3,-3.6) node (2213){\small $F_2^{(2)}F_1F_31_\lambda$};
	\draw (3.5,0) node (23){\small $F_2F_31_\lambda$};
	\draw (3.5,-1.2) node (123){\small $F_1F_2F_31_\lambda$};
	\draw (3.5,-2.4) node (1123){\small $F_1^{(2)}F_2F_31_\lambda$};
	\draw (5.7,-3.6) node (12213){\small $F_1F_2^{(2)}F_1F_31_\lambda$};
	
	\draw (6.5,0) node (323){\small $F_3F_2F_31_\lambda$};
	\draw (6.5,-1.2) node (1323){\small $F_1F_3F_2F_31_\lambda$};
	\draw (6.5,-2.4) node[draw, inner sep=3pt] (11323){\small $F_1^{(2)}F_3F_2F_31_\lambda$};
	\draw (9.5,-1.2) node (21323){\small $F_2F_1F_3F_2F_31_\lambda$};
	\draw (9.5,-2.4) node[draw, inner sep=3pt] (211323){\small $F_2F_1^{(2)}F_3F_2F_31_\lambda$};
	\draw (9,-3.6) node[draw, inner sep=3pt] (2211323){\small $F_2^{(2)}F_1^{(2)}F_3F_2F_31_\lambda$};
	
	\draw (1.5,-4.8) node[draw, inner sep=3pt] (3213){\small $F_3F_2F_1F_31_\lambda$};
	\draw (4.5,-4.8) node[draw, inner sep=3pt] (32213){\small $F_3F_2^{(2)}F_1F_31_\lambda$};
	\draw (7.5,-4.8) node[draw, inner sep=3pt] (132213){\small $F_1F_3F_2^{(2)}F_1F_31_\lambda$};
	
	\draw (13,-1.2) node[draw, inner sep=3pt] (321323){\small $F_3F_2F_1F_3F_2F_31_\lambda$};
	
	\draw (13,-2.4) node[draw, inner sep=3pt] (1321323){\small $F_1F_3F_2F_1F_3F_2F_31_\lambda$};
	
	\draw (13,-3.6) node[draw, inner sep=3pt] (21321323){\small $F_2F_1F_3F_2F_1F_3F_2F_31_\lambda$};
	
	\draw[blue,thick, ->] (0) -- (1);
	\draw[red, dashed, thick, ->] (0) -- (3);
	
	\draw[blue,thick, ->] (0) -- (1);
	\draw[red, dashed, thick, ->] (0) -- (3);
	
	\draw[green!50!black, densely dotted,thick, ->] (1) -- (21);
	\draw[red, dashed, thick, ->] (1) -- (13);
	
	\draw[green!50!black, densely dotted,thick, ->] (3) -- (23);
	\draw[blue,thick, ->] (3) -- (13);
	
	\draw[red, dashed, thick, ->] (23) -- (323);
	\draw[blue,thick, ->] (23) -- (123);
	
	\draw[green!50!black, densely dotted,thick, ->] (13) -- (213);
	
	\draw[red, dashed, thick, ->] (21) -- (321);
	
	\draw[blue,thick, ->] (323) -- (1323);
	
	\draw[blue,thick, ->] (123) -- (1123);
	\draw[red, dashed, thick, ->] (123) -- (1323);
	
	\draw[green!50!black, densely dotted,thick, ->] (213) -- (2213);
	\draw[red, dashed, thick, ->] (213) -- (3213);
	
	\draw[green!50!black, densely dotted,thick, ->] (1323) -- (21323);
	\draw[blue,thick, ->] (1323) -- (11323);
	
	\draw[green!50!black, densely dotted,thick, ->] (1123) -- (12213);
	\draw[red, dashed, thick, ->] (1123) -- (11323);
	
	\draw[blue,thick, ->] (2213) -- (12213);
	\draw[red, dashed, thick, ->] (2213) -- (32213);
	
	\draw[green!50!black, densely dotted,thick, ->] (3213) -- (32213);
	
	\draw[red, dashed, thick, ->] (21323) -- (321323);
	\draw[blue,thick, ->] (21323) -- (211323);
	
	\draw[green!50!black, densely dotted,thick, ->] (11323) -- (211323);
	
	\draw[red, dashed, thick, ->] (12213) -- (132213);
	
	\draw[blue,thick, ->] (32213) -- (132213);
	
	\draw[blue,thick, ->] (321323) -- (1321323);
	
	\draw[green!50!black, densely dotted,thick, ->] (211323) -- (2211323);
	\draw[red, dashed, thick, ->] (211323) -- (1321323);
	
	\draw[green!50!black, densely dotted,thick, ->] (1321323) -- (21321323);
	
	\draw[red, dashed, thick, ->] (2211323) -- (21321323);
	\draw[green!50!black, densely dotted,thick, ->] (1) -- (21);
	\draw[red, dashed, thick, ->] (1) -- (13);
	
	\draw[green!50!black, densely dotted,thick, ->] (3) -- (23);
	\draw[blue,thick, ->] (3) -- (13);
	
	\draw[red, dashed, thick, ->] (23) -- (323);
	\draw[blue,thick, ->] (23) -- (123);
	
	\draw[green!50!black, densely dotted,thick, ->] (13) -- (213);
	
	\draw[red, dashed, thick, ->] (21) -- (321);
	
	\draw[blue,thick, ->] (323) -- (1323);
	
	\draw[blue,thick, ->] (123) -- (1123);
	\draw[red, dashed, thick, ->] (123) -- (1323);
	
	\draw[green!50!black, densely dotted,thick, ->] (213) -- (2213);
	\draw[red, dashed, thick, ->] (213) -- (3213);
	
	\draw[green!50!black, densely dotted,thick, ->] (1323) -- (21323);
	\draw[blue,thick, ->] (1323) -- (11323);
	
	\draw[green!50!black, densely dotted,thick, ->] (1123) -- (12213);
	\draw[red, dashed, thick, ->] (1123) -- (11323);
	
	\draw[blue,thick, ->] (2213) -- (12213);
	\draw[red, dashed, thick, ->] (2213) -- (32213);
	
	\draw[green!50!black, densely dotted,thick, ->] (3213) -- (32213);
	
	\draw[red, dashed, thick, ->] (21323) -- (321323);
	\draw[blue,thick, ->] (21323) -- (211323);
	
	\draw[green!50!black, densely dotted,thick, ->] (11323) -- (211323);
	
	\draw[red, dashed, thick, ->] (12213) -- (132213);
	
	\draw[blue,thick, ->] (32213) -- (132213);
	
	\draw[blue,thick, ->] (321323) -- (1321323);
	
	\draw[green!50!black, densely dotted,thick, ->] (211323) -- (2211323);
	\draw[red, dashed, thick, ->] (211323) -- (1321323);
	
	\draw[green!50!black, densely dotted,thick, ->] (1321323) -- (21321323);
	
	\draw[red, dashed, thick, ->] (2211323) -- (21321323);
\end{tikzpicture}
\caption{The crystal associated to the $\UU_q(\gl(3|1))$-modules $K(\lambda)$ 
for $\lambda=\epsp_1-3\epsp_4$. The boxed elements
are those basis elements which are zero modulo $q$ in the simple quotient $V(\lambda)$}\label{fig:gl(3|1)vecind}
\end{figure}
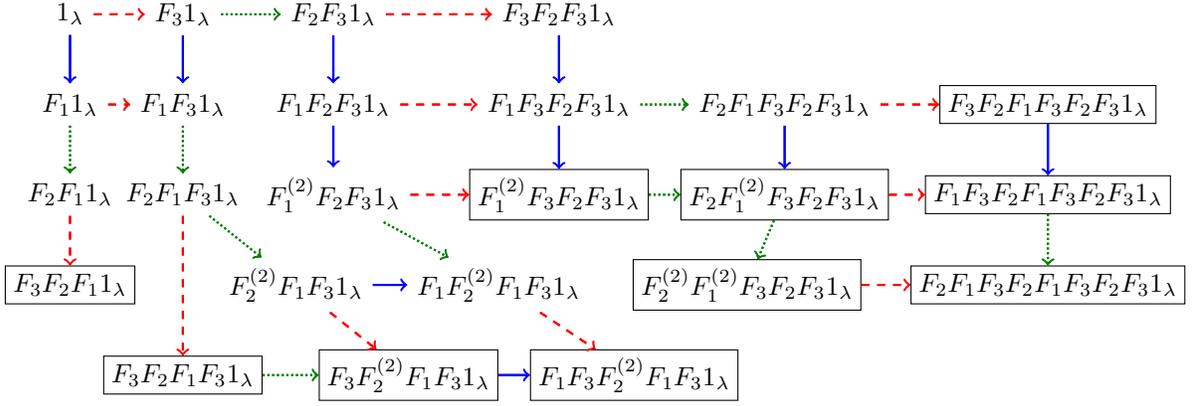

\begin{example}
Let us consider one more example. Let 
$\lambda_1=\epsp_1+\epsp_2-2\epsp_4$ and $\lambda_2=\epsp_1+\epsp_2-3\epsp_4$. 
Then for $\lambda=\lambda_1$ or $\lambda_2$, $K(\lambda)$ is a $24$-dimensional 
$\UU$-module.
In particular, letting $1_\lambda$
denote the highest weight vector of $K(\lambda)$, we can explicitly describe
the crystal and canonical bases; see Figure 
\ref{fig:gl(3|1)dualvecind}.
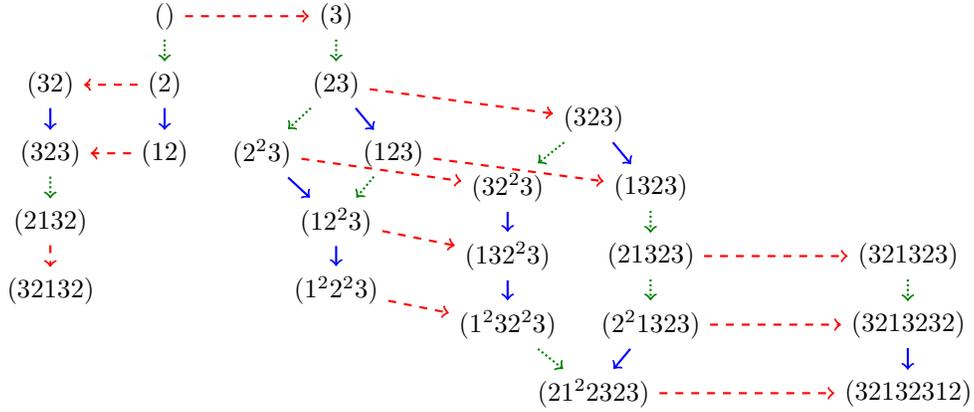
\begin{figure}[h]
\begin{tikzpicture}[scale=.76]
	\draw (0,0) node (0){\small $()$};
	\draw (0,-1.2) node (2){\small $(2)$};
	\draw (0,-2.4) node (12){\small $(12)$};
	
	\draw (-2,-1.2) node (32){\small $(32)$};
	\draw (-2,-2.4) node (132){\small $(323)$};
	\draw (-2,-3.6) node (2132){\small $(2132)$};
	\draw (-2,-4.8) node (32132){\small $(32132)$};
	
	\draw (3,-0) node (3){\small $(3)$};
	\draw (3,-1.2) node (23){\small $(23)$};
	\draw (1.7,-2.4) node (223){\small $(2^23)$};
	\draw (4,-2.4) node (123){\small $(123)$};
	\draw (3,-3.6) node (1223){\small $(12^23)$};
	\draw (3,-4.8) node (11223){\small $(1^22^23)$};
	
	\draw (7.5,-1.8) node (323){\small $(323)$};
	\draw (6,-3) node (3223){\small $(32^23)$};
	\draw (6,-4.2) node (13223){\small $(132^23)$};
	\draw (6,-5.4) node (113223){\small $(1^232^23)$};
	\draw (8.5,-3) node (1323){\small (1323)};
	\draw (8.5,-4.2) node (21323){\small $(21323)$};
	\draw (8.5,-5.4) node (221323){\small $(2^21323)$};
	\draw (7.5,-6.6) node (2112323){\small $(21^22323)$};
	
	\draw (13,-4.2) node (321323){\small (321323)};
	
	\draw (13,-5.4) node (3213232){\small (3213232)};
	
	\draw (13,-6.6) node (32132312){\small (32132312)};
	
	\draw[green!50!black, densely dotted,thick, ->] (0) -- (2);
	\draw[green!50!black, densely dotted,thick, ->] (3) -- (23);
	\draw[green!50!black, densely dotted,thick, ->] (23) -- (223);
	\draw[green!50!black, densely dotted,thick, ->] (132) -- (2132);
	\draw[green!50!black, densely dotted,thick, ->] (123) -- (1223);
	\draw[green!50!black, densely dotted,thick, ->] (323) -- (3223);
	\draw[green!50!black, densely dotted,thick, ->] (1323) -- (21323);
	\draw[green!50!black, densely dotted,thick, ->] (21323) -- (221323);
	\draw[green!50!black, densely dotted,thick, ->] (113223) -- (2112323);
	\draw[green!50!black, densely dotted,thick, ->] (321323) -- (3213232);
	\draw[blue,thick, ->] (2) -- (12);
	\draw[blue,thick, ->] (32) -- (132);
	\draw[blue,thick, ->] (23) -- (123);
	\draw[blue,thick, ->] (223) -- (1223);
	\draw[blue,thick, ->] (323) -- (1323);
	\draw[blue,thick, ->] (1223) -- (11223);
	\draw[blue,thick, ->] (3223) -- (13223);
	\draw[blue,thick, ->] (13223) -- (113223);
	\draw[blue,thick, ->] (221323) -- (2112323);
	\draw[blue,thick, ->] (3213232) -- (32132312);
	\draw[red, dashed,thick, ->] (0) -- (3);
	\draw[red, dashed,thick, ->] (2) -- (32);
	\draw[red, dashed,thick, ->] (12) -- (132);
	\draw[red, dashed,thick, ->] (23) -- (323);
	\draw[red, dashed,thick, ->] (223) -- (3223);
	\draw[red, dashed,thick, ->] (123) -- (1323);
	\draw[red, dashed,thick, ->] (2132) -- (32132);
	\draw[red, dashed,thick, ->] (1223) -- (13223);
	\draw[red, dashed,thick, ->] (11223) -- (113223);
	\draw[red, dashed,thick, ->] (21323) -- (321323);
	\draw[red, dashed,thick, ->] (221323) -- (3213232);
	\draw[red, dashed,thick, ->] (2112323) -- (32132312);

\end{tikzpicture}
\caption{The crystal associated to the $\UU_q(\gl(3|1))$-modules $K(\epsp_1+\epsp_2+d\epsp_4)$, $d\in\Z$. Here, $(i_1^{k_1}i_2^{k_2}\ldots i_t^{k_t})$ is shorthand for $F_{i_1}^{(k_1)}\ldots F_{i_t}^{(k_t)}1_\lambda$.}\label{fig:gl(3|1)dualvecind}
\end{figure}

First, let us compute a basis for $V(\lambda_1)$; we compute, again
using \cite{VdJ}, that $\dim V(\lambda_1)=9$.
Now in $K(\lambda_1)$, observe that $F_2F_31_{\lambda_1}-[2]F_3F_21_{\lambda_1}$
is a singular vector. In particular, in $V(\lambda_1)$ we have the identity
$F_3F_2v_{\lambda_1}=\frac{q}{1+q^2}F_2F_3v_{\lambda_1}$, which further implies
\[F_1F_3F_2v_{\lambda_1}=\frac{q}{1+q^2}F_1F_2F_3v_{\lambda_1},\quad
F_2F_1F_3F_2v_{\lambda_1}=\frac{q}{1+q^2}F_1F_2^{(2)}F_3v_{\lambda_1},\]
\[F_3F_2F_3v_{\lambda_1}=0,\quad F_3F_2F_1F_3F_2v_{\lambda_1}=0,\]
Thus, in this case, we see that have the canonical basis
\[
\cB(\lambda_1)=\set{xv_{\lambda_1}\mid x\in \cB(\infty)\text{ and }xv_{\lambda_1}\notin q\cL(\lambda_1)}=\set{\begin{array}{c}
v_{\lambda_1}, \\
F_2v_{\lambda_1}, \quad F_3v_{\lambda_1},\\
F_1F_2v_{\lambda_1}, \quad F_2F_3v_{\lambda_1},\\
F_2^{(2)}F_3v_{\lambda_1}, \quad F_1F_2F_3v_{\lambda_1}, \\
F_1F_2^{(2)}F_3v_{\lambda_1},\quad F_1^{(2)}F_2^{(2)}F_3v_{\lambda_1}
\end{array}
}\]

Similarly, we compute that $\dim V(\lambda_2)=20$. Well, in $K(\lambda_2)$, we have the unique singular vector 
\[F_2F_1F_3F_2F_31_{\lambda_2}-[2]F_1F_3F_2^{(2)}F_31_{\lambda_2}+[3]F_3F_2F_1F_3F_21_{\lambda_2},\]
hence in particular in $V(\lambda_2)$ we have the identity
\[F_3F_2F_1F_3F_2v_{\lambda_2}=
\frac{q+q^3}{1+q^2+q^4}F_1F_3F_2^{(2)}F_3v_{\lambda_2}-\frac{q^2}{1+q^2+q^4}F_2F_1F_3F_2F_3v_{\lambda_2}.\]
This further implies $F_3F_2F_1F_3F_2F_3v_{\lambda_2}=0$, and we find
we have the basis
\begin{align*}
\cB(\lambda_2)&=\set{xv_{\lambda_2}\mid x\in \cB(\infty)\text{ and }xv_{\lambda_2}\notin q\cL(\lambda_2)}\\&=\set{\begin{array}{c}
v_{\lambda_2}, \\
F_2v_{\lambda_2}, \quad F_3v_{\lambda_2},\\
F_1F_2v_{\lambda_2}, \quad F_2F_3v_{\lambda_2}, \quad F_3F_2v_{\lambda_2},\\
F_1F_3F_2v_{\lambda_2},\quad F_1F_2F_3v_{\lambda_2}, \quad F_3F_2F_3v_{\lambda_2},\quad F_2^{(2)}F_3v_{\lambda_2},\\
F_2F_1F_3F_2v_{\lambda_2}, \quad F_1F_2^{(2)}F_3v_{\lambda_2}, \quad
\quad F_3F_2^{(2)}F_3v_{\lambda_2},\quad F_1F_3F_2F_3v_{\lambda_2},\\
F_1^{(2)}F_2^{(2)}F_3v_{\lambda_2},\quad F_1F_3F_2^{(2)}F_3v_{\lambda_2},\quad F_2F_1F_3F_2F_3v_{\lambda_2},\\
F_1^{(2)}F_3F_2^{(2)}F_3v_{\lambda_2}, F_2^{(2)}F_1F_3F_2F_3v_{\lambda_2},\\
F_2F_1^{(2)}F_3F_2^{(2)}F_3v_{\lambda_2}
\end{array}
}.
\end{align*}
\end{example}

\subsection{The other rank 2 case}\label{subsec:CB(1|1|1)}
In Theorem \ref{thm:CBs equal}, we assume that we are working with the standard Borel
associated to the GCM $A^m$, but what about a canonical basis in other cases? The main
difference is that Corollary \ref{cor:deriv and root} doesn't apply, 
so to relate the lattices spanned by different PBW bases, 
there is an additional rank 2 case of comparing the lattices spanned by the
two PBW bases is associated to the GCM $A=\begin{bmatrix}0&1\\1&0\end{bmatrix}$;
note that $\Um(A)\cong \Qq\ang{F_1,F_2}/(F_1^2, F_2^2)$.

In this case, we have the two reduced expressions $\bi=(1,2,1)$ and $\bj=(2,1,2)$, 
which yield the root vectors $F_{\bi;\alpha_1+\alpha_2}=T_2(F_1)=F_2F_1+q^{-1}F_1F_2$
and $F_{\bj;\alpha_1+\alpha_2}=T_1(F_2)=F_1F_2+q^{-1}F_2F_1$,
thus the PBW bases are given by 
\[\BB_\bi=\set{F_1^a(F_2F_1+q^{-1}F_1F_2)^{(b)}F_2^c\mid a,c\in\set{0,1}, b\geq 0}\]
\[\BB_\bj=\set{F_2^c(F_1F_2+q^{-1}F_2F_1)^{(b)}F_1^a\mid a,c\in\set{0,1}, b\geq 0}.\]
In particular, observe that it is more natural to consider the $\Z[q^{-1}]$-span
of these bases then the $\Z[q]$-span; indeed, $F_2F_1\notin \Z[q]\BB_\bi$, 
so Lemma \ref{lem:same q span} doesn't hold in this case. Of course,
note that using the $\Z[q^{-1}]$ span for this case
also invalidates Lemma \ref{lem:same q span} for the full rank case; this is the
simplest instance of the obstruction caused by the chirality in $q$ mentioned in
\S 1.

One might hope to still construct a canonical basis by reinterpreting 
the lemma in this rank 2 case with $q$ replaced everywhere by $q^{-1}$. 
This is not unreasonable, since it is easy to verify that the $\Z[q,q^{-1}]$ 
span of the PBW bases agree and satisfy triangularity under the bar involution
(cf. \cite[Lemmas 7.6, 7.7]{CHW3} or \cite[Theorem 5.1]{Tin})
so standard bar-invariant bases exist. However, even the $\Z[q^{-1}]$-span
of the PBW bases are different: for instance, note that \[(\bB_\bi)_{2\alpha_1+2\alpha_2}=\set{F_1F_2F_1F_2, \frac{F_2F_1F_2F_1}{[2]}+q^{-2}\frac{F_1F_2F_1F_2}{[2]}},\]
\[(\bB_\bj)_{2\alpha_1+2\alpha_2}=\set{F_2F_1F_2F_1, \frac{F_1F_2F_1F_2}{[2]}+q^{-2}\frac{F_2F_1F_2F_1}{[2]}},\]
hence $F_1F_2F_1F_2\in q\Z[q^{-1}]\bB_\bj$.

Nevertheless, there is a natural choice of basis for $\Um$: 
\[\cB=\set{F_1^a(F_2F_1)^bF_2^c\mid a,c\in\set{0,1},b\in\Z_{\geq 0}}
=\set{F_2^a(F_1F_2)^bF_1^c\mid a,c\in\set{0,1},b\in\Z_{\geq 0}}.\] 
This is a bar-invariant basis of $\Um$ which is also trivially 
a basis of the Lusztig integral form, since for any $i\in I$, 
we have $F_i^{(n)}=\delta_{1,n} F_i$. Moreover, $\cB$ has a natural
crystal structure (where $\tf_i x= F_i x$ and $\te_i x=\bar e_i(x)$),
and is congruent modulo $q$ to the {\em unnormalized} PBW bases\footnote{Here, for consistency with earlier discussions of canonical/crystal bases, we switch to the bar-conjugate
version of the PBW bases.}
$\set{F_i^a(F_jF_i+qF_iF_j)^{b}F_j^c\mid a,c\in\set{0,1}, b\geq 0}$ for $\set{i,j}=I$.

Moreover, $\cB$ even satisfies some compatibilities with 
finite-dimensional weight modules. Note that in this case, the simple roots are
$\alpha_1=\epsp_1-\epsp_3$ and $\alpha_2=\epsp_3-\epsp_2$, so the coroots
are $h_1=\epsp_1^\vee+\epsp_3^\vee$ and $h_2=-\epsp_3^\vee-\epsp_2^\vee$.
For $\lambda=a\epsp_1+b\epsp_2+c\epsp_3\in P$, define $M(\lambda)$ to be
the Verma module of highest weight $\lambda$ as usual, and let $1_\lambda$
denote a highest weight vector.

\begin{lem}
Let $\lambda=a\epsp_1+b\epsp_2+c\epsp_3\in P$ and define $\lambda(i)=\ang{h_i,\lambda}$. 
The singular vectors (up to constant multiple) in $M(\lambda)$,
other than $1_\lambda$, are as follows.
\begin{enumerate}
\item $F_11_\lambda$ if and only if $\lambda(1)=0$;
\item $F_21_\lambda$ if and only if $\lambda(2)=0$;
\item $[\lambda(1)](F_1F_2)^{\lambda(1)+\lambda(2)}1_\lambda-[\lambda(2)](F_2F_1)^{\lambda(1)+\lambda(2)}1_\lambda$
if and only if $\lambda(1)+\lambda(2)> 0$.
\end{enumerate}
In particular, the unique simple quotient $V(\lambda)$ of $M(\lambda)$
is finite dimensional if and only if $\lambda(1)=\lambda(2)=0$ or
$\lambda(1)+\lambda(2)> 0$
\end{lem}
\begin{proof}
Observe that
\[E_1F_1^{x}(F_2F_1)^y F_2^z1_\lambda=\delta_{x,1}[\lambda(1)-y-z](F_2F_1)^yF_2^z1_\lambda
+\delta_{z,0}(-1)^{x+1}[\lambda(1)]F_1^x(F_2F_1)^{y-1}F_21_\lambda\]
\[E_2F_1^{x}(F_2F_1)^yF_2^z1_\lambda=\delta_{x,0}[\lambda(2)-y]F_1(F_2F_1)^{y-1}F_2^z1_\lambda
+\delta_{z,1}(-1)^x[\lambda(2)]F_1^x(F_2F_1)^{y}1_\lambda.\]
Observe that $(F_2F_1)^yF_21_\lambda$ is singular if and only if
$\lambda(2)=0$ and $y=0$, and likewise $F_1(F_2F_1)^y1_\lambda$ is singular if and only if
$\lambda(1)=0$ and $y=0$.
Furthermore,
$$c(F_1F_2)^y1_\lambda+d(F_2F_1)^{y}1_\lambda$$ is singular
if and only if there is some $y\geq 0$ such that
\[c[\lambda(1)-y]=d[\lambda(1)]\text{ and }c[\lambda(2)]=d[\lambda(2)-y].\]
In particular, these equations are satisfied when $c=d=0$;
$c\neq 0$, $d=0=\lambda(2)$, and $y=\lambda(1)\geq 0$; 
$d\neq 0$, $c=0=\lambda(1)$, and $y=\lambda(2)\geq 0$;
or $c\neq 0$ and $d\neq 0$, and 
$$[\lambda(1)-y][\lambda(2)-y]=[\lambda(1)][\lambda(2)],$$
which holds if and only if $y=\lambda(1)+\lambda(2)\geq 0$.
In the case $c\neq 0$ and $d\neq 0$, we see that
\[c[\lambda(2)]=-d[\lambda(1)].\]
This finishes the proof of the classification of singular vectors.

Recall that $N(\lambda)$ is the maximal submodule of $M(\lambda)$,
and $V(\lambda)=M(\lambda)/N(\lambda)$ is the simple quotient.
Let us denote by $v_\lambda$ the image of $1_\lambda$ in $V(\lambda)$.
Note that $N(\lambda)$ contains, and is generated by, 
all singular vectors other than (nonzero multiples of) $1_\lambda$.
We see that when $\lambda(1)=\lambda(2)=0$, 
$F_11_\lambda, F_21_\lambda\in N(\lambda)$
hence $V(\lambda)=M(\lambda)/N(\lambda)=\Qq v_\lambda$ is one-dimensional.
Further, note that if $\lambda(1)+\lambda(2)<0$ or $\lambda(1)=-\lambda(2)$,  then
$1_\lambda$ is the only singular vector, hence $N(\lambda)=0$ and $V(\lambda)$
is infinite-dimensional.
Finally, if $\lambda(1)+\lambda(2)>0$, then we have
\begin{equation}
V(\lambda)=\bigoplus_{\substack{0\leq x\leq 1;\\
0\leq y< \lambda(1)+\lambda(2)\\0\leq z\leq \delta_{\lambda(j),0};}} \Qq F_i^x(F_jF_i)^yF_j^zv_\lambda
\text{ where } I=\set{i,j} \text{ and }\lambda(i)\neq 0.
\end{equation}
\end{proof}

In particular, we observe the following corollary.

\begin{cor} Let
$L(\lambda)=L(\infty)v_\lambda$, where $L(\infty)$ is the crystal lattice of $\Um$.
Then $V(\lambda)$ has the canonical basis
\[\cB(\lambda)=\set{bv_\lambda\mid b\in \cB\text{ such that }bv_\lambda\notin qL(\lambda)}.\]
\end{cor}

We note that in the case $\lambda(1)=\lambda(2)>0$,
$(F_1F_2)^{\lambda(1)+\lambda(2)}v_\lambda=(F_2F_1)^{\lambda(1)+\lambda(2)}v_\lambda$;
in particular, in contrast to the results of Section \ref{subsec:CB(2|1)},
the projection $\pi_\lambda:\Um\rightarrow V(\lambda)$ does not restrict
to a bijection between $\set{b\in\cB\mid bv_\lambda\notin qL(\lambda)}$ and $\cB(\lambda)$.

\begin{rmk}
As a final remark, let us observe a trend in the results of Section \ref{sec:examples}. 
It is unfortunately not the case in general that the finite-dimensional irreducible modules
inherit a basis from the canonical basis of a half quantum $\gl(m|1)$, since the modules
with atypical highest weight generally will have linear dependencies.  Nevertheless,
we can often canonically remove the redundant elements 
by taking only the nonzero elements in the quotient
$L/qL$, where $L$ is the $\cA$-lattice generated by the images of canonical basis elements.
It would be interesting to realize this as some sort of crystal basis construction, though,
as we have noted in the examples, one would need a different definition of Kashiwara operators.
\end{rmk}


\begin{thebibliography}{CHW}

\bibitem[BKK]{BKK} G.~Benkart, S.-J.~Kang and M.~Kashiwara,
{\em Crystal bases for the quantum superalgebra }$U_q(\mathfrak{gl}(m,n))$,
Journal of Amer. Math. Soc.  {\bf 13} (2000), 295--331.


\bibitem[CW]{CW} S.-J.~Cheng and W.~Wang, {\em Dualities and Representations of Lie Superalgebras},
Graduate Studies in Mathematics \textbf{144}, American Mathematical Society 2012.

\bibitem[C]{C} S.~Clark, {\em Quantum $\osp(1|2n)$ knot invariants are the same as quantum $\mathfrak{so}(2n+1)$ knot invariants}, arXiv:1509.03533

\bibitem[CHW1]{CHW1} S.~Clark, D.~Hill and W.~Wang, {\em Quantum supergroups I. Foundations}, Transform. Groups {\bf 18} (2013), 1019-1053.

\bibitem[CHW2]{CHW2} S.~Clark, D.~Hill and W.~Wang, {\em Quantum supergroups II. Canonical basis},  Represent. Theory \textbf{18} (2014), 278-309.

\bibitem[CHW3]{CHW3} S.~Clark, D.~Hill and W.~Wang,
{\em Quantum shuffles and quantum supergroups of basic type},  Quant. Top. (to appear),
arXiv:1310.7523.


\bibitem[DG]{DG} J.~Du and H.~Gu, {\em Canonical bases for the quantum supergroups $\UU(\gl_{m|n})$},
Math. Zeit. \textbf{281}, 631-660

\bibitem[EK]{EK} H.~El~Turkey and J.~Kujawa, {\em Presenting Schur superalgebras},
Pacific J. Math. \textbf{262} (2013), 285-316.

\bibitem[EL]{EL} A.~Ellis and A.~Lauda,
{\em An odd categorification of $U_q(\mathfrak{sl}_2)$}, arXiv:1307.7816.

\bibitem[Hec]{He} I.~Heckenberger, {\em Lusztig isomorphisms for Drinfel'd doubles of bosonizations 
of Nichols algebras of diagonal type}, J. Algebra \textbf{323} (2010), 2130-2182

\bibitem[HY]{HY} I.~Heckenberger and H.~Yamane, {\em A generalization of Coxeter groups, 
root systems, and Matsumoto’s theorem}, Math. Zeit. \textbf{259} (2008),  255-276

\bibitem[HW]{HW}
D.~Hill and W.~Wang, \emph{{Categorification of quantum Kac-Moody
  superalgebras}}, Trans. Amer. Math.
Soc. \textbf{367} (2015), 1183-1216.

\bibitem[Kac]{Kac} V.~Kac, \emph{Lie superalgebras}, Adv. Math. \textbf{26} (1977), 8-96.

\bibitem[Kas]{Ka} M.~Kashiwara,
{\em On crystal bases of the $q$-analogue of universal enveloping algebras}, Duke Math. J. \textbf{63} (1991), 465--516.

\bibitem[KL]{KL} D.~Kazhdan and G. Lusztig,
{\em Representations of Coxeter Groups and Hecke Algebras}, Inven. Math. \textbf{53} (1979), 165--184.

\bibitem[Kho]{Kh} M.~Khovanov,
{\em  How to categorify one-half of quantum $\mathfrak{gl}(1|2)$}, arXiv:1007.3517.


\bibitem[KS]{KS} M.~Khovanov and J.~Sussan,
{\em A categorification of the positive half of quantum $\mathfrak{gl}(m|1)$},
Trans. AMS (to appear), arXiv:1406.1676.

\bibitem[Kwo1]{Kw09} J.-H.~Kwon,
{\em Crystal graphs for general linear Lie 
superalgebras and quasi-symmetric functions},
J. Comb. Theory A \textbf{116} (2009), 1199--1218.

\bibitem[Kwo2]{Kw14} J.-H.~Kwon,
{\em Crystal bases of $q$-deformed Kac modules over the 
quantum superalgebra $U_q(\gl(m|n))$},
Int. Math. Res. Notices \textbf{2} (2014), 512--550.

\bibitem[Kwo3]{Kw15} J.-H.~Kwon,
{\em Super duality and crystal bases for quantum ortho-symplectic superalgebras II},
Int. Math. Res. Notices \textbf{23} (2015), 12620--12677.


\bibitem[Lec]{Lec} B. Leclerc, 
{\em Dual canonical bases, quantum shuffles and $q$-characters}, Math. Z. \textbf{246} (2004),  691--732.

\bibitem[Lus1]{LuPaper} G. Lusztig,
        {\em Canonical bases arising from quantized enveloping algebras},
        Journ. AMS \textbf{3} (1990), 447--498.

\bibitem[Lus2]{LuBook} G. Lusztig,
        {\em Introduction to Quantum Groups},
        Progress in Math. \textbf{110}, Birkh\"{a}user 1993.
        
\bibitem[MZ]{MZ} I.~M.~Musson and Y.-M.~Zou,
{\em Crystal basis for $U_q(\mathfrak{osp}(1, 2r))$}, J. Alg. \textbf{210} (1998), 514--534.
        

\bibitem[Sar]{Sar} A.~Satori, {\em Categorification of tensor powers of the vector representation of $U_q(\gl(1|1))$},
Selecta Math. \textbf{22 } (2016), 669--734.

\bibitem[Sai]{Sai} Y.~Saito, {\em PBW basis of quantized universal
enveloping algebras}, Publ. RIMS. \textbf{30} (1994), 209--232.

\bibitem[Ser]{Se} V.~Serganova, {\em Kac-Moody superalgebras and integrability},
Prog. Math. \textbf{288}, 169--218.

\bibitem[Tia]{Tia2} Y.~Tian,
{\em A categorification of $U_T(\fsl(1|1))$ and its tensor product representations}, Geom. Topol. \textbf{18} (2014), 1635--1717.

\bibitem[Tin]{Tin} P.~Tingley,
{\em Elementary construction
of Lusztig's canonical basis}, arXiv:1602.04895.

\bibitem[VdJ]{VdJ} J.~Van~der~Jeugt, 
{\em Dimension formulas for the Lie superalgebra $\fsl(m/n)$}, 
Jour. Math. Phys. \textbf{36} (1995), 605--611.

\bibitem[VV]{VV}
M.~Varagnolo and E.~Vasserot, {\em Canonical bases and KLR-algebras}, J.
Reine Angew. Math. \textbf{659} (2011), 67--100.

\bibitem[Xi]{Xi} N.~Xi,
{\em Canonical basis for type $A_3$},
Comm. Alg. \textbf{27} (1999),  5703--5710.

\bibitem[Yam]{Ya} H.~Yamane,
{\em Quantized enveloping algebras associated with simple Lie superalgebras and their universal $R$-matrices},
Publ. Res. Inst. Math. Sci. \textbf{30} (1994), 15--87.

\bibitem[Zou]{Zou} Y.~M.~Zou,
{\em On the structure of $U_q(\mathfrak{sl}(m,1))$: crystal bases},
J. Phys. A: Math. Gen. \textbf{32} (1999), 8197--8207.

\end{thebibliography}
\end{document}